\documentclass[a4paper,11pt]{article}
\usepackage[latin1]{inputenc}
\usepackage[T1]{fontenc}
\usepackage[british,UKenglish,USenglish,english,american]{babel}
\usepackage{amsmath}
\usepackage{amsfonts}
\usepackage{hyperref}
\usepackage[T1]{fontenc}
\usepackage[latin1]{inputenc}
\usepackage{theorem}
\usepackage{graphics}
\usepackage{makeidx}
\usepackage{fancyhdr}
\usepackage{bbm}
\usepackage{amssymb}
\usepackage{color}
\numberwithin{equation}{section}
\usepackage[titletoc,title]{appendix}
\usepackage{enumerate}
\newtheorem{theorem}{Theorem}[section]

\newtheorem{corollary}[theorem]{Corollary}
\newtheorem{lemma}[theorem]{Lemma}

\newtheorem{definition}[theorem]{Definition}
\theorembodyfont{\rm}

\newtheorem{remark}[theorem]{Remark}
\newtheorem{example}[theorem]{Example}

\newcommand{\IN}{{\mathbb{N}}}

\renewcommand{\phi}{\varphi}
\renewcommand{\epsilon}{\varepsilon}
\newcommand{\eps}{\varepsilon}

\newcommand{\ltn}{\ensuremath{\left| \! \left| \! \left|}}
\newcommand{\rtn}{\ensuremath{\right| \! \right| \! \right|}}

\newcommand{\aM}{\mathcal{M}}
\newcommand{\aP}{\mathcal{P}}
\newcommand{\aI}{\mathcal{I}}

\newcommand{\aL}{\mathcal{L}}
\newcommand{\oN}{\overline{N}}
\newcommand{\oM}{\overline{M}}
\newcommand{\ov}{\overline{v}}
\newcommand{\omegaa}{\,\omega^{(2)}}
\newcommand{\omegaan}{\,\omega^{(2),n}}
\newcommand{\yt}{\widetilde{y}}
\newcommand{\zt}{\widetilde{z}}
\newcommand{\yb}{\overline{y}}
\newcommand{\zb}{\overline{z}}
\newcommand{\Tb}{\overline{T}}
\newcommand{\wtu}{\widetilde{u}}
\newcommand{\wtv}{\widetilde{v}}

\newcommand{\qed}{\hfill $\Box$\smallskip}

\title{Local mild solutions for rough stochastic partial differential equations}

\author{Robert Hesse\thanks{Institute for Mathematics, Friedrich Schiller University Jena (FSU),  Ernst-Abbe-Platz 2, 07743 Jena, Germany. E-Mail: robert.hesse@uni-jena.de.}~~and~Alexandra 
	Neam\c tu\thanks{Technical University of Munich (TUM), 
		Faculty of Mathematics, 85748 Garching bei M\"unchen, Germany. E-Mail: alexandra.neamtu@tum.de.\newline\hspace*{3 mm} We are grateful to M. J. Garrido-Atienza and B. Schmalfu\ss{} for helpful comments. We thank the referee for carefully
		reading the manuscript and for the valuable suggestions.
		\newline\hspace*{3 mm}
		AN acknowledges support by a DFG grant in the
		D-A-CH framework (KU 3333/2-1). 	}}

\begin{document}
	\maketitle
	\begin{abstract}
		We investigate mild solutions for stochastic evolution equations driven by a fractional Brownian motion (fBm) with Hurst parameter $H\in(1/3,1/2]$ in infinite-dimensional Banach spaces. Using elements from rough paths theory we introduce an appropriate integral with respect to the fBm. This allows us to solve pathwise our stochastic evolution equation in a suitable function space.
	\end{abstract}



	\section{Introduction}

In this work we develop a concise theory of rough evolution equations given by
\begin{align}\label{eq1}
\begin{cases}
d y_{t} = (A  y_{t} + F (y_{t}) )dt + G (y_{t}) d \omega_{t},~~ t\in[0,T] \\
y(0)=\xi.
\end{cases}
\end{align}
Here $T>0$, $A$ generates an analytic $C_{0}$-semigroup $(S(t))_{t\in[0,T]}$ on a separable Banach space $W$, the initial condition $\xi\in W$, $F:W\to W$ is a nonlinear term satisfying appropriate Lipschitz conditions and $G:W\to \mathcal{L}(V,W)$ is assumed to be three times continuously differentiable with bounded derivatives. The precise assumptions will be stated in Section \ref{preliminaries}. Finally, the random input $\omega$ is Gaussian process which has the regularity of a fractional Brownian motion with Hurst index $H\in(1/3,1/2]$. In order to solve (\ref{eq1}) 
we need to give a meaning of the rough integral
\begin{align}\label{integral}
\int\limits_{0}^{t} S(t-r)G(y_{r}) d\omega_{r}.
\end{align}

Results in this context are available in \cite{GarridoLuSchmalfuss2} via fractional calculus and in Gubinelli et al \cite{Gubinelli}, \cite{DeyaGubinelliTindel} \cite{GubinelliLejayTindel}, \cite{GubinelliTindel} using rough paths techniques. In this work, we combine Gubinelli's approach with the arguments employed by \cite{GarridoLuSchmalfuss2} to solve (\ref{eq1}). This theory should hopefully be more simple in order to investigate the long-time behavior of such equations using a random dynamical systems approach such in \cite{Arnold}, \cite{GarridoLuSchmalfuss1} or \cite{BailleulRiedelScheutzow}. In this work we establish only the existence of a local mild solution. We investigate in a forthcoming paper global solutions and random dynamical systems for (\ref{eq1}) as in \cite{GarridoLuSchmalfuss1}. This work should be seen as a first step in order to close the gap between rough paths and random dynamical systems in infinite-dimensional spaces.\\
	
	
Since the fractional Brownian motion is not a semi-martingale, the construction of an appropriate integral represents a challenging problem. This has been intensively investigated and numerous results and various techniques are available, see \cite{Zaehle}, \cite{GarridoLuSchmalfuss1}, \cite{Young}, \cite{HuNualart}, \cite{MaslowskiNualart} 
and the references specified therein. There is a huge literature where certain tools from fractional calculus (i.e. fractional/compensated fractional derivative/integral) are employed to give a pathwise meaning of the stochastic integral with respect to the fractional Brownian motion with Hurst parameter $H\in(1/2,1)$ or $H\in(1/3,1/2]$. A different method which has been recently introduced and explored is given by the rough path approach of Gubinelli et. al. \cite{Gubinelli}, \cite{GubinelliTindel}, \cite{GubinelliLejayTindel}. This goes through if $H\leq 1/2$. Moreover it is suitable to define (\ref{integral}) not only with respect to the fractional Brownian motion but also to Gaussian processes for which the covariance function satisfies certain structure, see \cite{FritzVictoir} or \cite[Chapter 10]{FritzHairer}. An overview on the connection between rough paths and fractional calculus can be looked up in \cite{HuNualart}. Of course, in some situations, various other techniques for $H\leq 1/3$ are available.\\

After using an appropriate integration theory with respect to the fractional Brownian motion, the next step is to analyze SDEs/SPDEs driven by this kind of noise. There is a growing interest in establishing suitable properties of the solution under several assumptions on the coefficients, consult  \cite{MaslowskiNualart}, \cite{MaslowskiS}, \cite{Hairer}, \cite{GLS}, \cite{GarridoLuSchmalfuss1}, \cite{Zaehle} and the references specified therein.\\

To our aims we combine techniques from the rough path theory used in \cite{GubinelliLejayTindel}, \cite{GubinelliTindel}, \cite{DeyaGubinelliTindel} with the approach used in \cite{GarridoLuSchmalfuss2}. Therefore, we can treat stochastic evolution equations under more general assumptions on the coefficients. Note that in our case $F$ and $G$ are nonlinear operators defined on suitable function spaces. We do not impose additional regularity assumptions on the covariance operator of the fractional Brownian motion as in \cite{MaslowskiNualart}. Another advantage of this theory is that we do not have to require that $G$ is a Hilbert-Schmidt operator like in \cite{GarridoLuSchmalfuss2}. This is demanded there to define (\ref{integral}) as a series of one-dimensional integrals. Comparing with \cite{GubinelliTindel} the main difference of our approach is that we work in a different function space which enables us to consider arbitrary initial conditions for (\ref{eq1}). This is crucial when working with random dynamical systems, compare \cite{GarridoLuSchmalfuss2} and will be explored in a future work. Due to this reason we have to provide a meaningful construction of (\ref{integral}) in this new function space.\\
	
	This work is structured as follows. In Section \ref{preliminaries} we collect well-known properties and estimates of analytic semigroups which are necessary in this framework. Section \ref{heuristics} provides the very general intuition of the required techniques.
	It describes the story in a nutshell and gives an insight in the rough path theory pointing out the main obstacles which occur in the infinite-dimensional setting. We state basic concepts and indicate how an appropriate pathwise integral should be constructed and how a solution of a rough evolution equation should look like. The next sections rigorously justifies the steps presented in Section \ref{heuristics} as follows.  Section \ref{sectsl} is the core of this work. Here we introduce a modified version of the Sewing Lemma, compare \cite{GubinelliTindel},  \cite{GubinelliLejayTindel}. This is a very general fundamental result which entails the existence of a rough integral in a suitable analytic and algebraic framework. This is, of course, the first main ingredient that we need in order to solve (\ref{eq1}). As already emphasized, in contrast to \cite{GubinelliTindel}, we work with modified H\"older spaces as introduced in Section \ref{preliminaries}. This version of the Sewing Lemma precisely fits in this setting.
	Section \ref{sec:supp:proc} is devoted to the construction of the supporting processes which are necessary to give an appropriate meaning of the rough integral. Inspired by  \cite{GarridoLuSchmalfuss2}, in order to define the supporting processes we first consider smooth approximations of the noise and thereafter pass to the limit.
	The existence of the corresponding processes is derived via classical tools, such as an integration by parts formula or using the Sewing Lemma introduced in Section \ref{sectsl}.
	For a better comprehension, we point out an example in which one can construct a pathwise integral using the integration by parts formula as well as the Sewing Lemma. Section \ref{sectfp} deals with the fixed-point argument and contains the main result of this work. Since certain a-priori estimates presented in Section \ref{sectfp} contain quadratic terms one cannot immediately conclude the existence of a global solution. This problem will be solved in a forthcoming paper under suitable assumptions on the coefficients using regularizing properties of analytic semigroups. We present an application of our theory in Section \ref{sect:app}. Finally, we collect some important results and computations in two appendices.
		\section{Preliminaries}\label{preliminaries}
			We fix $T>0$, set $\Delta:=\Delta_{T} := \left\{\left(t,s \right) \colon T\geq t > s \geq 0\right\}$ and let  $V$ stand for a Hilbert space and let $W$ denote a separable Banach space. For notational simplicity, if not further stated, we write  $\left| \cdot \right|$ for the norm of an arbitrary Banach space.  Furthermore $C$ denotes a universal constant which varies from line to line. The explicit dependence of $C$ on certain parameters will be precisely stated, whenever required.\\ 

			We firstly describe the noisy input. To this aim we recall the following essential concept in the rough path theory.
			\begin{definition}\label{hrp}\emph{($\alpha$-H\"older rough path)}
			We call a pair $(B,\mathbb{B})$ $\alpha$-H\"older rough path if $B$ is $V$-valued and $\alpha$-H\"older continuous, $\mathbb{B}$ is $V\otimes V$-valued and $2\alpha$-H\"older continuous.  Furthermore, $B$ and $\mathbb{B}$ are connected via Chen's relation, meaning that
			\begin{align*}
			\mathbb{B}_{ts} - \mathbb{B}_{us} - \mathbb{B}_{tu} = (B_{u} - B_{s})\otimes ( B _{t}- B_{u}) ,~~ \mbox{for } 0\leq s \leq u \leq t\leq T.
			\end{align*}
			In the literature $\mathbb{B}$ is referred to as L\'evy-area or second order process.	
			\end{definition}	
		
		We further describe an appropriate \emph{distance} between two $\alpha$-H\"older rough paths.
		\begin{definition}
			Let $(B,\mathbb{B})$ and $(\widetilde{B},\widetilde{\mathbb{B}})$ be two $\alpha$-H\"older rough paths. We introduce the $\alpha$-H\"older rough path (inhomogeneous) metric
			\begin{align}\label{rp:metric}
			d_{\alpha,[0,T]}((B,\mathbb{B}),(\widetilde{B},\widetilde{\mathbb{B}}) )
			:= \sup\limits_{(t,s)\in \Delta} \frac{|B_{t}-B_{s}-\widetilde{B}_{t}+\widetilde{B}_{s}|}{|t-s|^{\alpha}} 
			+ \sup\limits_{(t,s) \in \Delta}
			\frac{|\mathbb{B}_{ts}-\widetilde{\mathbb{B}}_{ts}|} {|t-s|^{2\alpha}}.
			\end{align}
			We set $d_{\alpha,T}:=d_{\alpha,[0,T]}$.
		\end{definition}
		
			 For more details on this topic consult~\cite[Chapter 2]{FritzHairer}.\\
			
			  We now emphasize the connection between the fractional Brownian motion and an $\alpha$-H\"older rough path. To this aim, let $\omega(\cdot)$ stand for a $V$-valued trace-class fractional Brownian motion with Hurst index $H\in(1/3,1/2]$. This means that
		\begin{align}
		\omega(t) = \sum\limits_{n=1}^{\infty}\sqrt{\lambda_{n}} b^{H}_{n}(t) e_{n}, 
		\end{align}
		where $(b^{H}_{n}(\cdot))_{n\in\mathbb{N}}$ is a sequence of one-dimensional independent standard fractional Brownian motions with the same Hurst parameter $H$ and $\sum\limits_{n=1}^{\infty}\lambda_{n}<\infty$.

\begin{remark}
	We stress that the trace-class condition is the only requirement we impose on the covariance of the fractional Brownian motion in contrast to \cite{MaslowskiNualart}.
\end{remark}

We point out the following result which will intensively be used in Section \ref{sec:supp:proc}.
		\begin{lemma}
Let $\omega$ be a $V$-valued trace class fractional Brownian, as introduced above. Then there exists a L\'evy-area $\omegaa$ such that the $\alpha$-H\"older rough path $(\omega,\omegaa)$ can be approximated by $(\omega^{n},\omegaan)$ with respect to the $d_{\alpha,T}$-metric. Here $(\omega^{n})_{n\in\mathbb{N}}$ is a sequence of piecewise linear functions and 
	\begin{align*}
	\omegaan_{ts}:=\int\limits_{s}^{t} ( \omega^{n}_{r} -\omega^{n}_{s}) \otimes d\omega^{n}_r.
	\end{align*}
	\end{lemma}
		 	\begin{proof}
		 		The proof follows by the same arguments as in Lemma 2 in \cite{GarridoLuSchmalfuss3}.
		 		\qed
		 		\end{proof}\\
		 
	 Having stated the random influences that we consider, we now introduce the assumptions on the linear part and on the coefficients $F$ and $G$. \\

				Since we are in the parabolic setting, i.e. $A$ is a sectorial operator, we can introduce its fractional powers, $(-A)^{\gamma}$ for $\gamma\geq 0$, see \cite[Section 2.6]{Pazy} or \cite{Lunardi}. We denote the domains of the fractional powers of $(-A)$ with $D_{\gamma}$, i.e. $D_{\gamma}:=D((-A)^{\gamma})$ and use the following estimates.\\
		
		For $\eta, \kappa\in\mathbb{R}$ we have
		\begin{align}
			||S(t) ||_{\mathcal{L}(D_{\kappa}, D_{\eta})} =||(-A)^{\eta}S(t)||_{\mathcal{L}(D_{\kappa},W)} \leq C t ^{k-\eta}, &~~\mbox{for  } \eta\geq \kappa \label{hg1}\\
			 ||S(t) - \mbox{Id} ||_{\mathcal{L}(D_{\sigma}, D_{\theta})} \leq C t^{\sigma-\theta}, &~~\mbox{for  } \sigma-\theta\in[0,1]\label{hg2}.
		\end{align}
		Furthermore, one can show that the following assertions hold true, consult \cite[Chapter 3]{Pazy}. 
		\begin{lemma}
			For any $\nu,\eta,\mu\in[0,1]$, $\kappa,\gamma,\rho\geq 0$ such that $\kappa\leq \gamma+\mu$, there exists a constant $C>0$ such that for $0<q<r<s<t$ we have that
\begin{align*}
&||S(t-r) - S(t-q)||_{\mathcal{L}(D_{\kappa},D_{\gamma})} \leq C (r-q)^{\mu}(t-r)^{-\mu-\gamma+\kappa},\\
&||S(t-r)- S(s-r)- S(t-q) + S(s-q)||_{\mathcal{L}(D_{\rho}, D_{\rho})} \leq C (t-s)^{\eta}(r-q)^{\nu}(s-r)^{-(\nu+\eta)}.
\end{align*}	
	\end{lemma}
	
		For our aims we introduce the following function spaces. Let $\alpha, \beta\in(0,1]$. 
	Let $\overline{W}$ stand for a further Hilbert space. We recall that $C^{\beta}([0,T],W)$ represents the space of $W$-valued H\"older continuous functions on $[0,T]$ and denote by $C^{\alpha}(\Delta_T,\overline{W})$ the space of $\overline{W}$-valued functions on $\Delta_T$ with $z_{tt}=0$ for all $t\in[0,T]$ and
		\begin{align*}
		 \left\| z \right\|_{\alpha}:= \sup\limits_{0\leq t \leq T} |z_{t0}| + \sup\limits_{0\leq s < t \leq T} \frac{\left|z_{ts}\right|}{(t-s)^{\alpha}}<\infty.
		\end{align*}

		Furthermore, we define $C^{\beta,\beta}([0,T],W)$ as the space of $W$-valued continuous functions on $[0,T]$ endowed with the norm
	
		\begin{align*}
		\left\|y\right\|_{\beta,\beta}
		:= \left\|y\right\|_ \infty + \ltn y\rtn_{\beta,\beta}
		:= \sup\limits_{0\leq t \leq T } |y_t| + \sup\limits_{0<s<t\leq T }s^{\beta}\frac{|y_t-y_s|}{(t-s)^{\beta}}.
		\end{align*}
Similarly we introduce $C^{\alpha+\beta,\beta}(\Delta_T, \overline{W})$ with the norm
\begin{align*}
 \left\| z \right\|_{\alpha+\beta,\beta}:=\sup\limits_{0\leq t \leq T } |z_{t0}| + \sup\limits_{0< s < t \leq T} s^{\beta}\frac{\left|z_{ts}\right|}{(t-s)^{\alpha+\beta}}.
\end{align*}
Again $z_{tt}=0$ for all $t\in [0,T]$.\\

		 These modified H\"older spaces are well-known in the theory of maximal regularity for parabolic evolution equations, consult \cite{Lunardi}. These were also used in \cite{GarridoLuSchmalfuss2}.\\
		
		In this framework we emphasize the following result which will be employed throughout this work. It is well-known that analytic $C_{0}$-semigroups are not H\"older continuous in $0$. However, the following lemma holds true.
		
		\begin{lemma}\label{betabeta}
		 Let $\left(S(t)\right)_{t\geq 0}$ be an analytic $C_{0}$-semigroup on $W$. Then we have for all $x \in W$ and all $\beta \in \left[0,1 \right]$ that
			\begin{align*}
			\left\|  S(\cdot) x\right\|_{\beta,\beta}\leq C \left|x\right|,
			\end{align*}
			where $C$ exclusively depends on the semigroup and on $\beta$.
		\end{lemma}
		\begin{proof}
			\begin{align*}
			\left\| S(\cdot) x\right\|_{\beta,\beta} 
			&= \sup\limits_{0 \leq t \leq T} \left|S(t) x\right| + \sup\limits_{0<s < t \leq T} s^\beta \frac{\left|(S(t)-S(s)) x\right| }{(t-s)^\beta}\\
			& \leq \sup\limits_{0 \leq t \leq T} \left|S(t) x\right| + \sup\limits_{0<s<t\leq T} s^{\beta} \frac{|(S(t-s)-\mbox{Id})S(s)x|}{(t-s)^{\beta}} \\
			&\leq C |x|,
			\end{align*}
			recall (\ref{hg1})
 and (\ref{hg2}).
 		\qed \end{proof}\\
	
This justifies our choice of working with the function space $C^{\beta,\beta}$. Note that if one lets $x\in D_{\beta}$ it suffices to consider only $C^{\beta}$. However, since we intend to analyze random dynamical systems generated by (\ref{eq1}) in $W$, compare \cite{GarridoLuSchmalfuss2}, we need to take the initial condition $x\in W$ instead of $D_{\beta}$.\\
 
On the coefficients we impose: 
\begin{itemize}
	\item[($F$)] $F \colon W \to W$ is  Lipschitz continuous.
	\item[($G$)] $G \colon W \to \aL(V,W)$ is three times Frech\'{e}t differentiable with bounded derivatives. As justified in the next sections, we additionally have to impose the following restriction on $G$. Namely, $G \colon W \to \aL(V,D_\beta)$ is Lipschitz continuous. Here we demand $\alpha + 2\beta >1$.
	\end{itemize}

A concrete example will be provided in Section \ref{sect:app}.
\begin{remark}
	 Note that the drift (i.e. $F$) does not represent a major obstacle. For simplicity we set $F\equiv 0,$ since there are no additional arguments required to treat this term.
\end{remark}
In order to develop a suitable theory which enables us to define (\ref{integral}) we rely on the algebraic framework of Gubinelli et. al., consult \cite{GubinelliTindel}, \cite{DeyaGubinelliTindel}. Therefore we fix here some important notations.\\

For $y \in C([0,T],W)$ and $z \in C(\Delta_T,\overline{W})$ we set
\begin{align*}
(\delta y)_{ts}&:= y_t-y_s, \\
(\hat\delta y)_{ts}&:= y_t-S(t-s)y_s, \\
(\delta_2 z)_{t \tau s}&:= z_{ts} - z_{t\tau} - z_{\tau s}, \\
(\hat\delta_2 z)_{t \tau s}&:= z_{ts} - z_{t\tau} - S(t-\tau)z_{\tau s}.
\end{align*}
Let us state some important algebraic properties. For more details and a more general framework, see \cite{GubinelliTindel}.
\begin{lemma}\label{lemma_preliminaries_hatdelta} 
			The following statements hold true:
			\begin{enumerate}[(i)]
				\item $\hat\delta_2 \circ \hat\delta \equiv 0$.\label{2.7i}
				\item Given $N \in C\left(\Delta_T,W \right)$ with $\hat\delta_2 N \equiv 0$. Then there exists $y \in C\left([0,T],W \right)$ with $(\hat\delta y)_{ts} = N_{ts}$.
				\item Consider $y^1,y^2 \in C\left([0,T],W \right)$ with $y^1_0=y^2_0$ and $(\hat\delta y^1)_{ts}= (\hat\delta y^2)_{ts}$. Then $y^1 \equiv y^2$. 
			\end{enumerate}
		\end{lemma}
		\begin{proof}
			\noindent
			\begin{enumerate}[(i)]
				\item Take an arbitrary $y \in C\left([0,T],W \right)$. 
				\begin{align*}
				(\hat\delta_2 \hat\delta y)_{t \tau s} 
				= \,& (\hat\delta y)_{ts} - (\hat\delta y)_{t\tau} - S(t-\tau) (\hat\delta y)_{\tau s}\\
				= \,& y_t - S(t-s) y_s - y_t + S(t-\tau) y_\tau - S(t-\tau) y_\tau + S(t-s) y_s =0.
				\end{align*}
				\item Let $\hat\delta_2 N \equiv 0$. Set $y_t:=N_{t0}$. Then, we have
				\begin{align*}
				(\hat\delta y)_{ts} = N_{t0} - S(t-s) N_{s0} =N_{ts}.
				\end{align*}
				\item Consider
				\begin{align*}
				y^1_t = (\hat\delta y^1)_{t0} + S(t) y^1_0 = (\hat\delta y^2)_{t0} + S(t) y^2_0 = y^2_t. 
				\end{align*}
			\end{enumerate}
		\qed  \end{proof}\\
	
	The second assertion of the previous Lemma is extremely important for the deliberations made in Section \ref{sectsl}, especially for Theorem \ref{lemma_sewing}, which ensures the existence of rough integrals.
	\section{Heuristic considerations}\label{heuristics}
	For a better comprehension and in order to point out the difficulties that arise in the infinite-dimensional setting, we shortly sketch the well-known results from the finite dimensional one. There, the solution  theory of (\ref{eq1}) is well-established and one needs very few ingredients to define (\ref{integral}) by rough paths techniques. This immediately entails a suitable solution concept for (\ref{eq1}). Regard that we use the notations introduced in Section \ref{preliminaries}.\\
	
	Since the trajectories of the noise are irregular, i.e. H\"older continuous with exponent $\alpha<1/2$, the Young integral (\cite{Young}) defined
	as 
	\begin{align}\label{young}
	\int\limits_{s}^{t} y_{r} d\omega_{r} = \lim\limits_{|\aP|\to 0} \sum\limits_{[u,v]\in\aP} y_{u} (\delta\omega)_{vu},
	\end{align}
	can no longer be used. Here  $0 \leq s < t$ are two time-points and  $\aP=\aP(s,t)$ is an arbitrary partition and $(\delta \omega)_{ts}:=\omega_{t}-\omega_{s}$.
	Therefore, Gubinelli \cite{Gubinelli} introduced the concept of  controlled rough integral, which extends the Young case. Regarding (\ref{young}), it turns out that we have to consider additional terms satisfying certain algebraic and analytic properties. This reads as follows
	\begin{align}\label{gub}
	\int\limits_{s}^{t} y_{r} d\omega_{r} = \lim\limits_{|\aP|\to 0} \sum\limits_{[u,v]\in\aP} y_{u} (\delta\omega)_{vu} + y'_{u} \omegaa_{vu}.
	\end{align}
	Here the pair $(y,y')$ stands for a controlled rough path. This can be interpreted as an abstract Taylor series, namely one assumes that there exists $y'\in C^{\alpha}$ (which is called Gubinelli's derivative), such that
	\begin{align}\label{crp}
	y_{t}=y_{s}+y'_{s} (\delta\omega)_{ts}+ R^{y}_{ts},
	\end{align}
	where the remainder $R^{y}$
 is $2\alpha$-H\"older regular. Here $\omega\in C^{\alpha}$ and $\omegaa\in C^{2\alpha}$ are connected via Chen's relation (recall Definition \ref{hrp}), meaning that for $0 \leq s \leq u \leq t$
 \begin{align}\label{chen}
 (\delta_2 \omegaa)_{tus}= (\delta\omega)_{us} \otimes (\delta\omega)_{tu}.
 \end{align}
  Consequently, $\omega^{(2)}$ can be thought of as the iterated integral
 \begin{align}\label{secorderprocess}
 \omegaa_{ts} =\int\limits_{s}^{t} (\delta\omega)_{rs}\otimes d\omega_{r}.
 \end{align}
 We emphasize that in order to construct (\ref{gub}) and thereafter the solution of (\ref{eq1}) one needs an appropriate algebraic and analytic setting which will be carefully analyzed in this work for stochastic evolution equation. The rigorous existence proof of (\ref{gub}) is based on a
  Sewing Lemma, see Lemma 4.2 in \cite{FritzHairer}. For more details on this topic consult \cite{Gubinelli} and \cite[Chapter 4]{FritzHairer}. \\
 
 This opens the door for the theory of rough SDEs using a completely pathwise approach. The only part where the stochastic analysis plays a role is hidden in (\ref{secorderprocess}). 
 Keeping this in mind one can solve (\ref{eq1}) by a fixed-point argument in the space of controlled rough paths. Regarding this, one can easily show that the solution of (\ref{eq1}) (recall $F\equiv 0$) is given by the pair 
	\begin{align}\label{sol:f:d}
	(y,y')= \left(S(\cdot)\xi  + \int\limits_{0}^{\cdot}S(\cdot-r) G(y_{r}) d\omega_{r}, G(y_{\cdot}) \right).
	\end{align}	The essential tool in defining (\ref{gub}) and  proving that (\ref{sol:f:d}) is the right object to solve (\ref{eq1}) in the finite-dimensional case is the regularity of $(S(t))_{t\geq 0}$. Note that a (semi)group generated by linear bounded operators is Lipschitz continuous, therefore the necessary H\"older regularity of the terms appearing in (\ref{crp}) and (\ref{gub}) cannot be influenced. More precisely,
	one can easily show that for a controlled rough path $(y,y')$ as specified in (\ref{crp}), the convolution with $(S(t))_{t\geq 0}$, i.e.~$(S(t-\cdot)y, S(t-\cdot)y')$ is again a controlled rough path. Due to this fact one can define $\int\limits_{s}^{t}S(t-r)y_{r}d\omega_{r}$ by (\ref{gub})
	and show that the mapping $$(y,y')\mapsto \left(\int\limits_{0}^{\cdot} S(\cdot-r)y_{r}d\omega_{r}, y_{\cdot}\right) $$ is linear and continuous on the space of controlled rough paths. Moreover, the composition of a controlled rough path with a regular function, is a  well-defined operation according to Lemma 7.3. in \cite{FritzHairer}. Consequently, $\int\limits_{s}^{t}S(t-r)G(y_{r})d\omega_{r}$ fits perfectly in the framework of (\ref{gub}). Regarding the notations introduced above one immediately observes that 
	$$\int\limits_{s}^{t}S(t-r) G(y_{r})d\omega_{r} $$ corresponds to
	\begin{align*}
	(\hat\delta y )_{ts} = (\delta y)_{ts} - (S(t-s)-\mbox{Id})y_{s}.
	\end{align*}

	Finally, by an appropriate fixed-point argument one establishes that (\ref{sol:f:d}) solves (\ref{eq1}). This would be the story in a nutshell of the finite-dimensional setting. For further information and applications see \cite{Gubinelli}, \cite{FritzHairer}, \cite{FritzVictoir}. \\
	
	However, in the infinite-dimensional case, since the analytic $C_{0}$-semigroup $(S(t))_{t\geq 0}$ is not Lipschitz continuous (not even H\"older continuous in $0$, recall Lemma \ref{betabeta}) it is no longer straightforward what are the appropriate objects required in order to obtain something similar to (\ref{sol:f:d}). It turns out that one has to construct additional supporting processes, consult also \cite{DeyaGubinelliTindel} in order to find the right way to define (\ref{integral}) together with the corresponding pair $(y,y')$ that solves (\ref{eq1}). This is the main topic of our work and for the beginning we illustrate heuristically the main ideas, which will be justified by the computations in the next sections. Furthermore, we stress that the noisy input $\omega$ is infinite-dimensional in contrast to \cite{DeyaGubinelliTindel}. Therefore one needs to make sure that the L\'evy-area $\omega^{(2)}$ exists in this case, consult \cite{GarridoLuSchmalfuss3} and the references specified therein.\\
	
	We make preliminary deliberations which will lead us to the right definition of (\ref{integral}). To this aim, similar to \cite{GarridoLuSchmalfuss2}, we firstly assume that $\omega$ is smooth and consider the following approximation of the integral:
		\begin{align*}
		\int\limits_{s}^{t}{S(t-r)G(y_r) d\omega_r} 
		&= \sum\limits_{[u,v]\in \aP} S(t-v) \int\limits_{u}^{v}{S(v-r)G(y_r) d\omega_r} \\
		&\approx \sum\limits_{[u,v]\in \aP} S(t-v) \int\limits_{u}^{v}{S(v-r)\left[G(y_u)+DG(y_u) (\delta y)_{ru} \right] d\omega_r} \\
		&=: \sum\limits_{[u,v]\in \aP} S(t-v) \Big[\omega^S_{vu}(G(y_u)) 
		+ \int\limits_{u}^{v}{S(v-r) DG(y_u) (\delta y)_{ru} d\omega_r}\Big].
		\end{align*} 
	In the first step we just plugged in the definition of the integral using Riemann-Stieltjes sums and in the second step we employed a Taylor expansion for $G$. Furthermore, we introduced the notation
		\begin{align}\label{omegas_h}
		\omega^{S}_{vu}(G(y_{u})):=\int\limits_{u}^{v} S(v-r) G (y_{u}) d\omega_{r},
		\end{align}
		respectively
		
		\begin{align}\label{z_h}
		z_{vu}(DG(y_{u})):= \int\limits_{u}^{v} S(v-r)  DG(y_{u}) (\delta y)_{ru}d\omega_{r}.
		\end{align}
		
		Since $\omega$ is smooth, all the expressions above are well-defined. We argue in Section \ref{sec:supp:proc} how to define the first integral for a rough input $\omega$ and derive important properties of $\omega^S$, using an integration by parts formula and regularizing properties of analytic semigroups.  
		Unfortunately, it is not at all clear how to  define $z$ if $\omega$ is not smooth. Therefore we have to continue our deliberations. \\
		
		The strategy is to construct the integral using an appropriate Sewing Lemma as derived in Section \ref{sectsl}. To this aim we need to introduce several processes satisfying appropriate analytic and algebraic conditions. We describe the general intuition of this approach which will allow us to define the integral $\mathcal{I}$ of
		\begin{align*}
		\Xi_{vu}^{(y)} := \Xi_{vu}^{(y)}(y,z):= \omega^S_{vu}(G(y_u)) + z_{vu}(DG(y_u)). 
		\end{align*}
		In order to employ the Sewing Lemma (Theorem \ref{lemma_sewing}) to obtain the existence together with suitable estimates of $\mathcal{I}\Xi^{(y)}$ we firstly have to compute (as rigorously justified in Section \ref{sectsl})
		\begin{align*}
		(\hat\delta_{2}\Xi^{(y)})_{vmu}=\Xi^{(y)}_{vu} - \Xi^{(y)}_{vm} - S(v-m) \Xi^{(y)}_{mu}.
		\end{align*}
		
		We can easily check
		\begin{align*}
		(\hat\delta_2\Xi^{(y)})_{vmu} 
		&= (\hat\delta_2 \omega^S)_{vmu}(G(y_u)) + \omega^S_{vm}(G(y_u)-G(y_m)) \\ 
		&+ (\hat\delta_2 z)_{vmu}(DG(y_u)) + z_{vm}(DG(y_u)-DG(y_m)).
\end{align*}
The first term obviously results in 
\begin{align*}
(\hat\delta_2 \omega^S)_{vmu}(G(y_u)) &= \omega^{S}_{vu}(G(y_{u})) - \omega^{S}_{vm}(G(y_{u})) - S(v-m)\omega^{S}_{mu}(G(y_{u}))\\
& = \int\limits_{u}^{v} S(v-r)G(y_{u})d\omega_{r} - \int\limits_{m}^{v} S(v-r)G(y_{u}) d\omega_{r} - S(v-m)\int\limits_{u}^{m} S(m-r) G(y_{u}) d\omega_{r}\\
& =0.
\end{align*}

Consequently,

\begin{align}\label{deltay}
		(\hat\delta_2\Xi^{(y)})_{vmu}
		&= \omega^S_{vm}(G(y_u)-G(y_m)) + (\hat\delta_2 z)_{vmu}(DG(y_u)) + z_{vm}(DG(y_u)-DG(y_m)).
		\end{align}
		Hence, it remains to investigate $\hat\delta_2 z(E)$, where $E \in \aL(W \otimes V,W)$ denotes a placeholder. For smooth paths $\omega$ we have a canonically given $z$ and compute 
		\begin{align*}
		(\hat\delta_2 z)_{vmu} (E)
		& = z_{vu}(E) - z_{vm}(E)-  S(v-m)z_{mu}(E)\\
		&=\int\limits_{u}^{v}{S(v-r) E (\delta y)_{ru} d\omega_r} 
		- \int\limits_{m}^{v}{S(v-r) E (\delta y)_{rm} d\omega_r} -
		\int\limits_{u}^{m}{S(v-r) E (\delta y)_{ru} d\omega_r} \\
		&=  \int\limits_{m}^{v}{S(v-r) E (\delta y)_{mu} d\omega_r} 
		=  \omega^S_{vm} (E (\delta y)_{mu}).
		\end{align*}
		As already mentioned, this term indeed exists even for a rough trajectory $\omega \in C^\alpha$ for $\alpha\in(1/3,1/2]$. If we assume that the algebraic relation 
		\begin{align}\label{algebraic:h}
		(\hat\delta_2 z)_{vmu} (E) = \omega^S_{vm} (E (\delta y)_{mu})
		\end{align}
		 holds true for any $E \in  \aL(W \otimes V,W)$ we obtain
		\begin{align*}
		(\hat\delta_2\Xi^{(y)})_{vmu} = \omega^S_{vm}(G(y_u)-G(y_m)+DG(y_u)(\delta y)_{mu}) + z_{vm}(DG(y_u)-DG(y_m)).
		\end{align*}
	Having this structure for $(\hat\delta_{2}\Xi^{(y)})_{vmu}$, under suitable regularity assumptions on $y$ and $z$ specified in Section \ref{sectsl} we are able to define
		\begin{align}\label{hy}
		\yb_t := \aI\Xi^{(y)}_t \quad \text{and} \quad \yt_t: = S(t) \xi + \yb_t .
		\end{align}
		Note that $\mathcal{I}\Xi^{(y)}_{t}$
	corresponds to (\ref{integral}) and 
	\begin{align*}
	(\hat\delta \mathcal{I}\Xi^{(y)})_{ts}=\int\limits_{s}^{t}S(t-r)G(y_{r})d\omega_{r}.
	\end{align*}
		\begin{remark}
			\begin{itemize}
				\item [1)] 	Since we demanded the existence of a suitable $z$ in order to construct the rough integral, it is necessary to define $\zt$ fulfilling 
				$(\hat\delta_2 \zt)_{vmu} (E) = \omega^S_{vm} (E (\delta \yt)_{mu})$. Only if this is valid we are able to iterate the solution mapping. 
				\item [2)] Note that if $S(\cdot)=\mbox{Id}$, then the algebraic relation~\eqref{algebraic:h} reduces to 
				$$ (\delta_{2} z)_{vmu}(E)=E (\delta y)_{mu}\otimes (\delta\omega)_{vm},$$
				compare~\eqref{chen}.
			\end{itemize}
		\end{remark}
			Again, for smooth $\omega$, $\zt$ is canonically given by
		\begin{align*}
		\zt_{ts}(E) 
		&= \int\limits_{s}^{t}{S(t-r) E (\delta \yt)_{rs} d\omega_r} = \sum\limits_{[u,v]\in \aP} S(t-v) \int\limits_{u}^{v}{S(v-r) E (\delta \yt)_{rs} d\omega_r} \\
		 &=\sum\limits_{[u,v]\in \aP} S(t-v) \int\limits_{u}^{v}{S(v-r) E (\hat\delta \yt)_{ru} d\omega_r} \\
		&+ \sum\limits_{[u,v]\in \aP} S(t-v) \int\limits_{u}^{v}{S(v-r) E S(r-u) \yt_{u} d\omega_r}- \sum\limits_{[u,v]\in \aP} S(t-v) \int\limits_{u}^{v}{S(v-r) E \yt_s d\omega_r}.
		\end{align*}
		Since $(\hat\delta \yt)_{ru}=\int\limits_{u}^{r}{S(r-q)G(y_q) d\omega_q}$ we have
		\begin{align*}
		\zt_{ts}(E) 
		= &\sum\limits_{[u,v]\in \aP} S(t-v) \int\limits_{u}^{v}{S(v-r) E \int\limits_{u}^{r}{S(r-q)G(y_q) d\omega_q} d\omega_r} \\
		&+ \sum\limits_{[u,v]\in \aP} S(t-v) a_{vu}(E,\yt_u) - \omega^S_{ts}(E \yt_s) \\
		\approx &\sum\limits_{[u,v]\in \aP} S(t-v) \int\limits_{u}^{v}{S(v-r) E \int\limits_{u}^{r}{S(r-q)G(y_u) d\omega_q} d\omega_r} \\
		&+ \sum\limits_{[u,v]\in \aP} S(t-v) a_{vu}(E,\yt_u) - \omega^S_{ts}(E \yt_s) \\ 
		=: &\sum\limits_{[u,v]\in \aP} S(t-v) \left[b_{vu}(E,G(y_u)) + a_{vu}(E,\yt_u)\right]  - \omega^S_{ts}(E \yt_s).
		\end{align*}
		Here we introduced the notation
		\begin{align*}
		b_{vu}(E,G(y_{u})) :=\int\limits_{u}^{v} S(v-r) E \int\limits_{u}^{r} S(r-q) G (y_{u}) d\omega_{q}d\omega_{r},
		\end{align*}
		respectively
		\begin{align*}
		a_{vu}(E,\widetilde{y}_{u}):=\int\limits_{u}^{v} S(v-r) E S (r-u) \widetilde{y}_{u} d\omega_{r}.
		\end{align*}
		
Hence, we set
\begin{align*}
&\Xi^{(z)}(y,\yt)_{vu}(E) := b_{vu}(E,G(y_u)) + a_{vu}(E,\yt_u), \\
& \zt_{ts}(E) := (\hat\delta \aI \Xi^{(z)}(y,\yt))_{ts}(E)-\omega^S_{ts}(E \yt_s).
\end{align*}

			This means that we have to define $a$, $b$ 
		and $\omega^{S}$ in order to describe $\widetilde{z}$. At the very first sight, it is not straightforward under which assumptions $b$ is well-defined, compare Remark 4.3 in  \cite{DeyaGubinelliTindel}. This problem will be addressed in Section \ref{sec:supp:proc}. For the sake of completeness we provide here a possible heuristic definition of $b$ which will be shown in Section \ref{sec:supp:proc} to be the right one. For a smooth path $\omega$ and a placeholder $K$ which stands for $G(y_{\cdot})$ we have
		\begin{align}
		\int\limits_{s}^{t}{S(t-r) E \int\limits_{s}^{r}{S(r-q)K}d\omega_q} d\omega_r 
		= &\sum\limits_{[u,v]\in \aP} S(t-v) \int\limits_{u}^{v}{S(v-r) E \int\limits_{s}^{r}{S(r-q)K}d\omega_q} d\omega_r\nonumber \\
		= &\sum\limits_{[u,v]\in \aP} S(t-v) \int\limits_{u}^{v}{S(v-r) E \int\limits_{s}^{u}{S(r-q)K}d\omega_q} d\omega_r\nonumber \\
		&+ \sum\limits_{[u,v]\in \aP} S(t-v) \int\limits_{u}^{v}{S(v-r) E \int\limits_{u}^{r}{S(r-q)K}d\omega_q} d\omega_r\nonumber \\
		\approx & \sum\limits_{[u,v]\in \aP} S(t-v) \int\limits_{u}^{v}{S(v-r) E \int\limits_{s}^{u}{S(u-q)K}d\omega_q} d\omega_r \nonumber\\
		&+ \sum\limits_{[u,v]\in \aP} S(t-v) \int\limits_{u}^{v}{S(v-r) E \int\limits_{u}^{r}{K}d\omega_q} d\omega_r\nonumber \\ 
		= :& \sum\limits_{[u,v]\in \aP} S(t-v) \left[\omega^S_{vu} (E \omega^S_{us}(K)) + c_{vu}(E,K) \right],\label{hb}
		\end{align} 
		where
		\begin{align*}
		c_{ts}(E,K):=\int\limits_{s}^{t} S(t-r) E K (\delta\omega)_{rs}  d\omega_{r}. 
		\end{align*}

The existence of $a$, $b$, $c$ and of all other auxiliary processes required in order to give a meaning to (\ref{integral}) will be justified in Section \ref{sec:supp:proc}. Motivated by this heuristic computations we first define similar to \cite{GarridoLuSchmalfuss2} these processes for smooth paths $\omega^{n}$ approximating $\omega$. Thereafter the passage to the limit entails a suitable construction/interpretation of all these expressions.\\

Concluding this heuristic computations, we introduce the following definition of a solution for~\eqref{eq1} (compare~\eqref{sol:f:d}). This is the counterpart of the solution concepts investigated in~\cite{HuNualart} and~\cite{GarridoLuSchmalfuss1}.\\

Recalling that $\omega$ is still assumed to be a smooth path we have.
\begin{definition}We call a pair $(y,z)$ mild solution for~\eqref{eq1} if
\begin{align}
y_{t} &= S(t) \xi + \aI \Xi^{(y)}(y,z)_{t}\nonumber\\
& = S(t) \xi + \lim\limits_{|\aP([0,t])|\to 0}\sum\limits_{[u,v]\in\aP([0,t])} S(t-v) [\omega^{S}_{vu}(G(y_{u})) + z_{vu} (DG(y_{u}))  ]\\
z_{ts}(E)&= (\hat\delta \aI \Xi^{(z)}(y,y))_{ts}(E) - \omega^S_{ts}(E y_s) \nonumber \\
&=\lim\limits_{|\aP([s,t])| \to 0 }\sum\limits_{[u,v]\in\aP([s,t])} S(t-v) [b_{vu}(E,G(y_{u})) + a_{vu}(E,y_{u})] -\omega^{S}_{ts}(Ey_{s}) .
\end{align}	
\end{definition}
Our aim is to rigorous justify this solution theory for a path $\omega$ having the regularity of a fractional Brownian motion with Hurst index $H\in(1/3,1/2]$, i.e.~$\omega\in C^{\alpha}$ with $1/3<\alpha\leq 1/2$.
This will be carried out in Section~\ref{sectfp} by means of a fixed-point argument in a suitable function space specially designed to incorporate the analytic and algebraic properties of the solution pair $(y,z)$. \\

Finally, we introduce further notations which will turn out to be useful for the computations in Section \ref{sectfp}.
\begin{remark}
Going back to the definition of $\widetilde{z}$, recalling (\ref{hy}), applying the bilinearity of $a$ and the linearity of $\omega^S$ entails
\begin{align*}
\zt_{ts}(E) 
\approx &\sum\limits_{[u,v]\in \aP} S(t-v) \left[b_{vu}(E,G(y_u)) + a_{vu}(E,\yb_u) + a_{vu}(E,S(u)\xi)\right] \\
&- \omega^S_{ts}(E \yb_s) - \omega^S_{ts}(E S(s) \xi).
\end{align*}
As justified in Section \ref{sec:supp:proc} (Corollary \ref{corollary_a_sum}) we get
\begin{align*}
\zt_{ts}(E)
\approx &\sum\limits_{[u,v]\in \aP} S(t-v) \left[b_{vu}(E,G(y_u)) + a_{vu}(E,\yb_u)\right] - \omega^S_{ts}(E \yb_s)\\
 &+a_{ts}(E,S(s)\xi) - \omega^S_{ts}(E S(s) \xi).
\end{align*}
Therefore, we can define
\begin{align*}
&\Xi^{(z)}(y,\yb)_{vu}(E):= b_{vu}(E,G(y_u)) + a_{vu}(E,\yb_u),  \\
&\zb_{ts}(E) := (\hat\delta \aI \Xi^{(z)}(y,\yb))_{ts}(E)-\omega^S_{ts}(E \yb_s), ~\text{ which yields }\\
&\zt_{ts}(E) = \zb_{ts}(E) + a_{ts}(E,S(s)\xi) - \omega^S_{ts}(E S(s) \xi).
\end{align*}
In Section \ref{sectfp} we will see that it more convenient to estimate $\zb$ than $\zt$. 
\end{remark}

\vspace*{0.2 cm}
We conclude this section pointing out the following fact which is valid for smoother $\omega$, i.e. $\omega\in C^{\alpha}$ with $\alpha\in(1/2,1)$.
\begin{remark}
Note that if the $\omega$ is $\alpha$-H\"older continuous with $\alpha\in(1/2,1)$, i.e. has the regularity of a fractional Brownian motion with Hurst index $H\in(1/2,1)$, then (\ref{integral}) can be defined using the Young integral, namely
	\begin{align*}
	\int\limits_{s}^{t} S(t-r) G(y_{r}) d\omega_{r} &=\lim\limits_{|\aP|\to 0} \sum\limits_{[u,v]\in\aP} S(t-v) \int\limits_{u}^{v} S(v-r) G(y_{r}) d\omega_{r}\\
		&\approx\lim\limits_{|\aP|\to 0} \sum\limits_{[u,v]\in\aP} S(t-v) \omega^{S}_{vu}(G(y_{u})).
	\end{align*}
This simplifies significantly the arguments. Moreover, in this setting the Young integral coincides with the rough integral, see \cite{DeyaGubinelliTindel}.
\end{remark}


		\section{Sewing Lemma}\label{sectsl}
		The following result is crucial for our work, since it gives us the existence of the rough integral together with all the necessary properties required to solve (\ref{eq1}). Since we work with the weighted H\"older spaces introduced in Section \ref{preliminaries} we have to extend the results obtained in \cite[Section 3]{GubinelliTindel} using similar techniques. The next statement is the analogue of Theorem 3.5 in \cite{GubinelliTindel} in our framework.
		
			 \begin{theorem}[Sewing Lemma] \label{lemma_sewing}
			Let $W$ be a separable Banach space and $(S(t))_{t\geq 0}$ be an analytic $C_{0}$-semigroup on $W$ with $\left\| S(t)\right\|_{\mathcal{L}(W)}\leq c_{S}$ for all $t \leq T$.
			Furthermore, let $\Xi \in C \! \left( \Delta , W\right)$ be an approximation term satisfying the following properties
			for all $0 \leq u \leq m \leq v \leq T$ :
			\begin{align}
			\left|\Xi_{vu}\right| &\leq c_1 \left(v-u \right)^{\alpha},\label{assumption_Xi}\\
			\left|\left(\hat\delta_2 \Xi\right)_{vmu} \right| & \leq c_2 u^{-\beta} \left(v-u \right)^{\rho},~~\mbox{for } u\neq 0.\label{assumption_deltaXi1}
			\end{align}
			Here we impose $0 \leq\alpha,~\beta\leq 1$, $\rho>1$ and $\alpha+\beta \leq \rho$.\\			
			
			Then there exists a unique $\aI\Xi \in C\!\left(\left[0,T\right],W \right)$, such that 
			\begin{align}
			\aI\Xi_0 &=0, \\
			\left|\left(\hat\delta \aI \Xi \right)_{ts} \right|&\leq C \left(c_1+c_2 \right) \left(t-s \right)^{\alpha} 
			\label{property_deltaIXiest}\\
			\left|(\hat\delta \aI \Xi)_{ts}-\Xi_{ts} \right|&\leq C c_2 s^{-\beta} \left(t-s \right)^{\rho}, ~~\mbox{for } s\neq 0.
			\label{property_IXiest1}
			\end{align}
		\end{theorem}
		\begin{proof}
			Firstly, note that the uniqueness of $\aI\Xi$ immediately follows from Lemma \ref{lemma_uniqueness}. Assuming by contradiction that there are two candidates $\mathcal{I}^1$ and $\mathcal{I}^2$ for a given $\Xi$, we have 
			\begin{align*}
			\mathcal{I}^1_0-\mathcal{I}^2_0&=0, \\
			\left|(\hat\delta(\mathcal{I}^1-\mathcal{I}^2))_{ts} \right| &\leq C \left(c_1+c_2 \right) \left(t-s \right)^{\alpha}, \\
			\left|(\hat\delta(\mathcal{I}^1-\mathcal{I}^2))_{ts} \right| &\leq C c_2 s^{-\beta} \left(t-s \right)^{\rho},~~\mbox{for } s\neq 0.
			\end{align*}
			Hence, Lemma \ref{lemma_uniqueness} implies that $\mathcal{I}^1 \equiv \mathcal{I}^2$. \\
			
			The following deliberations are conducted in order to prove the existence of $\aI\Xi$. To this aim, given $0 \leq s < t \leq T$, we let $\aP_n=\aP_n(s,t)$ be the $n$-th dyadic partition of $[s,t]$ for $n \in \IN_0$ and define
			\begin{align*}
			N^n_{ts} &:= \sum\limits_{\left[u,v\right]\in \aP_n} S(t-v) \Xi_{vu}, \\
			M^n_{ts} &:= \Xi_{ts} - N^n_{ts}. \\
			\end{align*}
			Note that $N^0_{ts} = \Xi_{ts}$ which implies that $ M^0_{ts} =0$. 
			\\
			
			Furthermore, setting $m:=\frac{u+v}{2}$, we derive
			\begin{align*}
			N^n_{ts} - N^{n+1}_{ts} = M^{n+1}_{ts} - M^{n}_{ts} 
			&= \sum\limits_{\left[u,v\right]\in \aP_n} \left(S(t-v) \Xi_{vu} -S(t-v) \Xi_{vm} - S(t-m) \Xi_{mu} \right) \\
			&= \sum\limits_{\left[u,v\right]\in \aP_n} S(t-v) (\hat\delta_2 \Xi)_{vmu}.
			\end{align*}
			Hence, we obtain
			\begin{align}
			\left| M^n_{ts} - M^{n+1}_{ts} \right| 
			\leq & C \sum\limits_{[u,v]\in \aP_n} \left|(\hat\delta_2 \Xi)_{vmu}\right|. \label{estimate_Mn}
			\end{align}
			Since we also have to deal with the case $s=0$, we apply \eqref{assumption_Xi} to the first term and use \eqref{assumption_deltaXi1} to estimate the other terms in \eqref{estimate_Mn}. This further entails
			\begin{align*}
			\left| M^n_{ts} - M^{n+1}_{ts} \right| 
			\leq & C c_1 \left(t-s\right)^{\alpha} 2^{-n \alpha}
			+ C \sum\limits_{\substack{\left[u,v\right]\in \aP_n \\ u \neq s} } { c_2 u^{-\beta} \left(t-s\right)^{\rho}2^{-n\rho}} \\
			\leq & C c_1 \left(t-s\right)^{\alpha} 2^{-n \alpha} 
			+ C c_2 \left(t-s\right)^{\rho-1} 2^{-n(\rho-1)} \sum\limits_{\substack{\left[u,v\right]\in \aP_n \\ u \neq s} } {u^{-\beta} (t-s) 2^{-n}} \\
			\leq & C c_1 \left(t-s\right)^{\alpha} 2^{-n \alpha} 
			+ C c_2 \left(t-s\right)^{\rho-1} 2^{-n(\rho-1)} \int\limits_{s}^{t} q^{-\beta} dq \\
			\leq & C c_1 \left(t-s\right)^{\alpha} 2^{-n \alpha} 
			+ C c_2 \left(t-s\right)^{\rho-\beta} 2^{-n(\rho-1)}.
			\end{align*}
			Since this expression is summable, we conclude that $M^n_{ts} \to M_{ts}$ as $n\to\infty$ for all $0 \leq s \leq t \leq T$. The previous computations give us the estimate
			\begin{align}
			\left|M_{ts} \right| \leq & C \left(c_1+c_2\right) \left(t-s\right)^{\alpha}.\label{estimate_M1}
			\end{align}
			Note that this is valid due to the fact that $\alpha+\beta\leq\rho$.\\
			
			Setting $N_{ts}:= \Xi_{ts}- M_{ts}$ immediately entails $N^n_{ts} \to N_{ts}$ as $n\to\infty$ for all $0 \leq s \leq t \leq T$ and
			\begin{align}
			\left|N_{ts} \right| \leq & C \left(c_1+c_2\right) \left(t-s\right)^{\alpha}.\label{estimate_N1}
			\end{align}
			Furthermore this also yields that $(\hat\delta_2 N) \equiv 0$.\\
			

			To prove this statement, note that it is equivalent to show that
			\begin{align*}
			(\hat\delta_2 N^n)_{t \tau s} \to 0, \quad \text{as } n \to 0,
			\end{align*}
			for all $0 \leq s \leq \tau \leq t \leq T$.\\
			
			To this aim we consider  a fixed interval $[u,v] \in \aP_n(s,t)$ and a finite number of nodes $u=r_0< r_1 < r_2 <\ldots < r_k <r_{k+1}=v$, for $k \in \IN$. Then,
			\begin{align*}
			S(t-v) \Xi_{vu} - \sum\limits_{j=0}^{k}{S(t-r_{j+1}) \Xi_{r_{j+1}r_j}}
			= \sum\limits_{j=0}^{k-1}{S(t-v) (\hat\delta_2 \Xi)_{v r_{j+1}r_j}}.
			\end{align*}
			Hence, we estimate
			\begin{align}
			\left|S(t-v) \Xi_{vu} - \sum\limits_{j=0}^{k}{S(t-r_{j+1}) \Xi_{r_{j+1}r_j}} \right|
			\leq 
			\begin{cases}
			C c_2 k u^{-\beta} (v-u)^{\rho}, & u>s \\
			C c_1 k (v-u)^\alpha, & u=s.
			\end{cases} \label{estimate_helphatdelta}
			\end{align}
			For fixed $0 \leq s \leq \tau \leq t \leq T$ we define $\pi_n(s,t):=\aP_n(s,t)\cup \aP_n(s,\tau)\cup \aP_n(\tau,t)$ and $\pi_n(s,\tau):=\pi_n(s,t)\cap[s,\tau]$, $\pi_n(\tau,t):=\pi_n(s,t)\cap[\tau,t]$.\\
			
			We define $N^{\pi_n}_{ts}$, $N^{\pi_n}_{t\tau}$ and $N^{\pi_n}_{\tau s}$ analogously to $N^n$, i.e.
			\begin{align*}
			N^{\pi_n}_{ts}:= \sum\limits_{[u,v] \in \pi_n(s,t)}{S(t-v) \Xi_{vu}}.
			\end{align*}

			Since $\pi_n(s,t)=\pi_n(s,\tau) \cup \pi_n(\tau,t)$ one obtains that $(\hat\delta_2 N^{\pi_n})_{t \tau s}=0$. 
			Consequently, we estimate
			\begin{align*}
			\left|(\hat\delta_2 N^n)_{t \tau s} \right|  
			\leq &\left|N^n_{ts} -N^{\pi_n}_{ts} \right| 
			+ \left|N^n_{t\tau} -N^{\pi_n}_{t\tau} \right| 
			+ \left|S(t-\tau) \left(N^n_{\tau s} -N^{\pi_n}_{\tau s}\right) \right|.
			\end{align*}
			Therefore it is left to show that all three summands tend to zero as $n \to \infty$. To this aim consider an arbitrary interval $[u,v] \in \aP_n(s,t)$ and let $k := \left\lfloor \frac{t-s}{\tau-s}\right\rfloor \vee \left\lfloor \frac{t-s}{t-\tau}\right\rfloor$. Then there are at most $k+1$ many nodes of $\pi_n$ between $u$ and $v$.
			Thus by \eqref{estimate_helphatdelta} we have  
			\begin{align*}
			\left|N^n_{ts} -N^{\pi_n}_{ts} \right| 
			\leq C (k+1) (c_1+c_2) \left(\left(\frac{t-s}{2^n} \right)^\alpha 
			+ \sum\limits_{\substack{\left[u,v \right]\in \aP_n(s,t) \\ u \neq s}} u^{-\beta} \left(\frac{t-s}{2^n}\right)^\rho \right) \to 0, \text{ as } n \to \infty. 
			\end{align*}
			The other summands tend to zero analogously.\\
			
			Since $(\hat\delta_2 N) \equiv 0$ we can apply Lemma \ref{lemma_preliminaries_hatdelta}. This ensures the unique existence of $\aI\Xi \in C([0,T],W)$ such that 
			\begin{align}
			\aI\Xi_0 &=0, \nonumber\\
			(\hat\delta \aI\Xi)_{ts} &= N_{ts} = \Xi_{ts}- M_{ts},  \text{  for all } 0 \leq s < t \leq T\label{stern}.
			\end{align}
			Hence, \eqref{estimate_N1} implies
			\begin{align*}
			\left| (\hat\delta \aI\Xi)_{ts} \right| \leq C \left(c_1+c_2\right) \left(t-s\right)^{\alpha}.
			\end{align*}
			We now show \eqref{property_IXiest1}. To this aim, if $s>0$ we apply \eqref{assumption_deltaXi1} to all summands in (\ref{estimate_Mn}) and obtain
			\begin{align*}
			\left| M^n_{ts} - M^{n+1}_{ts} \right| 
			\leq &C \sum\limits_{\left[u,v\right]\in \aP_n} { c_2 u^{-\beta} \left(t-s\right)^{\rho}2^{-n\rho}} \\
			\leq & C c_2 s^{-\beta} \left(t-s\right)^{\rho} 2^{-n(\rho-1)}.
			\end{align*} 
			Consequently,
			\begin{align}
			\left|M_{ts} \right| 
			\leq & C c_2 s^{-\beta} \left(t-s\right)^{\rho}, \label{estimate_M2}
			\end{align}
			which yields \eqref{property_IXiest1}.
		\qed \\ \end{proof}

		 The following result gives us an additional estimate necessary for the fixed-point argument. 
		\begin{corollary}\label{corollary_sewing_IXiestimate}
			Additionally to the assumptions of Theorem \ref{lemma_sewing}, we let
			\begin{align}
			\left|\left(\hat\delta_2 \Xi\right)_{vmu} \right| & \leq c_3 u^{-\beta'} \left(v-u \right)^{\rho'},\label{assumption_deltaXi2}
			\end{align}
			with $0 \leq\beta,~\beta' \leq 1$ and $\rho'-\beta' \leq \rho-\beta$.\\[1ex]
			Then it holds 
			\begin{align}
			\left|(\hat\delta \aI \Xi)_{ts}-\Xi_{ts} \right|&\leq C (c_2+c_3) s^{-\beta'} \left(t-s \right)^{\rho'}.
			\label{property_IXiest2}
			\end{align}
		\end{corollary}
		\begin{proof}
			The proof is analogous to the previous one. Recalling that
			\begin{align*}
			\left| M^n_{ts} - M^{n+1}_{ts} \right| 
			\leq & C \sum\limits_{[u,v]\in \aP_n} \left|(\hat\delta_2 \Xi)_{vmu}\right|, 
			\end{align*}
			we apply \eqref{assumption_deltaXi2} to the first summand and again \eqref{assumption_deltaXi1} to the other terms. This leads to
			\begin{align*}
			\left| M^n_{ts} - M^{n+1}_{ts} \right| 
			\leq & C c_3 s^{-\beta'} \left(t-s\right)^{\rho'} 2^{-n \rho'}
			+ C \sum\limits_{\substack{\left[u,v\right]\in \aP_n \\ u \neq s} } { c_2 u^{-\beta} \left(t-s\right)^{\rho}2^{-n\rho}} \\
			\leq & C c_3 s^{-\beta'} \left(t-s\right)^{\rho'} 2^{-n \rho'}
			+ C c_2 s^{-\beta'} \left(t-s\right)^{\rho-1} 2^{-n(\rho-1)} \sum\limits_{\substack{\left[u,v\right]\in \aP_n \\ u \neq s} } {u^{-(\beta-\beta')}(t-s) 2^{-n}} \\
			\leq & C c_3 s^{-\beta'} \left(t-s\right)^{\rho'} 2^{-n \rho'}
			+ C c_2 s^{-\beta'} \left(t-s\right)^{\rho-1} 2^{-n(\rho-1)} \int\limits_{s}^{t} q^{-(\beta-\beta')} dq \\
			\leq & C c_3 s^{-\beta'} \left(t-s\right)^{\rho'} 2^{-n \rho'} 
			+ C c_2 s^{-\beta'} \left(t-s\right)^{\rho-\beta+\beta'} 2^{-n(\rho-1)}.
			\end{align*}
			Since $\rho-\beta+\beta' \geq \rho'$ we have
			\begin{align*}
			\left| M_{ts}\right| \leq C (c_2+c_3) s^{-\beta'} \left(t-s\right)^{\rho'}.
			\end{align*}
		\qed \end{proof} \\
		In order to give a meaning to $\aI\Xi$ as a rough integral we firstly describe it as a limit of finite sums, compare Corollary 3.6 in \cite{Gubinelli}.  Note that in our case technical difficulties occur in the proof due to (\ref{estimate_M2}). 
		\begin{corollary}\label{corollary_sewing_integral}
			Under the assumptions of Theorem \ref{lemma_sewing} it holds that
			\begin{align}
			\left(\hat\delta \aI \Xi \right)_{ts} 
			&= \lim\limits_{\left|\aP\right| \to 0} \sum\limits_{\left[u,v\right] \in \aP} {S(t-v) \Xi_{vu}},
			\label{property_IXieq}
			\end{align}
			where $\left|\aP\right|$ stands for the mesh of the given partition $\aP=\aP(s,t)$.
		\end{corollary}
		\begin{proof}

			Consider an arbitrary partition $\aP$ of $[s,t]$. Then we have by \eqref{stern} that
			\begin{align*}
			\sum\limits_{\left[u,v\right]\in \aP}{S(t-v) \Xi_{vu}} 
			&= \sum\limits_{\left[u,v\right]\in \aP}{S(t-v) \left((\hat\delta \aI\Xi)_{v u} + M_{v u} \right)} \\
			&= (\hat\delta \aI\Xi)_{ts} + \sum\limits_{\left[u,v\right]\in \aP} {S(t-v) M_{v u}}.
			\end{align*}
			Therefore it is left to show that
	\begin{align*}
			\lim\limits_{|\aP| \to 0} \sum\limits_{\left[u,v\right]\in \aP}{S(t-v) M_{v u}} = 0. 
			\end{align*}
			To this aim, we prove the sufficient statement
			\begin{align*}
			\lim\limits_{|\aP| \to 0} \sum\limits_{\left[u,v\right]\in \aP}\left|M_{v u}\right| = 0.
			\end{align*}
As concluded within the proof of Theorem \ref{lemma_sewing} we have two estimates for $M$, recall \eqref{estimate_M1} and \eqref{estimate_M2}. Namely we obtained that
\begin{align*}
\left|M_{vu} \right| &\leq C(c_1+c_2) (v-u)^{\alpha}, 
\\
\left|M_{vu} \right| &\leq C c_2 u^{-\beta} \left(v-u \right)^{\rho},~~\mbox{for } u \neq 0. 
\end{align*}
Clearly, which one of them is more restrictive depends on the relation between $u$ and $v-u$. \\

Hence, we introduce $\widetilde{\aP}: = \left\{\left[u,v\right] \in \aP \colon u<v-u \right\}$. We order the intervals of $\widetilde{\aP}$ by their starting point and write $\widetilde{\aP} = \left\{\left[\widetilde{u}_k,\widetilde{v}_k\right] \colon k=1, \ldots, m\right\}$,
 where $s\leq\wtu_1<\wtv_1 \leq \wtu_2 < \wtv_2 \leq \ldots \leq \wtu_m < \wtv_m \leq t$.\\
 
For $k=1, \ldots m-1$ we get
\begin{align*}
\wtu_k &<\frac{\wtv_k}{2} \leq \frac{\wtu_{k+1}}{2},
\end{align*} which yields
\begin{align*}
\wtu_k &< \wtu_m \, 2^{-(m-l)} < \left(\wtv_m-\wtu_m\right) 2^{-(m-l)} \leq \left|\aP \right| 2^{-(m-l)}.
\end{align*}
All in all this means that
\begin{align*}
\wtv_k-\wtu_k &\leq \wtv_k \leq \wtu_{k+1} \leq \left|\aP \right| 2^{-(m-l-1)}.
\end{align*}
For $k=m$ we trivially have $\wtv_m-\wtu_m \leq \left|\aP\right|\leq 2 \left|\aP \right|$.
Hence, by using \eqref{estimate_M1} we derive
\begin{align*}
\sum\limits_{[u,v] \in \widetilde{\aP}} \left|M_{vu} \right|
\leq C (c_1+c_2) \sum\limits_{k=1}^{m} \left(\wtv_k-\wtu_k \right)^{\alpha} 
\leq C (c_1+c_2)\left|\aP\right|^{\alpha}  \sum\limits_{k=1}^{m} 2^{-(m-l-1)\alpha}
\leq C (c_1+c_2) \left|\aP\right|^{\alpha}.
\end{align*}
	If $\left[u,v\right] \in \aP \backslash \! \widetilde{\aP} $ we have 
			\begin{align*}
			v-u \leq u \mbox{, so }  v\leq 2 u, \mbox{ therefore }u^{-\beta} \leq 2^\beta v^{-\beta}.
			\end{align*}
So, by applying \eqref{estimate_M2} we infer
\begin{align*}
\sum\limits_{\left[u,v\right]\in \aP\backslash\widetilde{\aP}} \left|M_{v u}\right|
&\leq C c_2 \sum\limits_{\left[u,v\right]\in \aP\backslash \tilde{\aP} } u^{-\beta} \left(v-u \right)^{\rho} \\
&\leq C c_2 \left|\aP \right|^{\rho-1} \sum\limits_{\left[u,v\right]\in \aP\backslash \tilde{\aP} } 
v^{-\beta} \left(v-u \right) \\
&\leq C c_2 \left|\aP \right|^{\rho-1} \int\limits_{s}^{t} q^{-\beta} dq \\
&\leq C c_2 \left(t-s \right)^{1-\beta} \left|\aP \right|^{\rho-1}.
\end{align*}
Consequently, putting both estimates together we have
\begin{align*}
\sum\limits_{\left[u,v\right]\in \aP}\left|M_{v u}\right| 
&\leq \sum\limits_{\left[u,v\right] \in \widetilde{\aP}}\left| M_{v u}\right|
		+ \sum\limits_{\left[u,v\right]\in \aP\backslash\widetilde{\aP}} \left|M_{v u}\right| \\
&\leq C (c_1+c_2) \left|\aP\right|^{\alpha} + C c_2 \left(t-s \right)^{1-\beta} \left|\aP \right|^{\rho-1},
\end{align*}
 which tends to $0$ as $ \left|\aP \right| \to 0$. This proves the statement.
			
		\qed \end{proof}
		\begin{remark}
	Note that the above limit is independent of the approximating sequence of partitions.
\end{remark}
	
The next result gives us straightforward the shift property of the constructed rough integral.
To this aim we introduce for $\tau>0$
\begin{align}\label{shift:process}
\theta_\tau \Xi_{vu} := \Xi_{v+\tau,u+\tau}.
\end{align}

\begin{corollary}\label{corollary_sewing_shift}
Under the assumptions of Theorem \ref{lemma_sewing} we have the shift property, namely
	\begin{align*}
	(\hat\delta \aI \Xi)_{ts} = (\hat\delta \aI \theta_{\tau} \Xi)_{t-\tau,s-\tau}, ~~\mbox{for } \tau \leq s \leq t.
	\end{align*}
\end{corollary}
\begin{proof}
	The proof is a direct consequence of Corollary \ref{corollary_sewing_integral}.
	\begin{align*}
	(\hat\delta \aI \Xi)_{ts}
	&= \lim\limits_{\left|\aP \right| \to 0} \sum\limits_{[u,v] \in \aP(s,t)} S(t-v) \Xi_{vu}\\
	&= \lim\limits_{\left|\aP \right| \to 0} \sum\limits_{[u,v] \in \aP(s-\tau,t-\tau)} S(t-\tau-v) \Xi_{v+\tau,u+\tau}\\
	&= \lim\limits_{\left|\aP \right| \to 0} \sum\limits_{[u,v] \in \aP(s-\tau,t-\tau)} S(t-\tau-v) \theta_{\tau} \Xi_{vu} \\
	&= (\hat\delta \aI \theta_{\tau} \Xi)_{t-\tau,s-\tau}.
	\end{align*}
	\qed
\end{proof}\\
	The next result contains necessary estimates for $\hat\delta \mathcal{I}\Xi$ in a suitable fractional domain. These will be required later on (Corollary~\ref{corollary_sewing_deltaestimate}) to estimate $\delta\mathcal{I}\Xi$.
		\begin{corollary}\label{corollary_sewing_epsestimate}
			Under the assumptions of Theorem \ref{lemma_sewing} we have
			\begin{align}
			\left|\left(\hat\delta \aI \Xi \right)_{ts} \right|_{D_\eps}
			&\leq C \left(c_1+c_2 \right) \left(t-s \right)^{\alpha-\eps}  \text {,  for all } 0\leq \eps<\alpha.
			\label{property_deltaIXiest_eps}
			\end{align}
		\end{corollary}
		\begin{proof}
			Analogously to the proof of Theorem \eqref{lemma_sewing} we introduce
			\begin{align*}
			\overline{N}^n_{ts} &:= \sum\limits_{\substack{\left[u,v\right]\in \aP_n\\ v \neq t}} S(t-v) \Xi_{vu}, \\
			\overline{M}^n_{ts} &:= \Xi_{ts} - \overline{N}^n_{ts}. \\
			\end{align*}
			As mentioned in the proof of Theorem \ref{lemma_sewing} we have $\overline{N}^0_{ts} = 0 $ which means that $\overline{M}^0_{ts} =\Xi_{ts}$.\\
			
			We further set $\overline{v}_n := \max\left\{v < t \colon [u,v] \in \aP_n \right\}$. Then we derive
			\begin{align*}
			\oN^n_{ts} - \oN^{n+1}_{ts} & = \oM^n_{ts} - \oM^{n+1}_{ts} \\
			& = N^n_{ts} - N^{n+1}_{ts} - \Xi_{t \ov_{n}} + \Xi_{t \ov_{n+1}} \\
			&= \sum\limits_{\substack{\left[u,v\right]\in \aP_n \\ v \neq t}} S(t-v) (\hat\delta_2 \Xi)_{vmu} 
			+ (\hat\delta \Xi)_{t \ov_{n+1}\ov_{n}} - \Xi_{t \ov_{n}} + \Xi_{t \ov_{n+1}} \\
			&=\sum\limits_{\substack{\left[u,v\right]\in \aP_n \\ v \neq t}} S(t-v) (\hat\delta_2 \Xi)_{vmu} - S(t-\ov_{n+1}) \Xi_{\ov_{n+1}\ov_n}.
			\end{align*}
			For $\eps < \alpha$ we estimate
			\begin{align*}
			\left| \oN^n_{ts} - \oN^{n+1}_{ts}  \right|_{D_\eps} 
			\leq C \left(t-\ov_{n+1} \right)^{-\eps}\left|\Xi_{\ov_{n+1}\ov_n}\right| + C \sum\limits_{\substack{\left[u,v\right]\in \aP_n \\ v \neq t}} {\left(t-v \right)^{-\eps} \left|(\hat\delta_2 \Xi)_{vmu}\right|} .
			\end{align*}
			We now apply \eqref{assumption_Xi} to the first summand and \eqref{assumption_deltaXi1} to the others. Note that $t-\ov_{n+1}= \ov_{n+1}-\ov_n$ and for $n=0$ the sum is zero while for $n \geq 1 $ we get $t-v \geq \frac{t-s}{2}$ if $[s,v] \in \aP_n$. Keeping this in mind, we have
			\begin{align*}
			\left| \oN^n_{ts} - \oN^{n+1}_{ts}  \right|_{D_\eps}
			\leq & C c_1 (t-s)^{\alpha-\eps} 2^{-n(\alpha-\eps)} 
			+ C c_2 \sum\limits_{\substack{\left[u,v\right]\in \aP_n \\ u \neq s, v \neq t}} {\left(t-v \right)^{-\eps} u^{-\beta} (t-s)^{\rho}2^{-n\rho}} \\
			= & C c_1 (t-s)^{\alpha-\eps} 2^{-n(\alpha-\eps)} 
			+ C c_2 (t-s)^{\rho}2^{-n\rho} \sum\limits_{\substack{\left[u,v\right]\in \aP_n \\ u \neq s, v \neq t}} {\left(t-v \right)^{-\eps} u^{-\beta} }.
			\end{align*}
			We now have to estimate the term
			\begin{align*}
			J_{ts}
			&:= \sum\limits_{\substack{\left[u,v\right]\in \aP_n \\ u \neq s , v \neq t} } {\left(t-v \right)^{-\eps} u^{-\beta} }\\
			&= \sum\limits_{k=1}^{2^n-2} \left(s+\frac{k (t-s)}{2^n} \right)^{-\beta} \left(t-s-\frac{(k+1)(t-s)}{2^n} \right)^{-\eps}\\
			&\leq \left(t-s \right)^{-\beta-\eps} 2^{n(\beta+\eps)} \sum\limits_{k=1}^{2^n-2} {k^{-\beta} \left(2^n-1-k \right)^{-\eps}}.
			\end{align*}
			By Lemma \ref{lemma_appendix1}  we obtain
			\begin{align*}
			J_{ts}& \leq \left(t-s \right)^{-\beta-\eps} 2^{n(\beta+\eps)} 
			\sum\limits_{k=1}^{2^n-2} {k^{-\beta} \left(2^n-1-k \right)^{-\eps}} \\
			&\leq C\left(t-s \right)^{-\beta-\eps} 2^{n(\beta+\eps)} 
			\sum\limits_{k=0}^{2^n-2} {\left(k+1 \right)^{-\beta} \left(2^n-1-k \right)^{-\eps}} \\
			& = C\left(t-s \right)^{-\beta-\eps} 2^{n(\beta+\eps)}
			\sum\limits_{j=1}^{2^n-1} {j^{-\eps} \left(2^n-j \right)^{-\beta}}.
			\end{align*}
			Using again Lemma \ref{lemma_appendix1} entails
			\begin{align*}
			J_{ts} &\leq C \left(t-s \right)^{-\beta-\eps} 2^{n(\beta+\eps)}
			\sum\limits_{j=1}^{2^n-1} {j^{-\eps} \left(2^n-j \right)^{-\beta}} \\
			&\leq C\left(t-s \right)^{-\beta-\eps} 2^{n(\beta+\eps)}
			\sum\limits_{j=0}^{2^n-1} {\left(j+1\right)^{-\eps} \left(2^n-j \right)^{-\beta}} \\
			&\leq C \left(t-s \right)^{-\beta-\eps} 2^{n(\beta+\eps)}
			\int\limits_{0}^{2^n} {q^{-\eps} (2^n-q)^{-\beta}}dq \\
			&=C\left(t-s \right)^{-\beta-\eps} 2^n B(1-\eps,1-\beta).
			\end{align*}
			Consequently, this results in 
			\begin{align}
			\left| \oN^n_{ts} - \oN^{n+1}_{ts}  \right|_{D_\eps} 
			\leq C c_1 (t-s)^{\alpha-\eps} 2^{-n(\alpha-\eps)} 
			+C c_2 \left(t-s \right)^{\rho-\eps-\beta}2^{-n(\rho-1)}.
			\end{align}
			Since the right hand side is again summable, we obtain that $\oN^n_{ts} \to \oN_{ts}$ in $D_\eps$ as $n\to\infty$. 
			Regarding that $\alpha\leq \rho-\beta$, leads to the estimate
			\begin{align*}
			\left|\oN_{ts} \right|_{D_\eps} \leq C \left(c_1+c_2 \right) \left(t-s \right)^{\alpha-\eps}.
			\end{align*}
			In order to obtain \eqref{property_deltaIXiest_eps} we only have to show that $\oN \equiv N$. We have
			\begin{align*}
			\left|N_{ts}-\oN_{ts} \right| = \lim\limits_{n \to \infty} \left|N^n_{ts}-\oN^n_{ts} \right|
			= \lim\limits_{n \to \infty}\left|\Xi_{t \ov_n} \right|
			\stackrel{\eqref{assumption_Xi}}{\leq} 
			c_1 \lim\limits_{n \to \infty} \left(t-\ov_{n+1} \right)^\alpha =0.
			\end{align*}
			Therefore we conclude that $\oN \equiv N$, which proves the statement. 
			\qed \\ \end{proof}
	
	Now we can apply these results to estimate $\delta\aI\Xi$. 
		\begin{corollary}\label{corollary_sewing_deltaestimate}
			Given the assumptions of Theorem \ref{lemma_sewing}. Then for all $\gamma < \alpha$ it holds
			\begin{align}
			\left| \left(\delta \aI\Xi\right)_{ts} \right| 
			\leq C (c_1+c_2) \left(t-s \right)^{\gamma} T^{\alpha-\gamma}.
			\end{align}
		\end{corollary}
		\begin{proof}
			By applying \eqref{property_deltaIXiest} and \eqref{property_deltaIXiest_eps} with $\eps=\gamma$ we get
			\begin{align*}
			\left| \left(\delta \aI\Xi\right)_{ts} \right| 
			\leq &\left| \left(\hat\delta \aI\Xi\right)_{ts} \right| + \left| \left(S(t-s)-\mbox{Id} \right) (\hat\delta\aI\Xi)_{s0} \right| \\
			\leq & C (c_1+c_2)\left(t-s \right)^{\alpha}  + C \left(t-s \right)^\gamma  \left| (\hat\delta\aI\Xi)_{s0}\right|_{D_\gamma} \\
			\leq & C (c_1+c_2)\left(t-s \right)^{\alpha}  + C (c_1+c_2)\left(t-s \right)^{\gamma} s^{\alpha-\gamma} \\
			\leq & C (c_1+c_2) \left(t-s \right)^{\gamma} T^{\alpha-\gamma}. 
			\end{align*}
		\qed  \end{proof}\\
		This immediately implies the next result.
		\begin{corollary}\label{corollary_sewing_norm}
			Given the assumptions of Theorem \ref{lemma_sewing}. Then for all $\gamma < \alpha$ it holds
			\begin{align}
			\left\|\aI\Xi \right\|_{\gamma} &\leq C(c_1+c_2) T^{\alpha-\gamma}, \\
			\left\|\aI\Xi \right\|_{\gamma,\gamma} &\leq C(c_1+c_2) T^{\alpha}.
			\end{align}
		\end{corollary}
		
		\begin{remark}\label{remark_sewing_linearity}
			Note that by construction $\aI$ is a linear mapping. More precisely, according to \cite[Section 4]{FritzHairer} or \cite[Section 3.3]{GubinelliLejayTindel} one can introduce the space $\hat{C}^{\alpha,\rho,\beta}(\Delta_T,W)$ of all elements $\Xi$ satisfying assumptions \eqref{assumption_Xi} and \eqref{assumption_deltaXi1}.\\
			  Then one can show that the mapping $\aI \colon \hat{C}^{\alpha,\rho,\beta}(\Delta_T,W) \to C^{\gamma,\gamma}([0,T],W) $, for $\gamma<\alpha$, is linear.\\
			  
			  In particular, considering $\Xi^1$, $\Xi^2$ with
			\begin{align*}
			\left|\Xi^1_{vu}-\Xi^2_{vu}\right| &\leq \widetilde{c}_1 \left(v-u \right)^{\alpha},\\
			\left|\left(\hat\delta_2 \Xi^1\right)_{vmu}- \left(\hat\delta_2 \Xi^2\right)_{vmu} \right| & \leq \widetilde{c}_2 u^{-\beta} \left(v-u \right)^{\rho},\quad
			\mbox{for all } 0 < u \leq m \leq v \leq T,
			\end{align*}
			yields
			\begin{align*}
			\left\|\aI\Xi^1 - \aI\Xi^2 \right\|_{\gamma,\gamma} 
			=  \left\|\aI(\Xi^1 - \Xi^2) \right\|_{\gamma,\gamma}
			\leq C(\widetilde{c}_1+\widetilde{c}_2) T^{\alpha}.
			\end{align*}
		\end{remark}

		\section{Construction of the supporting processes}\label{sec:supp:proc}
		We recall that $\omega \in C^{\alpha}\left(\left[0,T\right],V \right)$ stands for a fractional Brownian motion and $\omegaa \in C^{2\alpha}\left(\Delta_T,V \otimes V \right)$ for its L\'evy-area. 
		\begin{remark}
			In this setting $V\otimes V $ denotes the usual tensor product of Hilbert spaces. If one wishes to work in Banach spaces, then one should consider the projective tensor product, since the property
			\begin{align*}
			\mathcal{L}(V,\mathcal{L}(V,W))\hookrightarrow \mathcal{L} (V\otimes V, W)
			\end{align*} 
			is required.
			This is known to hold true, consult Theorem 2.9 in \cite{Ryan}.
			In the following, for notational simplicity we drop the tensor symbol. 
		\end{remark}
		
		Let $0 \leq s \leq \tau \leq t \leq T$ be fixed. As argued in Section \ref{heuristics}, in order to introduce an infinite-dimensional rough integral we first need to define the following processes and investigate their algebraic and analytic properties. Recall that throughout this section $K$ and $E$ should be interpreted as placeholder which stand for $G$, respectively $DG$. Keeping Section \ref{heuristics} in mind, we begin analyzing $a$, $c$ and $\omega^{S}$. More precisely,
		
		\begin{align}\label{definition_omegaS}
		&\omega^S_{ts} \colon \aL(V,W) \to W,
		&\omega^S_{ts}(K) := \int\limits_{s}^{t}{S(t-r) K }d\omega_r.
		\end{align}
		\begin{align}\label{definition_a}
		&a_{ts} \colon \aL(W \otimes V,W) \times W \to W, 
		&a_{ts}(E,x) := \int\limits_{s}^{t} {S(t-r) E S(r-s) x} d\omega_r.
		\end{align}
		\begin{align}\label{definition_c}
		&c_{ts} \colon \aL(W \otimes V,W) \times \aL(V,W) \to W,
		&c_{ts}(E,K) := \int\limits_{s}^{t}S(t-r) E  K (\delta\omega)_{rs} d\omega_r.
		\end{align}
		\begin{remark}
			Note that some of the processes above exist even if $\omega$ is not smooth, as shown in the following deliberations. However, at the very first sight, it is not at all clear why for instance \eqref{definition_c} is well-defined. 
			\end{remark}
		Similar to \cite{GarridoLuSchmalfuss2} we consider a smooth approximating sequence $\left(\omega^n,\omegaan \right)\to \left(\omega,\omegaa \right) $ in 
		$C^\alpha\left(\left[0,T\right],V \right) \times C^{2\alpha} \left(\Delta_T,V \otimes V \right)$, prove that the previous processes exist for this approximation terms and finally pass to the limit. Therefore we analyze
		\begin{align}
				&\omega^{S,n}_{ts}(K) := \int\limits_{s}^{t}{S(t-r) K }d\omega^{n}_r\\
			&a^{n}_{ts}(E,x) := \int\limits_{s}^{t} {S(t-r) E S(r-s) x} d\omega^{n}_r\\
		&c^{n}_{ts}(E,K) := \int\limits_{s}^{t}{S(t-r) E  K \left(\delta\omega^{n} \right)}_{rs} d\omega^{n}_r.
		\end{align}
		
			In the following we establish algebraic and analytic properties which will be employed further on. We begin with the algebraic structure.
		\begin{lemma}\label{algebraic}
			The properties
			\begin{align}
			&(\hat\delta_{2} \omega^{S,n})_{t \tau s}(K) = 0 \label{algebraic_omegaS} \\
			&(\hat\delta_{2} a^{n})_{t \tau s}(E,x) = a^{n}_{t \tau} (E,(S(\tau-s)-\mbox{Id} )x)\label{algebraic_a}\\
			&(\hat\delta_{2} c^{n})_{t \tau s}(E,K) = \omega^{S,n}_{t \tau}(EK (\delta \omega)_{\tau s}) \label{algebraic_c}
			\end{align} 
			are satisfied.
		\end{lemma}
		\begin{proof}
			One can easily verify that
			\begin{align*}
			(\hat\delta_{2} \omega^{S,n})_{t \tau s}(K) = \int\limits_{s}^{t} {S(t-r)K d\omega^n_r} - \int\limits_{\tau}^{t} {S(t-r)K d\omega^n_r} - \int\limits_{s}^{\tau} {S(t-r)K d\omega^n_r} = 0.
			\end{align*}
			Furthermore,
			\begin{align*}
			(\hat\delta_{2} a^n)_{t \tau s}(E,x) 
			&= \int\limits_{s}^{t}{S(t-r) E S(r-s) x d\omega^n_r} - \int\limits_{\tau}^{t}{S(t-r) E S(r-\tau) x d\omega^n_r}\\ 
			&- \int\limits_{s}^{\tau}{S(t-r) E S(r-s) x d\omega^n_r} \\
			&=  \int\limits_{\tau}^{t}{S(t-r) E \left(S(r-s)-S(r-\tau)\right) x d\omega^n_r} \\
			&=  \int\limits_{\tau}^{t}{S(t-r) E S(r-\tau) \left(S(\tau-s)-id \right) x d\omega^n_r} \\
			&=  a^n_{t \tau} (E,\left(S(\tau-s)-\mbox{Id}\right) x).
			\end{align*}
			Finally,
			\begin{align*}
			(\hat\delta_{2} c^n)_{t \tau s}(E,K) 
			&=  \int\limits_{s}^{t} {S(t-r) EK (\delta \omega^n)_{rs} d\omega^n_r} - \int\limits_{\tau}^{t} {S(t-r) EK (\delta \omega^n)_{r\tau} d\omega^n_r} \\
			& -\int\limits_{s}^{\tau} {S(t-r) EK (\delta \omega^n)_{rs} d\omega^n_r} \\
			&=  \int\limits_{\tau}^{t} {S(t-r) EK (\delta \omega^n)_{\tau s} d\omega^n_r} \\
			&=  \omega^{S,n}_{t\tau}(EK(\delta \omega^n)_{\tau s}).
			\end{align*}
			\qed \end{proof}\\
		
The analytic estimates are contained in the next result. Throughout this section $c_{S}$ stands for a constant which exclusively depends on the semigroup. 
		\begin{lemma}\label{analytic} For the processes $\omega^{S,n}_{ts}$, $a^{n}_{ts}$ and $c^{n}_{ts}$
		the following estimates hold true:
				\begin{align}
				&\left|\omega^{S,n}_{ts}(K) \right| \leq c_S \ltn \omega^n \rtn_{\alpha} \left| K \right| \left(t-s \right)^\alpha \label{estimate_omegaS}\\
				&\left| a^{n}_{ts}(E,x) \right| \leq c_S \ltn \omega^n \rtn_{\alpha} \left| E \right| |x|_{W} \left(t-s \right)^{\alpha}, ~~\mbox{for } x\in W\label{estimate_a}\\
				&\left|a^n_{ts}(E,x) - \omega^{S,n}_{ts}(Ex) \right| \leq c_S \ltn \omega^{n} \rtn_{\alpha} \left| E \right| \left|x \right|_{D_\beta} (t-s)^{\alpha+\beta}, ~~~\mbox{for } x\in D_{\beta}\label{estimate_a2}\\
				&\left|c^n_{ts}(E,K) \right| \leq c_S \left(\ltn \omega^n \rtn_{\alpha}+\left\|\omegaan \right\|_{2\alpha} \right) \left|E \right| \left| K \right| \left(t-s \right)^{2\alpha}.\label{estimate_c}
				\end{align}
		
			\end{lemma}
											
		\begin{proof}
			Using the integration by parts formula, see Theorem 3.5 in \cite{Pazy}, leads to 
			\begin{align*}
			\omega^{S,n}_{ts}(K) &= \int\limits_{s}^{t}{S(t-r)K}d\omega^{n}_r 
			= S(t-s) K (\delta \omega^n)_{ts} -A \int\limits_{s}^{t}{S(t-r)K(\delta \omega^n)_{tr}} dr\\
			a^n_{ts}(E,x) &= \int\limits_{s}^{t}{S(t-r)ES(r-s)x}d\omega^n_r
			= -\int\limits_{s}^{t}{\partial_r \omega^{S,n}_{tr}\left(ES(r-s)x\right) }dr \\
			&= \omega^{S,n}_{ts}(Ex) + \int\limits_{s}^{t}{\omega^{S,n}_{tr}\left(EAS(r-s)x\right)} dr \\
			c^n_{ts}(E,K) &= \int\limits_{s}^{t}{S(t-r) E K \left(\delta\omega^n\right)}_{rs} d\omega^n_r = \int\limits_{s}^{t} S(t-r) EK(\delta\omega^{n})_{ts}d\omega^{n}_{r} -\int\limits_{s}^{t}S(t-r)EK (\delta\omega^{n})_{tr}d\omega^{n}_{r} \\
			& = \omega^{S,n}_{ts}(EK(\delta\omega^n)_{ts}) - \int\limits_{s}^{t} S(t-r)EK~ d\omega^{(2),n}_{tr} \\
			&= \omega^{S,n}_{ts} (EK \left(\delta \omega^n \right)_{ts}) 
			- S(t-s) EK \omegaan_{ts}
			- \int\limits_{s}^{t} A S(t-r) EK \omegaan_{tr} dr.
			\end{align*}
			For a similar construction, see~\cite[Section 6.1]{DeyaGubinelliTindel}. Based on these identities we easily derive the analytic estimates as follows.\\
			
			A standard computation immediately entails
			\begin{align*}
			\left|\omega^{S,n}_{ts}(K) \right| 
			&\leq \left|S(t-s) K (\delta \omega^n)_{ts}\right| + \left|\int\limits_{s}^{t}{AS(t-r)K(\delta \omega^n)_{tr}} dr \right| \\
			&\leq c_S \left|K \right| \ltn \omega^n \rtn_{\alpha} \left(t-s \right)^{\alpha}.
			\end{align*}
			Recalling \eqref{algebraic_omegaS} we infer that
			\begin{align*}
			\left|a^n_{ts}(E,x) \right| 
			= &\left|\omega^{S,n}_{ts}(Ex) +\int\limits_{s}^{t}{\omega^{S,n}_{tr}\left(E A S(r-s) x \right)} dr\right| \\
			= &\left|\omega^{S,n}_{ts}(Ex) + \int\limits_{s}^{t}{\omega^{S,n}_{ts}\left(E A S(r-s) x \right)} dr 
			- \int\limits_{s}^{t}{S(t-r)\omega^{S,n}_{rs}\left(E A S(r-s) x \right)} dr \right| \\
			\leq &\left|\omega^{S,n}_{ts}(ES(t-s)x) \right| + \left| \int\limits_{s}^{t}{S(t-r)\omega^{S,n}_{rs}\left(E A S(r-s) x \right)} dr \right|\\
			\leq &c_S \left|E \right| \left|x \right|_{W} \ltn \omega^n\rtn_{\alpha} \left(t-s \right)^{\alpha} 
			+ c_S \left|E \right| \left|x \right|_{W} \ltn \omega^n\rtn_{\alpha} \int\limits_{s}^{t}{(r-s)^{\alpha-1}} dr \\
			\leq &c_S \left|E \right| \left|x \right|_{W} \ltn \omega^n\rtn_{\alpha} \left(t-s \right)^\alpha.
			\end{align*}
			For our aims it is also necessary to derive estimates for $x \in D_\beta$ with $0< \beta\leq 1$. In this situation we have
			\begin{align*}
			\left|a^n_{ts}(E,x)-\omega^{S,n}_{ts}(Ex) \right| 
			&= \left|\int\limits_{s}^{t}{\omega^{S,n}_{tr}\left(E A S(r-s) x \right)} dr\right| \\
			&\leq c_S \left|E \right| \left|x \right|_{D_\beta} \ltn \omega^n\rtn_{\alpha} 
			\int\limits_{s}^{t}{\left(t-r\right)^\alpha \left(r-s \right)^{\beta-1}} dr \\
			&= c_S \left|E \right| \left|x \right|_{D_\beta} \ltn \omega^n\rtn_{\alpha} \left(t-s \right)^{\alpha+\beta}.
			\end{align*}
			Furthermore, we obtain
			\begin{align*}
			\left|c^n_{ts}(E,K) \right| 
			\leq &\left|\omega^{S,n}_{ts} (EK \left(\delta \omega^n \right)_{ts})\right| 
			+ \left|S(t-s) EK \omegaan_{ts} \right|
			+\left| \int\limits_{s}^{t} A S(t-r) EK \omegaan_{tr} dr \right| \\
			\leq &c_S \left|E \right| \left|K \right| \ltn \omega^n\rtn_{\alpha}^2 \left(t-s \right)^{2\alpha}
			+ c_S \left|E \right| \left|K \right| \left\|\omegaan \right\|_{2\alpha} \left(t-s\right)^{2\alpha} \\
			+ &c_S \left|E \right| \left|K \right| \left\|\omegaan \right\|_{2\alpha} \int\limits_{s}^{t}{\left(t-r \right)^{2\alpha-1}}dr \\
			\leq & c_S \left|E \right| \left|K \right| \left(\ltn \omega^n\rtn_{\alpha}^2 + \left\|\omegaan\right\|_{2\alpha}\right) \left(t-s \right)^{2\alpha}.
			\end{align*}
		\qed \\ \end{proof}

	
		Consequently, keeping Lemma \ref{analytic} in mind we are justified to define the supporting processes via
		\begin{align}
		\label{definition_omegaS2}
		\omega^{S}_{ts}(K) &:= S(t-s) K (\delta \omega)_{ts} -A \int\limits_{s}^{t}{S(t-r)K(\delta \omega)_{tr}} dr \\
		\label{definition_a2}
		a_{ts}(E,x) &:= \omega^{S}_{ts}(Ex) + \int\limits_{s}^{t}{\omega^{S}_{tr}\left(E A S(r-s) x \right)} dr \\
		\label{definition_c2}
		c_{ts}(E,K)&:= \omega^{S}_{ts} (EK \left(\delta \omega \right)_{ts}) 
		- S(t-s) EK \omegaa_{ts}- \int\limits_{s}^{t} A S(t-r) EK \omegaa_{tr} dr.
		\end{align}

	\begin{lemma}\label{lemma:cont:dependence1} We have that
	\begin{align}
	&\omega^{S,n} \to \omega^S \text{ in } C^{\alpha}\left(\left[0,T\right],\aL(\aL(V,W),W) \right) \\
	&a^n \to a \text{ in } C^{\alpha}\left(\left[0,T\right],\aL(\aL(W \otimes V,W) \times W,W) \right) \label{convergence_a}\\
	&c^n \to c \text{ in } C^{2\alpha}\left(\left[0,T\right],\aL(\aL(W \otimes V,W) \times \aL(V,W),W) \right).
	\end{align}
\end{lemma}
\begin{proof}

	Similarly to the proof of Lemma \ref{analytic} we obtain
	\begin{align*}
	\left|\left(\omega^S-\omega^{S,n}\right)_{ts}(K) \right|
	&= \left| S(t-s) K (\delta (\omega-\omega^n))_{ts} -A \int\limits_{s}^{t}{S(t-r)K(\delta (\omega-\omega^n))_{tr}} dr\right| \\
	&\leq c_S \ltn \omega-\omega^n \rtn_{\alpha} \left|K \right| \left(t-s \right)^{\alpha},
	\end{align*}
	which shows that $\omega^{S,n} \to \omega^S$ in $C^{\alpha}\left(\left[0,T\right], \aL(\aL(V,W),W) \right)$. \\[1ex]
	The same deliberations as in the proof of (\ref{estimate_a2}) lead to	
	\begin{align*}
	\left|a_{ts}(E,x) - a^n_{ts}(E,x) \right|
	&\leq \left|\left(\omega^{S}-\omega^{S,n}\right)_{ts}(S(t-s)Ex) \right| \\
	&+ \left| \int\limits_{s}^{t}{S(t-r)\left(\omega^{S}-\omega^{S,n}\right)_{rs}\left(E A S(r-s) x \right)} dr \right|\\
	&\leq c_S \left|E \right| \left|x \right| \ltn \omega-\omega^n\rtn_{\alpha} \left(t-s \right)^\alpha.
	\end{align*}
	The last term yields
	\begin{align*}
	\left|c_{ts}(E,K) - c^n_{ts}(E,K) \right|
	\leq &\left|\omega^{S}_{ts} (EK \left(\delta \omega \right)_{ts}) - \omega_{ts}^{S,n} (EK \left(\delta \omega^n \right)_{ts})  \right|\\
	&+ \left| S(v-u) EK \omegaa_{ts} -S(v-u) EK \omegaan_{ts} \right| \\
	&+ \left|\int\limits_{s}^{t} A S(t-r) EK \omegaa_{tr} dr - \int\limits_{s}^{t} A S(t-r) EK \omegaan_{tr} dr\right|\\
	\leq & c_S \left(\ltn \omega \rtn_\alpha \ltn \omega-\omega^n\rtn_\alpha + \left\| \omegaa-\omegaan\right\|_{2\alpha} \right) \left|E \right| \left|K \right| \left(t-s \right)^{2\alpha}.
	\end{align*}
	
	\qed \end{proof}

		\begin{remark}
			Note that the algebraic and analytic properties proved in Lemmas \ref{algebraic} and \ref{analytic} remain valid.
		\end{remark}	

Furthermore, we observe.

\begin{corollary}\label{corollary_a_sum}
For an arbitrary partition $\aP=\aP(s,t)$ the following identity holds true
			\begin{align*}
			a_{ts}(E,x) = \sum\limits_{\left[u,v\right]\in \aP} S(t-v) a_{vu}(E,S(u-s)x).
			\end{align*}
\end{corollary}
\begin{proof}
			Using  \eqref{algebraic_a} and the bilinearity of $a$ we notice that
			\begin{align*}
			a_{ts}(E,x) 
			&= a_{t\tau}(E,x)+ S(t-\tau) a_{\tau s}(E,x) +  a_{t \tau} (E,(S(\tau-s)-\mbox{Id})x) \\
			&= a_{t\tau}(E,S(\tau-s)x)+ S(t-\tau) a_{\tau s}(E,x).
			\end{align*}
			Iterating this identity for any given partition $\aP(s,t)$ proves the claim.
\qed \end{proof}

	
		\begin{remark}
			Alternatively, these processes can also be defined using Theorem \ref{lemma_sewing}. For a better comprehension we illustrate this technique for $a$ and emphasize the fact that both approaches are equivalent.
		\end{remark}
		Heuristically, similar to Section \ref{heuristics}, we notice that for a smooth function $\omega$  we can approximate $a_{ts}$ as follows:
		\begin{align*}
		a_{ts}(E,x) := \int\limits_{s}^{t} {S(t-r) E S(r-s) x} d\omega_r
		&= \sum\limits_{\left[u,v\right]\in \aP} S(t-v) \int\limits_{u}^{v} {S(v-r) E S(r-s) x} d\omega_r \\
		&\approx \sum\limits_{\left[u,v\right]\in \aP} S(t-v) \int\limits_{u}^{v} {S(v-r) E S(u-s) x} d\omega_r \\
		&= \sum\limits_{\left[u,v\right]\in \aP} S(t-v)\, \omega^S_{vu}(E S(u-s) x).
		\end{align*}
		Keeping this in mind, the deliberations made in Section \ref{sectsl} lead to the following result.
		\begin{lemma}
			Let $0 \leq s \leq T$. For all $ s \leq \tau \leq t\leq T$ we define 
			\begin{align}
			\Xi^{(a),s}_{t \tau}(E,x):= \omega_{t\tau}^S(E S(\tau-s)x).
			\end{align}
			Then we have \begin{align}a_{ts}= \left(\hat\delta\mathcal{I}\Xi^{(a),s}\right)_{t s}.
			\end{align}
		\end{lemma}
		\begin{proof}
			In order to apply Theorem \ref{lemma_sewing} we have to analyze the term $\Xi ^{(a),s}_{vu}$. Therefore we estimate
			\begin{align*}
			\left| \Xi^{(a),s}_{vu}(E,x) \right| = \left| \omega_{vu}^S(E S(u-s)x) \right| \leq c_S \left| E \right| \left| x \right| \ltn \omega\rtn_{\alpha} \left(v-u \right)^\alpha
			\end{align*}
			and
			\begin{align*}
			\left|(\hat\delta_2 \Xi^{(a),s})_{vmu}\right| 
			= &\left|\omega_{vu}^S(E S(u-s)x) - \omega_{vm}^S(E S(m-s)x) - S(v-m) \omega_{mu}^S(E S(u-s)x)\right| \\
			= &\left|\omega_{vm}^S(E \left(S(u-s) - S(m-s) \right)x) \right|\\
			\leq & c_S \left| E \right| \left| x \right|_{D_{\beta}} \ltn \omega\rtn_\alpha \left(u-s \right)^{-\beta} \left(v-u \right)^{\alpha+\beta}. 
			\end{align*}
			Hence, Theorem \ref{lemma_sewing} yields the existence of $\mathcal{I}\Xi^{(a),s}$ and by Corollary \ref{corollary_sewing_integral} 
			\begin{align*}
			&(\hat\delta\mathcal{I}\Xi^{(a),s})_{t\tau} = \lim\limits_{\left|\aP\right| \to 0} \sum\limits_{\left[u,v\right]\in \aP( \tau, t)} S(t-v) \Xi^{(a),s}_{vu},
			\end{align*}
			which further implies
			\begin{align*}
			&(\hat\delta\mathcal{I}\Xi^{(a),s})_{ts}(E,x) = \lim\limits_{\left|\aP\right| \to 0} \sum\limits_{\left[u,v\right]\in \aP(s, t)} S(t-v) \omega_{vu}^S(E S(u-s)x).
			\end{align*}
		We	define 
			\begin{align*}
			\widetilde{a}_{ts} := (\hat\delta\mathcal{I}\Xi^{(a),s})_{ts}
			\end{align*}
			and show that $a=\widetilde a$. \\
By Corollary \ref{corollary_a_sum} we know that
\begin{align*}
a_{ts}(E,x) = \sum\limits_{\left[u,v\right]\in \aP(s,t)} S(t-v) a_{vu}(E,S(u-s)x).
\end{align*}

			Particularly, this also holds for the limit $\left|\aP \right| \to 0$. \\
			
			Regarding this, in order to prove the statement, i.e. that $a=\widetilde{a}$, we have to estimate the difference between $a_{vu}$ and $\omega^{S}_{vu}$. To this aim, we
			consider now a dyadic partition $\aP_n$ and have that
			\begin{align*}
			&\left|\sum\limits_{\left[u,v\right]\in \aP_n} \left( a_{vu}(E,S(u-s)x)- \omega^S_{vu}(ES(u-s)x) \right) \right|\\
			\leq &\sum\limits_{\left[u,v\right]\in \aP_n} \left| a_{vu}(E,S(u-s)x)- \omega^S_{vu}(ES(u-s)x) \right|.
			\end{align*}
			We apply \eqref{estimate_a2} for the first term with $\beta=0$ and for the other terms with $1-\alpha<\beta<1$, and obtain
			\begin{align*}
			&\sum\limits_{\left[u,v\right]\in \aP_n} \left| a_{vu}(E,S(u-s)x)- \omega^S_{vu}(ES(u-s)x) \right| \\
			\leq & c_s \left| E \right| \ltn \omega \rtn_\alpha \left| x \right| \frac{\left(t-s\right)^\alpha}{2^{n\alpha}} 
			+ \sum\limits_{\substack{\left[u,v\right]\in \aP_n \\ u \neq s}} 
			c_s \left| E \right| \ltn \omega \rtn_\alpha \left| x \right| 
			\left( u-s\right)^{-\beta} \frac{\left(t-s \right)^{\alpha+\beta}}{2^{n(\alpha+\beta)}} \\
			\leq & c_s \left| E \right| \ltn \omega \rtn_\alpha \left| x \right| 
			\left(\frac{\left(t-s\right)^\alpha}{2^{n\alpha}} 
			+ \frac{\left(t-s \right)^{\alpha+\beta-1}}{2^{n(\alpha+\beta-1)}} \int\limits_{s}^{t} \left(q-s \right)^{-\beta} dq \right)\\
			\leq & c_s \left| E \right| \ltn \omega \rtn_\alpha \left| x \right| \left(t-s \right)^{\alpha} 
			\left(2^{-n\alpha} + 2^{-n(\alpha+\beta-1)} \right) 
			\stackrel{n \to \infty}{\longrightarrow} 0. 
			\end{align*}
	This proves the statement.
	\qed \\ \end{proof}

		In order to complete the construction of the supporting processes, recall Section \ref{heuristics} we focus now on
		\begin{align}\label{definition_b}
		\begin{split}
		&b_{ts} \colon \aL(W \otimes V,W) \times \aL(V,W) \to W,\\
		&b_{ts}(E,K) := \int\limits_{s}^{t}{S(t-r) E \int\limits_{s}^{r}{S(r-q)K}d\omega_q} d\omega_r.
		\end{split}
		\end{align}
		\begin{remark}
			To our best knowledge it is not possible to define this process via integration by parts, see Remark 4.3 in \cite{DeyaGubinelliTindel}. We use Theorem \ref{lemma_sewing} to show that at least under some additional regularity assumption on $K$ (specified in Lemma \ref{reg_f}) it is possible to define $b(E,K)$.	
		\end{remark}
		
	Inspired by the definition of $a$ we follow the heuristic intuition given in Section \ref{heuristics}. We saw in~\eqref{hb}, that for a smooth $\omega$ we have	
	\begin{align*}
	b_{ts}(E,K)	\approx  \sum\limits_{[u,v]\in \aP} S(t-v) \left[\omega^S_{vu} (E \omega^S_{us}(K)) + c_{vu}(E,K) \right].
		\end{align*} 
		At the first sight the approximation above appears quite arbitrarily but we will rigorously show that this gives us the right approach to define $b$. \\
		
		As previously argued, we consider again a smooth approximating sequence $\left(\omega^n,\omegaan \right)$ of $ \left(\omega,\omegaa \right)$ and define 
		\begin{align}\label{bn}
		b^n_{ts}(E,K) := \int\limits_{s}^{t}{S(t-r) E \int\limits_{s}^{r}{S(r-q)K}d\omega^n_q} d\omega^n_r.
		\end{align}
		Furthermore, for all $0\leq s \leq \tau \leq t\leq T$ we introduce $$ \Xi^{(b),s}_{t \tau}(E,K):= \omega^S_{t\tau}(E \omega^S_{\tau s}(K)) + c_{t \tau}(E,K).$$
		Here the additional regularity assumption on $K$ plays a crucial role. This translates into the restriction on the diffusion coefficient $G$, recall assumption (G) in Section \ref{preliminaries}.
		
		\begin{lemma}\label{reg_f}
Let $K \in \aL(V,D_\beta)$ with $\alpha+2\beta>1$ and $\alpha>\beta$.
			Then there exists $$b_{ts}:= \left(\hat\delta\mathcal{I}\Xi^{(b),s}\right)_{t s}.$$ Moreover the following statements are valid
			\begin{enumerate}
				\item [(i)] analytic property:
				\begin{align}
				\left|b_{ts}(E,K) \right| \leq c_S \left| E \right| \left| K \right|_{D_{\beta}} \left(\ltn \omega\rtn^2_\alpha + \left\| \omegaa\right\|_{2\alpha} \right) \left(t-s \right)^{2\alpha}. \label{estimate_b}
				\end{align}
				\item [(ii)] continuous dependence on the paths of the noise: 
				\begin{align}
				b^n \to b \text{ in } C^{2\alpha}\left(\left[0,T\right],\aL(\aL(W \otimes V,W) \times \aL(V,D_\beta),W) \right).\label{convergence_b}
				\end{align}
				\item [(iii)] algebraic property:
				\begin{align}
				(\hat\delta_{2} b)_{t \tau s}(E,K) = a_{t\tau} (E,\omega^S_{\tau s}(K)).\label{algebraic_b}
				\end{align}
			\end{enumerate}
		\end{lemma}
		\begin{proof}
			As seen before, in order to apply Theorem \ref{lemma_sewing}, we have to analyze $\Xi^{(b),s}_{t\tau}$. Obviously,
			\begin{align*}
			\left| \Xi^{(b),s}_{t \tau}(E,K) \right| 
			\leq &\left|\omega^S_{t\tau}(E \omega^S_{\tau s}(K))\right| + \left|c_{t \tau}(E,K)\right| .
			\end{align*}
			Applying \eqref{estimate_omegaS} and \eqref{estimate_c} we have
			\begin{align*}
			&\left|\omega^S_{t\tau}(E \omega^S_{\tau s}(K))\right| + \left|c_{t \tau}(E,K)\right| \\
			&\leq c_S \ltn \omega \rtn_\alpha \left|E \omega^S_{\tau s}(K) \right| \left(t-\tau \right)^\alpha 
			+ c_S \left(\ltn \omega \rtn_{\alpha}+\left\|\omegaa \right\|_{2\alpha} \right) \left|E \right| \left| K \right| \left(t-\tau \right)^{2\alpha} \\
			&\leq c_S \ltn \omega \rtn^2_\alpha  \left|E \right| \left| K \right| \left(\tau-s \right)^{\alpha} \left(t-\tau \right)^\alpha 
			+ c_S \left(\ltn \omega \rtn_{\alpha}+\left\|\omegaa \right\|_{2\alpha} \right) \left|E \right| \left| K \right| \left(t-\tau \right)^{2\alpha}.
			\end{align*}
			Furthermore, we compute
			\begin{align*}
			\hat\delta_2\Xi^{(b),s}_{vmu}(E,K) 
			= & (\hat\delta_2\omega^S)_{vmu}(E \omega^S_{u s}(K)) + \omega^S_{vm} (E \left(\omega^S_{us}(K)-\omega^S_{ms}(K) \right)) 
			+ (\hat\delta_2 c)_{vmu}(E,K).
			\end{align*}
			By applying \eqref{algebraic_omegaS}, \eqref{algebraic_c} and \eqref{estimate_omegaS} we obtain
			\begin{align*}
			\hat\delta_2\Xi^{(b),s}_{vmu}(E,K) 
			= & \omega^S_{vm} (E \left(\omega^S_{us}(K)-\omega^S_{ms}(K) \right)) + \omega^S_{vm}(EK (\delta \omega)_{mu}) \\
			= & \omega^S_{vm} (E \left(K (\delta \omega)_{mu} - \omega^S_{ms}(K) + \omega^S_{us}(K) \right)).
			\end{align*}
			This further entails
			\begin{align}
			\left|\hat\delta_2\Xi^{(b),s}_{vmu}(E,K) \right| 
			\leq & \left| E \right| \ltn \omega \rtn_\alpha \left(v-m\right)^{\alpha} 
			\left|K (\delta \omega)_{mu} - \omega^S_{ms}(K) + \omega^S_{us}(K) \right|.
			\end{align}
			Consequently, we need appropriate estimates for the last term. We apply \eqref{algebraic_omegaS} and infer that
			\begin{align*}
			\left|K (\delta \omega)_{mu} - \omega^S_{ms}(K) + \omega^S_{us}(K) \right|
			&= \left|K (\delta \omega)_{mu} - \omega^S_{mu}(K) - \left(S(m-u)-\mbox{Id} \right) \omega^S_{us}(K)\right| \nonumber\\
			&\leq  \left|K (\delta \omega)_{mu} - \omega^S_{mu}(K) \right| + \left| \left(S(m-u)-\mbox{Id} \right) \omega^S_{us}(K)\right|.
			\end{align*}
			For the next steps the additional assumption $K:V \to D_\beta$ is required. Keeping this in mind
			and using \eqref{definition_omegaS2} we estimate the first term of the previous inequality as follows:
			\begin{align*}
			\left|K (\delta \omega)_{mu} - \omega^S_{mu}(K) \right|
			&\leq \left|\left(S(m-u)-\mbox{Id} \right) K (\delta \omega)_{mu} \right|
			+ \left|\int\limits_{u}^{m}{A S(m-r)K (\delta \omega)_{mr} } dr \right| \\
			& \leq c_{S} \left| K \right|_{D_\beta} \ltn \omega \rtn_\alpha \left(m-u \right)^{\alpha+\beta}.
			\end{align*}

			On the other hand, we have
			\begin{align*}
			\left| \left(S(m-u)-\mbox{Id} \right) \omega^S_{us}(K)\right| \leq c_S \left| \omega^S_{us}(K)\right|_{D_{2\beta}} \left(m-u \right)^{2\beta}.
			\end{align*}
			Applying again \eqref{definition_omegaS2} we derive 
			\begin{align*}
			\left| \omega^S_{us}(K) \right|_{D_{2\beta}} 
			\leq &\left|S(u-s) K (\delta\omega)_{us} \right|_{D_{2\beta}} 
			+ \left|A \int\limits_{s}^{u} S(u-r) K (\delta\omega)_{ur} dr \right|_{D_{2\beta}} \\
			\leq & c_S \left| K\right|_{D_\beta} \ltn \omega\rtn_{\alpha} \left(u-s \right)^{\alpha-\beta} 
			+ c_S \left| K\right|_{D_\beta} \ltn \omega\rtn_{\alpha} \int\limits_{s}^{u} \left(u-r \right)^{\alpha-\beta-1} dr \\
			\leq & c_S \left| K\right|_{D_\beta} \ltn \omega\rtn_{\alpha} \left(u-s \right)^{\alpha-\beta}.
			\end{align*}
			This finally leads to 
			\begin{align*}
			\left| \left(S(m-u)-\mbox{Id} \right) \omega^S_{us}(K)\right| \leq c_S \left| K \right|_{D_\beta} \ltn \omega \rtn_\alpha \left(u-s \right)^{\alpha-\beta} \left(m-u \right)^{2\beta}. 
			\end{align*}
			Putting all these together 
			we get
			\begin{align*}
			\left|K (\delta \omega)_{mu}- \omega^S_{ms}(K) + \omega^S_{us}(K) \right|
			\leq c_S T^{\alpha-\beta} \left| K\right|_{D_\beta} \ltn \omega\rtn_{\alpha}  \left(m-u\right)^{2\beta}.
			\end{align*}
			Consequently, regarding (5.26), 
			we obtain
			\begin{align*}
			\left|\hat\delta_{2} \Xi^{(b),s}_{vmu}(E,K) \right| 
			\leq c_S T^{\alpha-\beta} \left| E\right| \left| K\right|_{D_\beta} \ltn \omega\rtn^{2}_{\alpha}  \left(v-u \right)^{\alpha+2\beta}. 
			\end{align*}
			
			Theorem \ref{lemma_sewing} ensures the existence of $\aI\Xi^{(b),s}$ such that for all $s\leq \tau \leq t\leq T$ we have
			\begin{align*}
			(\hat\delta \aI\Xi^{(b),s})_{t\tau}(E,K) = \lim\limits_{\left|\aP\right| \to 0} \sum\limits_{\left[u,v\right]\in \aP(\tau, t)} S(t-v)\left[\omega^S_{vu}(E \omega^S_{u s}(K)) + c_{vu}(E,K) \right]
			\end{align*}
			and
			\begin{align*}
			\left|(\hat\delta \aI\Xi^{(b),s})_{t\tau}(E,K)\right| \leq c_{S}\left| E \right| \left| K \right|_{D_{\beta}} 
			\left(\ltn \omega\rtn^2_\alpha + \left\| \omegaa\right\|_{2\alpha} \right) 
			\left(t-\tau \right)^\alpha \left(\left(\tau-s \right)^\alpha + \left(t-\tau \right)^\alpha \right).
			\end{align*}
			In particular setting $\tau=s$ we can define $b_{ts}:=(\hat\delta \aI\Xi^{(b),s})_{ts}$ and infer from the previous estimate that
			\begin{align}
			\left|b_{ts}(E,K) \right| \leq c_S \left| E \right| \left| K \right|_{D_{\beta}} \left(\ltn \omega\rtn^2_\alpha + \left\| \omegaa\right\|_{2\alpha} \right) \left(t-s \right)^{2\alpha},	
			\end{align}
			which precisely gives us $(i)$.\\
			
			In order to prove $(iii)$ we compute as before 
			\begin{align*}
			(\hat\delta_2 b)_{t\tau s} (E,K)
			&= \lim\limits_{\left|\aP\right| \to 0} \sum\limits_{\left[u,v\right]\in \aP(s, t)} S(t-v)\left[\omega^S_{vu}(E \omega^S_{u s}(K)) + c_{vu}(E,K) \right] \\
			&- \lim\limits_{\left|\aP\right| \to 0} \sum\limits_{\left[u,v\right]\in \aP(\tau, t)} S(t-v)\left[\omega^S_{vu}(E \omega^S_{u \tau}(K)) + c_{vu}(E,K) \right]\\
			&- \lim\limits_{\left|\aP\right| \to 0} \sum\limits_{\left[u,v\right]\in \aP(s, \tau)} S(t-v)\left[\omega^S_{vu}(E \omega^S_{u s}(K)) + c_{vu}(E,K) \right] \\
			&= \lim\limits_{\left|\aP \right| \to 0} \sum\limits_{\left[u,v\right]\in \aP(\tau, t)} S(t-v)\omega^S_{vu}(E \left(\omega^S_{u s}-\omega^S_{u \tau}\right)(K)) \\
			&=  \lim\limits_{\left|\aP\right| \to 0} \sum\limits_{\left[u,v\right]\in \aP(\tau, t)} S(t-v)\omega^S_{vu}(E \omega^S_{\tau s}(K)) \\
			&=  a_{t\tau}(E,\omega^S_{\tau s}(K)).
			\end{align*}
			It only remains to show that assertion $(ii)$ holds true.  Regarding that $\omega^n$ and $\omegaan$ are smooth approximation terms, we are allowed to choose $\alpha>1/2$.\\
			
			We define $$\Xi^{(b),n,s}_{t \tau}(E,K):= \omega^{S,n}_{t\tau}(E \omega^{S,n}_{\tau s}(K)) + c^n_{t \tau}(E,K).$$ 
			Similar computations entail the existence of
			$$\widetilde{b}^{n}_{ts}(E,K):=
			(\hat\delta \aI\Xi^{(b),n,s})_{ts}(E,K).
			$$
			Moreover, using the same deliberations as above we obtain the analytic estimate
			\begin{align*}
			&\left|\widetilde{b}^n_{ts}(E,K) \right| \leq c_S \left| E \right| \left| K \right|_{D_{\beta}} \left(\ltn \omega^n\rtn^2_\alpha + \left\| \omegaan\right\|_{2\alpha} \right) \left(t-s \right)^{2\alpha}
			\end{align*}
			together with the algebraic structure
			\begin{align*}
			&(\hat\delta_2 \widetilde{b}^n)_{t\tau s} (E,K)=a^n_{t\tau}(E,\omega^{S,n}_{\tau s}(K)).
			\end{align*} 
			A straightforward computation for $b^n$ gives us 
			\begin{align*}
			(\hat\delta_2 b^n)_{t\tau s} (E,K) 
			= &\int\limits_{s}^{t} S(t-r) E \int\limits_{s}^{r} S(r-q) K d\omega^n_q d\omega^n_r \\
			&-\int\limits_{\tau}^{t} S(t-r) E \int\limits_{\tau}^{r} S(r-q) K d\omega^n_q d\omega^n_r \\
			&-\int\limits_{s}^{\tau} S(t-r) E \int\limits_{s}^{r} S(r-q) K d\omega^n_q d\omega^n_r \\
			= & \int\limits_{\tau}^{t} S(t-r) E \left(\int\limits_{s}^{r} S(r-q) K d\omega^n_q - \int\limits_{\tau}^{r} S(r-q) K d\omega^n_q \right)d\omega^n_r \\
			= & \int\limits_{\tau}^{t} S(t-r) E S(r-\tau) \int\limits_{s}^{\tau} S(\tau-q) K d\omega^n_q d\omega^n_r \\
			= & a^n_{t\tau}(E,\omega^{S,n}_{\tau s}(K)).
			\end{align*}
						Consequently, we obtain that $\hat\delta_2 (b^n-\widetilde{b}^n) \equiv 0$. Hence, for all $E,K$ by Lemma \ref{lemma_preliminaries_hatdelta} there exists $\kappa \in C(\left[0,T\right],W)$ such that
			\begin{align*}
			\kappa_0=0 \quad \mbox{ and } \quad (\hat\delta \kappa)_{ts} = (b^n-\widetilde{b}^n)_{ts}(E,K).
			\end{align*}
			We have 
			\begin{align*}
			\left|(\hat\delta \kappa)_{ts} \right| 
			\leq  \left|b^n_{ts}(E,K)\right|+ \left|\widetilde{b}^n_{ts}(E,K) \right| 
			\leq C(S,\omega^n,\omegaan,E,K) \left(t-s \right)^{2\alpha}.
			\end{align*}
			Since we assumed $\alpha>\frac{1}{2}$ Lemma \ref{lemma_hatdelta_uniqueness} yields $\kappa \equiv 0$ which implies $b^n=\widetilde{b}^n$. We know by Remark \ref{remark_sewing_linearity} that $\widetilde{b}^n$ converges to $b$ which proves the assertion.
		\qed \\ \end{proof}
		
		\begin{remark}\label{remark_supportingprocesses_shift}
			For our latter purpose it is important to consider integrals with respect to an appropriate time-shift of $\omega$. This is defined as usually by
			\begin{align}\label{shift5}
			\theta_\tau \omega_{t}: = \omega_{t+\tau}-\omega_{\tau},
			\end{align}
			for $\tau>0$.
			
		By a slight abuse of notation we introduce
		\begin{align*}
		\theta_\tau \omega^S_{ts}(K) := \int\limits_{s}^{t}{S(t-r) K }d\theta_\tau\omega_r
		\end{align*}
		and $\theta_\tau a$, $\theta_\tau b$ and $\theta_\tau c$ analogously. Since all supporting processes can be approximated by smooth functions it becomes clear that we have
\begin{align*}
\theta_\tau \omega^S_{ts} = \omega^S_{t+\tau, s+\tau}, \quad
\theta_\tau a_{ts} = a_{t+\tau, s+\tau},\quad
\theta_\tau b_{ts} = b_{t+\tau, s+\tau},\quad
\theta_\tau c_{ts} = c_{t+\tau, s+\tau}.
\end{align*}
Note that this is consistent with \eqref{shift:process}.
		\end{remark}

		\section{The fixed-point argument}\label{sectfp}
	Throughout this section we impose $\frac{1}{3}<\beta < \alpha\leq\frac{1}{2}$ such that $\alpha+2\beta>1$. Recall that all necessary assumptions on the coefficients and on the noise were stated in Section \ref{preliminaries}. \\
	
We now derive the existence of a solution for (\ref{eq1}) which is given by a pair $(y,z)$, as argued in Section \ref{heuristics}. Here	
 $(y_t)_{t \in \left[0,T\right]}$ stands for a $W$-valued path and $(z_{ts})_{(t,s)\in \Delta_T}$,
	$ z_{ts} \in \aL(\aL(W \otimes V,W),W)
	$ denotes the area term.\\

		
			Therefore, we are justified to introduce the Banach space
		\begin{align*}
		X_{\omega,T}:=\Big\{(y,z):  ~~& y \in C^{\beta,\beta}\left([0,T],W\right),\\
		 &z \in C^{\alpha}\left(\Delta_T,\aL(\aL(W \otimes V,W),W) \right) 
		\cap C^{\alpha+\beta,\beta}\left(\Delta_T,\aL(\aL(W \otimes V,W),W) \right),\\
		&	(\hat\delta_2 z)_{t \tau s} = \omega^S_{t\tau}(\cdot (\delta y)_{\tau s})
		 \Big\},
		\end{align*}
		endowed with norm
		\begin{align*}
		\left\|(y,z) \right\|_{X} := \left\|y \right\|_\infty + \ltn y \rtn_{\beta,\beta} 
		+ \left\| z \right\|_{\alpha} + \left\| z \right\|_{\alpha+\beta,\beta}.
		\end{align*}
		
\begin{remark}
	Note that the norm given above is equivalent to
	\begin{align*}
	\left\|y\right\|_{\infty} 
	+ \ltn y \rtn_{\beta,\beta}
	+ \sup\limits_{0 \leq s <  t \leq T} \frac{\left|z_{ts} \right|}{(t-s)^\alpha}
	+\sup\limits_{0 < s <  t \leq T} s^\beta \frac{\left|z_{ts} \right|}{(t-s)^{\alpha+\beta}},
	\end{align*} which essentially simplifies the computation. By a slight abuse of notation we use the same symbols.
\end{remark}		
		Using the same notations as in Section \ref{heuristics},
		we consider the map
		\begin{align*}
		\aM_T\colon X_{\omega,T} \to X_{\omega,T} \quad \quad \aM_T(y,z) = (\widetilde{y},\widetilde{z}),
		\end{align*}
		where
		\begin{align*}
		\yt_t &= S(t) \xi + \aI\Xi^{(y)}(y,z)_t \\
		\yb_t&= \yt_t - S(t) \xi = \aI\Xi^{(y)}(y,z)_t.
		\end{align*}
		Furthermore, for $E \in \aL(W \otimes V,W)$
		the second component of the solution is constituted by
		\begin{align*}
		\zt_{ts}(E) &= \left(\hat\delta \aI\Xi^{(z)}(y,\yt) \right)_{ts}(E) - \omega^S_{ts}(E \yt_s), \\
		&= \left(\hat\delta \aI\Xi^{(z)}(y,\yb) \right)_{ts}(E) - \omega^S_{ts}(E \yb_s) +a_{ts}(E,S(s)\xi) - \omega^S_{ts}(E S(s) \xi),\\
		\zb_{ts}(E) &= \left(\hat\delta \aI\Xi^{(z)}(y,\yb) \right)_{ts}(E) - \omega^S_{ts}(E \yb_s).
		\end{align*}
		
		Regarding this we define for $(u,v)\in \Delta_{T}$
		\begin{align*}
		\Xi^{(y)}_{vu}=\Xi^{(y)}(y,z)_{vu}
		&= \omega^S_{vu}(G(y_u)) + z_{vu}(DG(y_u)), \\
		\Xi^{(z)}_{vu}(E) = \Xi^{(z)}(y,\yb)_{vu}(E) 
		&= b_{vu}(E,G(y_u)) + a_{vu}(E,\yb_u).
		\end{align*}
		In order to show that $\mathcal{M}_{T}$ maps $X_{\omega, T}$ into itself and is a contraction we have to derive suitable a-priori estimates. We proceed step by step and split these results into several Lemmas.
		\begin{remark}
		Note that the universal constant $C$ occurring in the estimates below depends on $\ltn \omega\rtn _{\alpha}$,
		$\left\| \omega^{(2)}\right\|_{2\alpha}$, $\alpha$, $\beta$, $S(\cdot)$, $G$ uniformly with respect to $T$. We stress that this is independent of $\xi$.	
		\end{remark}
		 
		\begin{lemma}[Estimates of the $y$-integral]\label{lemma_y}
			For a pair $(y,z) \in X_{\omega,T}$ the following estimates are valid:
			\begin{align}
			\left|(\hat\delta \yb)_{ts} \right| &\leq C \left(1+\left\|(y,z) \right\|_{X}^2 \right) \left(t-s \right)^{\alpha},
			\label{estimate_y_hatdelta} \\
			\left| \yb_s \right|_{D_\beta} & \leq C \left(1+\left\|(y,z) \right\|_{X}^2 \right) s^{\alpha-\beta},
			\label{estimate_y_beta} \\
			\left\| \yb \right\|_{\beta,\beta} &\leq C \left(1+\left\|(y,z) \right\|_{X}^2 \right) T^{\alpha},
			\label{estimate_y_norm} \\
			\left| (\hat\delta \yb)_{ts} - \omega^S_{ts}(G(y_s)) \right| &\leq C \left(1+\left\|(y,z) \right\|_{X}^2 \right) s^{-\beta} \left(t-s \right)^{\alpha+\beta} \label{estimate_y_help1}.
			\end{align}
		\end{lemma}
		\begin{proof}
			Regarding the definition of $\Xi^{(y)}_{vu}$, the $\alpha$-H\"older continuity of $\omega$, the regularity of $G$ and the definition of the norm in $X_{\omega,T}$ we infer
			\begin{align*}
			\left|\Xi^{(y)}_{vu} \right| 
			&\leq \left| \omega^S_{vu}(G(y_u)) \right| + \left| z_{vu}(DG(y_u)) \right|\\
			&\leq C \ltn \omega \rtn_\alpha \left|G(y_u) \right| \left(v-u \right)^\alpha 
			+ C \left\| z \right\|_{\alpha} \left| DG(y_u) \right| \left(v-u \right)^{\alpha} \\
			&\leq C \left(\ltn \omega \rtn_\alpha \left(1+\left\|y \right\|_{\infty} \right) + \left\| z \right\|_{\alpha} \right) \left(v-u \right)^\alpha \\
			&\leq C (1+\left\|(y,z) \right\|_{X}) \left(v-u \right)^{\alpha}.
			\end{align*}
			Recalling (\ref{deltay})
			\begin{align*}
			(\hat\delta_2\Xi^{(y)})_{vmu} 
			=  \omega^S_{vm}(G(y_u)-G(y_m)+DG(y_u)(\delta y)_{mu}) + z_{vm}(DG(y_u)-DG(y_m)),
			\end{align*}
			together with the regularity assumptions on $y$ and $z$, further results in
			\begin{align*}
			\left| (\hat\delta_2\Xi^{(y)})_{vmu} \right|
			&\leq \left| \omega^S_{vm}(G(y_u)-G(y_m)+DG(y_u)(\delta y)_{mu}) \right|
			+ \left| z_{vm}(DG(y_u)-DG(y_m)) \right| \\
			&\leq C \ltn \omega \rtn_\alpha \left|G(y_u)-G(y_m) + DG(y_u)(\delta y)_{mu} \right| \left(v-m \right)^\alpha \\
			&+ C \left\| z \right\|_{\alpha+\beta,\beta} \left|DG(y_u)-DG(y_m) \right| m^{-\beta} \left(v-m \right)^{\alpha+\beta} \\
			&\leq C \ltn \omega \rtn_\alpha \left|y_u-y_m \right|^2 \left(v-m \right)^\alpha 
			+ C \left\| z \right\|_{\alpha+\beta,\beta} \left|y_u - y_m \right| m^{-\beta} \left(v-m \right)^{\alpha+\beta}.
			\end{align*}
			We observe that we have two different possibilities to estimate $\left|y_u - y_m \right|$. Obviously,
			\begin{align*}
			\left|y_u - y_m \right| \leq 2 \left\| y \right\|_\infty \quad \text{and} \quad
			\left|y_u - y_m \right| \leq \ltn y\rtn_{\beta,\beta} u^{-\beta} \left(m-u \right)^\beta.
			\end{align*}
			Therefore, on the one hand we get
			\begin{align*}
			\left| (\hat\delta_2\Xi^{(y)})_{vmu} \right|
			&\leq  C \ltn \omega \rtn_\alpha \ltn y\rtn_{\beta,\beta}^2 u^{-2\beta} \left(m-u \right)^{2\beta} \left(v-m \right)^\alpha \\
			&+C \left\| z \right\|_{\alpha+\beta,\beta} \ltn y\rtn_{\beta,\beta} u^{-\beta} \left(m-u \right)^{\beta} m^{-\beta} \left(v-m \right)^{\alpha+\beta} \\
			&\leq C \left(\ltn \omega \rtn_\alpha \ltn y\rtn_{\beta,\beta}^2 + \left\| z \right\|_{\alpha+\beta,\beta} \ltn y\rtn_{\beta,\beta} \right) u^{-2\beta} \left(v-u \right)^{\alpha+2\beta} \\
			&\leq C \left(1+\left\|(y,z) \right\|_{X}^2 \right) u^{-2\beta} \left(v-u \right)^{\alpha+2\beta}.
			\end{align*}
			By applying Theorem \ref{lemma_sewing} we can show the existence of $\aI\Xi^{(y)}=\aI\Xi^{(y)}(y,z)$ and obtain the estimate
			\begin{align*}
			\left|(\hat\delta \yb)_{ts} \right| = \left|(\hat\delta \aI\Xi^{(y)})_{ts} \right| \leq C\left(1+\left\|(y,z) \right\|_{X}^2 \right) \left(t-s \right)^{\alpha}. 
			\end{align*}
			 Corollary \ref{corollary_sewing_epsestimate} entails
			\begin{align*}
			\left|(\hat\delta \aI\Xi^{(y)})_{ts} \right|_{D_\beta} &\leq C\left(1+\left\|(y,z) \right\|_{X}^2 \right) \left(t-s \right)^{\alpha-\beta}, \end{align*}
			which implies that
			\begin{align*} \left|\yb_s \right|_{D_\beta} & \leq C\left(1+\left\|(y,z) \right\|_{X}^2 \right) s^{\alpha-\beta}.
			\end{align*}
			Furthermore, by Corollary \ref{corollary_sewing_norm} we obtain
			\begin{align}
			\left\|\yb \right\|_{\beta,\beta}
			&\leq C \left(1+\left\|(y,z) \right\|_{X}^2 \right)  T^\alpha. 
			\end{align}
			On the other hand we also have
			\begin{align*}
			\left| (\hat\delta_2\Xi^{(y)})_{vmu} \right| 
			&\leq  C \ltn \omega \rtn_\alpha \left\|y \right\|_{\infty} \ltn y\rtn_{\beta,\beta} u^{-\beta} \left(m-u \right)^{\beta} \left(v-m \right)^\alpha \\
			&+C \left\| z \right\|_{\alpha+\beta,\beta} \left\|y \right\|_{\infty} m^{-\beta} \left(v-m \right)^{\alpha+\beta} \\
			&\leq  C \left(\ltn \omega \rtn_\alpha \left\|y \right\|_{\infty} \ltn y\rtn_{\beta,\beta} + \left\| z \right\|_{\alpha+\beta,\beta} \left\|y \right\|_{\infty} \right) u^{-\beta} \left(v-u \right)^{\alpha+\beta} \\
			&\leq  C \left(1+\left\|(y,z) \right\|_{X}^2 \right) u^{-\beta} \left(v-u \right)^{\alpha+\beta}.
			\end{align*}
			Hence we can apply Corollary \ref{corollary_sewing_IXiestimate} and obtain
			\begin{align*}
			\left|(\hat\delta \aI\Xi^{(y)})_{ts} - \Xi^{(y)}_{ts} \right| &\leq C\left(1+\left\|(y,z) \right\|_{X}^2 \right) s^{-\beta} \left(t-s\right)^{\alpha+\beta},
			\end{align*}
			which leads to
			\begin{align*}
		  \left| (\hat\delta \yb)_{ts} - \omega^S_{ts}(G(y_s)) \right| &\leq \left|(\hat\delta \aI\Xi^{(y)})_{ts} - \Xi^{(y)}_{ts} \right| + \left|z_{ts}(DG(y_s)) \right| \\
			&\leq C \left(1+\left\|(y,z) \right\|_{X}^2 \right) s^{-\beta} \left(t-s \right)^{\alpha+\beta}.
			\end{align*}
		\qed \\ \end{proof}
		
		We now focus in deriving suitable estimates for $\overline{z}$.
		
		\begin{lemma}[Estimates of the $z$-integral]\label{lemma_z}
			Let $(y,z) \in X_{\omega,T}$. The following estimates are valid:
			\begin{align}
			\left|\zb_{ts}(E)\right| &\leq C |E| \left(1+\left\|(y,z) \right\|_{X}^2 \right) \left[\left(t-s \right)^{2\alpha} + s^{\alpha-\beta} \left(t-s \right)^{\alpha+\beta}\right],
			\label{estimate_z}\\
			\left\| \zb \right\|_{\alpha+\beta} &\leq C \left(1+\left\|(y,z) \right\|_{X}^2 \right) T^{\alpha-\beta}.
			\label{estimate_z_norm}
			\end{align}
		\end{lemma}
		\begin{proof}
			Applying Theorem \ref{lemma_sewing}
			we get
			\begin{align*}
			\left|\Xi^{(z)}_{vu}(E) \right| \leq \left| b_{vu}(E,G(y_u)) \right| + \left| a_{vu}(E,\yb_u) \right|.
			\end{align*}
			
			Furthermore, due to~\eqref{estimate_a} and~\eqref{estimate_b} together with the Lipschitz continuity of the mapping $G:W\to\mathcal{L}(V,D_{\beta})$, we infer that
			\begin{align*}
			\left|\Xi^{(z)}_{vu}(E) \right| 
			\leq & C \left|E \right| \left(1+\left\|y \right\|_{\infty} \right) \left(\ltn \omega\rtn_\alpha^2 + \left\|\omegaa \right\|_{2\alpha} \right) \left(v-u \right)^{2\alpha}
			+ C \left| E \right| \left\|\yb \right\|_{\infty} \ltn \omega\rtn_\alpha \left(v-u \right)^{\alpha} \\
			\leq & C \left|E \right|\left[ \left(1+\left\|y \right\|_{\infty} \right) \left(\ltn \omega\rtn_\alpha^2 + \left\|\omegaa \right\|_{2\alpha} \right) + \left\|\yb \right\|_{\infty} \ltn \omega\rtn_\alpha \right] \left(v-u \right)^{\alpha}. \\
			\end{align*}
			Using \eqref{estimate_y_norm} we obtain
			\begin{align*}
			\left|\Xi^{(z)}_{vu}(E) \right| 
			\leq &C \left|E \right| \left(1+\left\|(y,z) \right\|_{X}^2 \right) \left(v-u \right)^\alpha.
			\end{align*}
			Furthermore, we have by \eqref{algebraic_a} and \eqref{algebraic_b} that
			\begin{align*}
			(\hat\delta_2 \Xi^{(z),s})_{vmu} (E) 
			&= (\hat\delta_2 b)_{vmu}(E,G(y_u)) + b_{vm}(E,G(y_u)-G(y_m)) \\
			&+ (\hat\delta_2 a)_{vmu}(E,\yb_u) + a_{vm}(E,\yb_u-\yb_m) + \\
			&= a_{vm}(E,\omega^S_{mu}(G(y_u))) + b_{vm}(E,G(y_u)-G(y_m)) \\
			&+ a_{vm}(E,(S(m-u)-\mbox{Id})\yb_u) + a_{vm}(E,\yb_u-\yb_m) \\
			&=  a_{vm}(E,\omega^S_{mu}(G(y_u))-(\hat\delta \yb)_{mu}) + b_{vm}(E,G(y_u)-G(y_m)).
			\end{align*}
			This leads to
			\begin{align*}
			\left| (\hat\delta_2 \Xi^{(z),s})_{vmu} (E) \right|
			&\leq \left| a_{vm}(E,\omega^S_{mu}(G(y_u))-(\hat\delta \yb)_{mu}) \right| + \left| b_{vm}(E,G(y_u)-G(y_m)) \right| \\
			&\leq  C \ltn \omega\rtn_\alpha \left| E \right| \left| \omega^S_{mu}(G(y_u))-(\hat\delta \yb)_{mu} \right| \left(v-m \right)^\alpha \\
			&+ C \left(\ltn \omega\rtn_\alpha^2 + \left\|\omegaa \right\|_{2\alpha} \right) \left|E \right| \left|G(y_u)-G(y_m)\right| \left(v-m \right)^{2\alpha}.
			\end{align*}
			Applying \eqref{estimate_y_help1} entails
			\begin{align*}
			\left| (\hat\delta_2 \Xi^{(z),s})_{vmu} (E) \right|
			&\leq  C \left| E \right| \left(1+\left\|(y,z) \right\|_{X}^2 \right) u^{-\beta} \left(m-u \right)^{\alpha+\beta} \left(v-m \right)^\alpha \\
			&+ C \left|E \right| \ltn y\rtn_{\beta,\beta} u^{-\beta} \left(m-u \right)^{\beta} \left(v-m \right)^{2\alpha} \\
			&\leq C \left| E \right| \left(1+\left\|(y,z) \right\|_{X}^2 \right)
			u^{-\beta} \left(v-u \right)^{2\alpha+\beta}.\\
			\end{align*}
			Hence, we can apply again Theorem \ref{lemma_sewing}, however this does not give us the appropriate estimates. By a slightly different computation we obtain
			\begin{align*}
			\left| (\hat\delta_2 \Xi^{(z),s})_{vmu} (E) \right| 
			&\leq  C \ltn \omega\rtn_\alpha \left|E\right| \left| \left(\omega^S_{mu}(G(y_u)) \right|
			+\left|(\hat\delta \yb)_{mu} \right|\right) \left(v-m \right)^\alpha \\
			&+C \left(\ltn \omega\rtn_\alpha^2 + \left\|\omegaa \right\|_{2\alpha} \right) \left|E \right| \left|G(y_u)-G(y_m)\right| \left(v-m \right)^{2\alpha}.
			\end{align*}
			Now  \eqref{estimate_y_hatdelta} entails
			\begin{align*}
			\left| (\hat\delta_2 \Xi^{(z),s})_{vmu} (E) \right| 
			\leq C \left|E\right| \left(1+\left\|(y,z) \right\|_{X}^2 \right) \left(v-u \right)^{2\alpha}.
			\end{align*}
			Hence, by Corollary \ref{corollary_sewing_IXiestimate} we derive
			\begin{align*}
			\left|(\hat\delta \aI \Xi^{(z)})_{ts}(E)-\Xi^{(z)}_{ts}(E) \right|&\leq C \left|E\right| \left(1+\left\|(y,z) \right\|_{X}^2 \right)\left(t-s\right)^{2\alpha}.
			\end{align*}
			Furthermore, let us consider
			\begin{align*}
			\left|\Xi^{(z)}_{ts}(E)-\omega^S_{ts}(E \yb_s) \right|
			\leq & \left| b_{ts}(E,G(y_s)) \right| + \left| a_{ts}(E,\yb_s)-\omega^S_{ts}(E \yb_s) \right|.
			\end{align*}
			Applying \eqref{estimate_a2} and \eqref{estimate_b} entails
			\begin{align*}
			\left|\Xi^{(z)}_{ts}(E)-\omega^S_{ts}(E \yb_s) \right|
			&\leq C \left|E \right| \left(1+\left\|y \right\|_{\infty} \right) \left(\ltn \omega\rtn_\alpha^2 + \left\|\omegaa \right\|_{2\alpha} \right) \left(t-s \right)^{2\alpha}\\
			&+ C \left| E \right| \ltn \omega\rtn_\alpha \left|\yb_s \right|_{D_\beta} \left(t-s \right)^{\alpha+\beta}.
			\end{align*}
			By using \eqref{estimate_y_beta} we obtain
			\begin{align*}
			\left|\Xi^{(z)}_{ts}(E)-\omega^S_{ts}(E \yb_s) \right| 
			\leq C \left| E \right| \left(1+\left\|(y,z) \right\|_X^2 \right) 
			\left[s^{\alpha-\beta} \left(t-s\right)^{\alpha+\beta} + \left(t-s \right)^{2\alpha} \right].
			\end{align*}
			Summarizing, we conclude
			\begin{align*}
			\left|\zb_{ts}(E)\right| 
			&= \left|(\hat\delta \aI \Xi^{(z)})_{ts}(E)- \omega^S_{ts}(E \yb_s)\right| \\
			&\leq \left|(\hat\delta \aI \Xi^{(z)})_{ts}(E)-\Xi^{(z)}_{ts}(E) \right| 
			+ \left|\Xi^{(z)}_{ts}(E)-\omega^S_{ts}(E \yb_s) \right| \\ 
			&\leq C |E| \left(1+\left\|(y,z) \right\|_{X}^2 \right) \left[\left(t-s \right)^{2\alpha} + s^{\alpha-\beta} \left(t-s \right)^{\alpha+\beta}\right].
			\end{align*}
			Consequently, we get
			\begin{align*}
			\left\| \zb \right\|_{\alpha+\beta} &\leq C \left(1+\left\|(y,z) \right\|_{X}^2 \right) T^{\alpha-\beta}.
			\end{align*}
		\qed \\ \end{proof}
		After establishing suitable analytic properties we focus now on the algebraic setting.
		\begin{lemma}\label{lemma_algebraic}
		The	following algebraic property
			\begin{align*}
			(\hat\delta_2 \zb)_{vmu}(E) = \omega^S_{vm}(E (\delta \yb)_{mu})
			\end{align*}
			holds true.
		\end{lemma}
		\begin{proof}
			By \eqref{algebraic_omegaS} and Lemma \ref{lemma_preliminaries_hatdelta} \eqref{2.7i} we have
			\begin{align*}
			(\hat\delta_2 \zb)_{vmu}(E) 
			&= (\hat\delta_2 \hat\delta \aI \Xi^{(z)})_{vmu}(E) - (\hat\delta_2 \omega^S)_{vmu}(E \yb_u) + \omega^S_{vm}(E (\delta \yb)_{mu}) \\
			&= \omega^S_{vm}(E (\delta \yb)_{mu}).
			\end{align*}
		\qed \\ \end{proof}
We now have all the necessary ingredients to analyze the mapping $\mathcal{M}_{T}$ and proceed towards our fixed-point argument.
		\begin{theorem}\label{theorem_M}
			The mapping $\mathcal{M}_{T}$ maps $X_{\omega,T}$ into itself. Moreover, the estimate
			\begin{align*}
			\left\|\aM_T(y,z)\right\|_{X} \leq C \left(\left|\xi\right| + \left(1+\left\|(y,z) \right\|_X^2\right)T^\alpha \right) 
			\end{align*}
			holds true.  
		\end{theorem}
		\begin{proof}
			Recall
			\begin{align*}
			\yt_t = S(t)\xi + \yb_t.
			\end{align*}
			Hence, applying \eqref{estimate_y_norm} we derive
			\begin{align*}
			\left\|\yt \right\|_\infty 
			&\leq C \left(\left| \xi \right| + \left\| \yb \right\|_\infty \right) \\
			&\leq C \left(\left| \xi \right| + \left(1+\left\|(y,z)\right\|_X^2 \right) T^\alpha \right)
			\end{align*}
			and
			\begin{align*}
			\ltn \yt \rtn_{\beta,\beta} 
			&\leq \ltn S(\cdot) \xi \rtn_{\beta,\beta} + \ltn \yb \rtn_{\beta,\beta} \\
			&\leq C \left(\left| \xi \right| + \left(1+\left\|(y,z)\right\|_X^2 \right) T^\alpha \right).
			\end{align*}
			Moreover, we have
			\begin{align*}
			\zt_{ts}(E)= \zb_{ts}(E)+ a_{ts}(E,S(s)\xi) - \omega^S_{ts}(ES(s)\xi).
			\end{align*}
			On the one hand by \eqref{estimate_z_norm}, \eqref{estimate_a} and \eqref{estimate_omegaS} we get 
			\begin{align*}
			\left\|\zt \right\|_{\alpha} 
			&\leq \left\| \zb \right\|_{\alpha+\beta} T^{\beta} + C \left\|S(\cdot) \xi \right\|_{\infty} \\ 
			&\leq C \left(\left|\xi \right| + \left(1+\left\|(y,z) \right\|_{X}^2 \right)T^\alpha \right).
			\end{align*}
			On the other hand we apply \eqref{estimate_z_norm} and \eqref{estimate_a2} and infer that
			\begin{align*}
			\left\|\zt \right\|_{\alpha+\beta,\beta} 
			&\leq \left\| \zb \right\|_{\alpha+\beta} T^{\beta} + C \sup\limits_{0<s<T} s^\beta \left|S(s) \xi \right|_{D_\beta} \\
			&\leq C \left(\left|\xi \right| + \left(1+\left\|(y,z) \right\|_{X}^2 \right)T^\alpha \right).
			\end{align*}
			Summarizing we obtain the required estimate
			\begin{align*}
			\left\|\aM(y,z)\right\|_{X} \leq C\left(\left|\xi \right| + \left(1+\left\|(y,z) \right\|_X^2 \right) T^\alpha \right).
			\end{align*}
			The last step is to prove the corresponding algebraic relation. Due to Lemma \ref{lemma_algebraic} and of the algebraic relations \eqref{algebraic_omegaS} and \eqref{algebraic_a} we compute
			\begin{align*}
			(\hat\delta_2 \zt)_{t \tau s}(E) 
			= &(\hat\delta_2 \zb)_{t \tau s}(E) + (\hat\delta_2 a)_{t \tau s}(E,S(s) \xi) + a_{t \tau}(E,S(s) \xi - S(\tau) \xi)\\
			+& (\hat\delta_2 \omega^S)_{t \tau s}(E S(s)\xi) + \omega^S_{t \tau} (E (S(s)\xi - S(\tau) \xi)) \\
			= & \omega^S_{t \tau} (E (\yb_\tau-\yb_s)) + a_{t \tau} (E, S(\tau)\xi-S(s)\xi) + a_{t \tau}(E,S(s) \xi - S(\tau) \xi)\\
			+& \omega^S_{t \tau} (E (S(s)\xi - S(\tau) \xi)) \\
			= & \omega^S_{t \tau} (E (\yt_\tau-\yt_s)).
			\end{align*} 
		\qed \\ \end{proof}

		
		In order to show the existence of a unique local mild solution for (\ref{eq1}) by means of Banach's Fixed-Point Theorem we verify that $\mathcal{M}$ is a contraction. To this aim, analogously to Lemmas \ref{lemma_y} and \ref{lemma_z} we derive the necessary estimates.
		 
		\begin{lemma}[estimate for $\delta y$-integral]\label{lemma_y_contraction}
			Let $(y^1,z^1),\mbox{ and }(y^2,z^2) \in X_{\omega,T}$. Then we have 
			\begin{align}
			\left|(\hat\delta (\yb^1-\yb^2))_{ts} \right| 
			&\leq C \left(1+\left\| (y^1,z^1) \right\|_X^2 +\left\| (y^2,z^2) \right\|_X^2 \right)
			\left\| (y^1-y^2,z^1-z^2) \right\|_X \left(t-s \right)^{\alpha},
			\label{estimate_deltay_hatdelta} \\
			\left|\yb^1_s-\yb^2_s \right|_{D_\beta} 
			&\leq C \left(1+\left\| (y^1,z^1) \right\|_X^2 +\left\| (y^2,z^2) \right\|_X^2 \right)
			\left\| (y^1-y^2,z^1-z^2) \right\|_X s^{\alpha-\beta}
			\label{estimate_deltay_beta} \\
			\left\|\yb^1 - \yb^2 \right\|_{\beta,\beta} 
			&\leq C \left(1+\left\| (y^1,z^1) \right\|_X^2 +\left\| (y^2,z^2) \right\|_X^2 \right)
			\left\| (y^1-y^2,z^1-z^2) \right\|_X T^\alpha
			\label{estimate_deltay_norm}
			\end{align}
			as well as
			\begin{align}
			\begin{split}
			&\left| (\hat\delta (\yb^1-\yb^2))_{ts} - \omega^S_{ts}(G(y^1_s)-G(y^2_s)) \right| \\ 
			&\leq C \left(1+\left\| (y^1,z^1) \right\|_X^2 +\left\| (y^2,z^2) \right\|_X^2 \right)
			\left\| (y^1-y^2,z^1-z^2) \right\|_X s^{-\beta} \left(t-s\right)^{\alpha+\beta}.
			\end{split}
			\label{estimate_deltay_help1}
			\end{align}
			
		\end{lemma}
		\begin{proof}
			We have
			\begin{align*}
			\yb^1_t - \yb^2_t = \aI\Xi^{(y)}(y^1,z^1)_t - \aI\Xi^{(y)}(y^2,z^2)_t = \aI\left(\Xi^{(y)}(y^1,z^1)-\Xi^{(y)}(y^2,z^2)\right)_t.
			\end{align*}
			We make the same deliberations as in Lemma \ref{lemma_y} and use the assumptions on $G$. 	\begin{align*}
			&\left|\Xi^{(y)}(y^1,z^1)_{vu}-\Xi^{(y)}(y^2,z^2)_{vu}\right| \\
			&\leq \left| \omega^S_{vu}(G(y^1_u)-G(y^2_u)) \right| + \left| z^1_{vu}(DG(y^1_u)) - z^2_{vu}(DG(y^2_u)) \right|\\
			&\leq \left| \omega^S_{vu}(G(y^1_u)-G(y^2_u)) \right| + \left| z^1_{vu}\left(DG(y^1_u)-DG(y^2_u)\right)\right|
			+ \left| \left(z^1-z^2\right)_{vu}(DG(y^2_u)) \right|\\
			&\leq C \ltn \omega \rtn_\alpha \left\|y^1-y^2 \right\|_{\infty} \left(v-u \right)^\alpha 
			+ C \left\| z^1 \right\|_{\alpha} \left\|y^1-y^2 \right\|_{\infty} \left(v-u \right)^{\alpha}
			+ C \left\| z^1-z^2\right\|_{\alpha} \left(v-u \right)^{\alpha} \\
			&\leq  C \left(1+\left\|(y^1,z^1)\right\|_X\right) \left\|(y^1-y^2,z^1-z^2) \right\|_X.
			\end{align*}
			Furthermore,
			\begin{align*}
			&\left| (\hat\delta_2\Xi^{(y)}(y^1,z^1))_{vmu} - (\hat\delta_2\Xi^{(y)}(y^2,z^2))_{vmu} \right| \\
			&\leq \left| \omega^S_{vm}\left(G(y^1_u)-G(y^1_m)+DG(y^1_u)(\delta y^1)_{mu}-G(y^2_u)+G(y^2_m) - DG(y^2_u)(\delta y^2)_{mu}\right)\right| \\
			&+ \left| z^1_{vm}(DG(y^1_u)-DG(y^1_m)) - z^2_{vm}(DG(y^2_u)-DG(y^2_m)) \right|\\
			&\leq C \ltn \omega \rtn_{\alpha}\left| G(y^1_u)-G(y^1_m)+DG(y^1_u)(\delta y^1)_{mu}-G(y^2_u)+G(y^2_m) - DG(y^2_u)(\delta y^2)_{mu}\right|\left(v-u \right)^{\alpha} \\
			&+ \left| z^1_{vm}(DG(y^1_u)-DG(y^1_m)-DG(y^2_u)+DG(y^2_m)) \right| 
			+ \left|\left(z^1 - z^2\right)_{vm}(DG(y^2_u)-DG(y^2_m)) \right|.
			\end{align*}
			As in Lemma \ref{lemma_y} we have two possibilities to estimate these terms.\\[0.5ex]
			
			By \eqref{estimate_G1beta} and \eqref{estimate_G1infty} we infer
			\begin{align*}
			&\left| (\hat\delta_2\Xi^{(y)}(y^1,z^1))_{vmu} - (\hat\delta_2\Xi^{(y)}(y^2,z^2))_{vmu} \right| \\
			&\leq  C \ltn \omega \rtn_{\alpha} 
			\left[\left(\ltn y^1 \rtn_{\beta,\beta} + \ltn y^2\rtn_{\beta,\beta}\right) \ltn y^1-y^2\rtn_{\beta,\beta} 
			+ \ltn y^2\rtn_{\beta,\beta}^2 \left\|y^1-y^2 \right\|_{\infty} \right]
			u^{-2\beta} \left(v-u \right)^{\alpha+2\beta} \\
			&+ C \left\| z^1 \right\|_{\alpha+\beta,\beta} \left[\ltn y^1-y^2\rtn_{\beta,\beta} + \ltn y^2\rtn_{\beta,\beta} \left\|y^1-y^2 \right\|_{\infty} \right] u^{-2\beta}\left(v-u \right)^{\alpha+2\beta} \\
			&+ C\left\|z^1-z^2 \right\|_{\alpha+\beta,\beta} \ltn y^2\rtn_{\beta,\beta} u^{-2\beta}\left(v-u \right)^{\alpha+2\beta}\\
			&\leq C \left(1+\left\| (y^1,z^1) \right\|_X^2 +\left\| (y^2,z^2) \right\|_X^2 \right) \left\|(y^1-y^2,z^1-z^2)\right\|_X
			u^{-2\beta}\left(v-u \right)^{\alpha+2\beta}.
			\end{align*}
			Again, applying Theorem \ref{lemma_sewing} entails
			\begin{align*}
			\left|(\hat\delta (\yb^1-\yb^2))_{ts} \right| 
			&\leq C \left(1+\left\| (y^1,z^1) \right\|_X^2 +\left\| (y^2,z^2) \right\|_X^2 \right)
			\left\| (y^1-y^2,z^1-z^2) \right\|_X \left(t-s \right)^{\alpha}.
			\end{align*}
			By Corollary \ref{corollary_sewing_epsestimate} we obtain
			\begin{align*}
			\left|\yb^1_s-\yb^2_s \right|_{D_\beta} 
			&\leq C \left(1+\left\| (y^1,z^1) \right\|_X^2 +\left\| (y^2,z^2) \right\|_X^2 \right)
			\left\| (y^1-y^2,z^1-z^2) \right\|_X s^{\alpha-\beta},
			\end{align*}
			and with Corollary \ref{corollary_sewing_norm} we have
			\begin{align*}
			\left\|\yb^1 - \yb^2 \right\|_{\beta,\beta} 
			&\leq C \left(1+\left\| (y^1,z^1) \right\|_X^2 +\left\| (y^2,z^2) \right\|_X^2 \right)
			\left\| (y^1-y^2,z^1-z^2) \right\|_X T^\alpha.
			\end{align*}
			On the other hand using \eqref{estimate_G1infty} and \eqref{estimate_G2infty} we get
			\begin{align*}
			&\left| (\hat\delta_2\Xi^{(y)}(y^1,z^1))_{vmu} - (\hat\delta_2\Xi^{(y)}(y^2,z^2))_{vmu} \right| \\
			&\leq  C \ltn \omega \rtn_{\alpha} 
			\left(\ltn y^1\rtn_{\beta,\beta} + \ltn y^2\rtn_{\beta,\beta} + \ltn y^2\rtn_{\beta,\beta} \left\| y^2\right\|_{\infty} \right) \left\|y^1-y^2 \right\|_{\infty} u^{-\beta} \left(v-u \right)^{\alpha+\beta} \\
			& + C \left\| z^1\right\|_{\alpha+\beta,\beta}\left( \left\| y^1-y^2\right\|_{\infty} + \left\| y^2\right\|_{\infty} \left\|y^1-y^2 \right\|_{\infty} \right) u^{-\beta} \left(v-u \right)^{\alpha+\beta} \\
			& + C\left\|z^1-z^2 \right\|_{\alpha+\beta,\beta} \left\| y^2\right\|_{\infty} u^{-\beta}\left(v-u \right)^{\alpha+\beta}\\
			&\leq C \left(1+\left\| (y^1,z^1) \right\|_X^2 +\left\| (y^2,z^2) \right\|_X^2 \right) \left\|(y^1-y^2,z^1-z^2)\right\|_X
			u^{-\beta}\left(v-u \right)^{\alpha+\beta}.
			\end{align*}
			Hence we can apply Corollary \ref{corollary_sewing_IXiestimate} and obtain
			\begin{align*}
			&\left| (\hat\delta (\yb^1-\yb^2))_{ts} - \omega^S_{ts}(G(y^1_s)-G(y^2_s)) \right| \\ 
			&\leq C \left(1+\left\| (y^1,z^1) \right\|_X^2 +\left\| (y^2,z^2) \right\|_X^2 \right) \left\| (y^1-y^2,z^1-z^2) \right\|_X 
			s^{-\beta} \left(t-s\right)^{\alpha+\beta}.
			\end{align*}
		\qed \\ \end{proof}

		
		\begin{lemma}[estimate for $\delta z$-integral]\label{lemma_z_contraction}
			Let $(y^1,z^1),\mbox{ and }(y^2,z^2) \in X_{\omega,T}$. Then the following estimates are valid
			\begin{align}
			\begin{split}
			\left|(\zb^1_{ts}-\zb^2_{ts})(E)\right| 
			\leq C &|E| 
			\left(1+\left\| (y^1,z^1) \right\|_X^2 +\left\| (y^2,z^2) \right\|_X^2 \right) \left\| (y^1-y^2,z^1-z^2) \right\|_X \\
			&\left[\left(t-s \right)^{2\alpha} + s^{\alpha-\beta} \left(t-s \right)^{\alpha+\beta}\right],
			\end{split} \label{estimate_deltaz}\\
			\left\| \zb^1-\zb^2 \right\|_{\alpha+\beta} 
			\leq C &\left(1+\left\| (y^1,z^1) \right\|_X^2 +\left\| (y^2,z^2) \right\|_X^2 \right) \left\| (y^1-y^2,z^1-z^2) \right\|_X T^{\alpha-\beta}.
			\label{estimate_deltaz_norm}
			\end{align}
		\end{lemma}
		\begin{proof}
			Recall that
			\begin{align*}
			(\zb^1_{ts}-\zb^2_{ts})(E) 
			&= (\hat\delta\aI\Xi^{(y)}(y^1,\yb^1))_{ts}(E) -\omega^S_{ts}(E \yb^1_s) 
			- (\hat\delta\aI\Xi^{(z)}(y^2,\yb^2))_{ts}(E) + \omega^S_{ts}(E \yb^2_s) \\
			&=(\hat\delta\aI[\Xi^{(z)}(y^1,\yb^1)-\Xi^{(z)}(y^2,\yb^2)])_{ts}(E) 
			- \left(\omega^S_{ts}(E \yb^1_s)- \omega^S_{ts}(E \yb^2_s) \right). 
			\end{align*}
			Building the difference of $\Xi^{(z)}$ for $(y^{1}, \overline{y}^{1})$  and $(y^{2},\overline{y}^{2})$ entails
			\begin{align*}
			&\left|\Xi^{(z)}(y^1,\yb^1)_{vu}(E)-\Xi^{(z)}(y^2,\yb^2)_{vu}(E) \right| \\
			&\leq \left| b_{vu}(E,G(y^1_u)-G(y^2_u)) \right| + \left| a_{vu}(E,\yb^1_u-\yb^2_u) \right| \\
			&\leq C \left|E \right| \left(\ltn \omega\rtn_\alpha^2 
			+ \left\|\omegaa \right\|_{2\alpha} \right) \left\|y^1-y^2 \right\|_{\infty} \left(v-u \right)^{2\alpha}
			+ C \left| E \right| \ltn \omega\rtn_\alpha \left\|\yb^1-\yb^2 \right\|_{\infty} \left(v-u \right)^{\alpha}.
			\end{align*}
			By \eqref{estimate_deltay_norm} we get
			\begin{align*}
			&\left|\Xi^{(z)}(y^1,\yb^1)_{vu}(E)-\Xi^{(z)}(y^2,\yb^2)_{vu}(E) \right|\\
			&\leq C \left|E \right| \left(1+\left\| (y^1,z^1) \right\|_X^2 +\left\| (y^2,z^2) \right\|_X^2 \right) 
			\left\| (y^1-y^2,z^1-z^2) \right\|_X
			\left(v-u \right)^{\alpha}.
			\end{align*}
			Furthermore,
			\begin{align*}
			&\left|(\hat\delta_2 \Xi^{(z)}(y^1,\yb^1))_{vmu} (E) - (\hat\delta_2 \Xi^{(z)}(y^2,\yb^2))_{vmu} (E) \right|\\
			&\leq \left| a_{vm}(E,\omega^S_{mu}(G(y^1_u)-G(y^2_u))-(\hat\delta \yb^1)_{mu}+(\hat\delta \yb^2)_{mu}) \right| \\
			&+ \left| b_{vm}(E,G(y^1_u)-G(y^1_m)-G(y^2_u)+G(y^2_m)) \right|\\
			&\leq C \ltn \omega\rtn_\alpha \left| E \right| 
			\left| \omega^S_{mu}(G(y^1_u)-G(y^2_u))-(\hat\delta (\yb^1-\yb^2))_{mu}\right| \left(v-u \right)^\alpha \\
			&+ C \left(\ltn \omega\rtn_\alpha^2 + \left\|\omegaa \right\|_{2\alpha} \right) \left|E \right| \left| G(y^1_u)-G(y^1_m)-G(y^2_u)+G(y^2_m) \right| \left(v-u \right)^{2\alpha}.
			\end{align*}
			By applying \eqref{estimate_deltay_help1} and Lemma \ref{estimate_G1beta} we derive 
			\begin{align*}
			&\left|(\hat\delta_2 \Xi^{(z)}(y^1,\yb^1))_{vmu} (E) - (\hat\delta_2 \Xi^{(z)}(y^2,\yb^2))_{vmu} (E) \right|\\
			&\leq C \left| E \right|
			\left(1+\left\| (y^1,z^1) \right\|_X^2 +\left\| (y^2,z^2) \right\|_X^2 \right) \left\| (y^1-y^2,z^1-z^2) \right\|_X 
			u^{-\beta} \left(v-u \right)^{2\alpha+\beta}. 
			\end{align*}
			On the other hand we can estimate
			\begin{align*}
			&\left|(\hat\delta_2 \Xi^{(z)}(y^1,\yb^1))_{vmu} (E) - (\hat\delta_2 \Xi^{(z)}(y^2,\yb^2))_{vmu} (E) \right|\\
			&\leq C \left| E \right| 
			\left(\left| \omega^S_{mu}(G(y^1_u)-G(y^2_u)) \right| +\left|(\hat\delta (\yb^1-\yb^2))_{mu}\right| \right)
			\left(v-u \right)^\alpha \\
			&+ C \left|E \right| \left(\left| G(y^1_u) -G(y^2_u) \right| + \left|G(y^1_m) - G(y^2_m) \right| \right) 
			\left(v-u \right)^{2\alpha}.
			\end{align*}
			By applying \eqref{estimate_deltay_hatdelta} we obtain
			\begin{align*}
			&\left|(\hat\delta_2 \Xi^{(z)}(y^1,\yb^1))_{vmu} (E) - (\hat\delta_2 \Xi^{(z)}(y^2,\yb^2))_{vmu} (E) \right|\\
			\leq &C \left| E \right| 
			\left(1+\left\| (y^1,z^1) \right\|_X^2 +\left\| (y^2,z^2) \right\|_X^2 \right) \left\| (y^1-y^2,z^1-z^2) \right\|_X
			\left(v-u \right)^{2\alpha}.
			\end{align*}
			Again with Corollary \ref{corollary_sewing_IXiestimate} we conclude
			\begin{align*}
			&\left|(\hat\delta\aI[\Xi^{(z)}(y^1,\yb^1)-\Xi^{(z)}(y^2,\yb^2)])_{ts}(E)
			-(\Xi^{(z)}(y^1,\yb^1)-\Xi^{(z)}(y^2,\yb^2))_{ts}(E) \right|\\
			&\leq C \left|E\right| 
			\left(1+\left\| (y^1,z^1) \right\|_X^2 +\left\| (y^2,z^2) \right\|_X^2 \right) \left\| (y^1-y^2,z^1-z^2) \right\|_X
			\left(t-s\right)^{2\alpha}.
			\end{align*}
			Furthermore we have
			\begin{align*}
			&\left|(\Xi^{(z)}(y^1,\yb^1)-\Xi^{(z)}(y^2,\yb^2))_{ts}(E)-\omega^S_{ts}(E (\yb^1_s-\yb^2_s)) \right|\\
			&\leq \left| b_{ts}(E,G(y^1_s)-G(y^2_s)) \right| + \left| a_{ts}(E,\yb^1_s-\yb^2_s)-\omega^S_{ts}(E (\yb^1_s-\yb^2_s)) \right|.
			\end{align*}
			Then \eqref{estimate_a2} yields
			\begin{align*}
			&\left|(\Xi^{(z)}(y^1,\yb^1)-\Xi^{(z)}(y^2,\yb^2))_{ts}(E)-\omega^S_{ts}(E (\yb^1_s-\yb^2_s)) \right|\\
			&\leq C \left|E \right| \left\|y^1-y^2 \right\|_{\infty}  
			\left(\ltn \omega\rtn_\alpha^2 + \left\|\omegaa \right\|_{2\alpha} \right) \left(t-s \right)^{2\alpha}\\
			&+ C \left| E \right| \ltn \omega\rtn_\alpha \left|\yb^1_s-\yb^2_s \right|_{D_\beta} \left(t-s \right)^{\alpha+\beta}.
			\end{align*}
			By applying \eqref{estimate_deltay_beta} we see
			\begin{align*}
			&\left|(\Xi^{(z)}(y^1,\yb^1)-\Xi^{(z)}(y^2,\yb^2))_{ts}(E)-\omega^S_{ts}(E (\yb^1_s-\yb^2_s)) \right|\\
			&\leq C \left| E \right|
			\left(1+\left\| (y^1,z^1) \right\|_X^2 +\left\| (y^2,z^2) \right\|_X^2 \right) \left\| (y^1-y^2,z^1-z^2) \right\|_X
			\left[s^{\alpha-\beta} \left(t-s\right)^{\alpha+\beta} + \left(t-s \right)^{2\alpha} \right].
			\end{align*}
			Finally, we derive
			\begin{align*}
			\left|(\zb^1_{ts}-\zb^2_{ts})(E)\right| 
			&\leq \left|(\hat\delta\aI[\Xi^{(z)}(y^1,\yb^1)-\Xi^{(z)}(y^2,\yb^2)])_{ts}(E) 
			-(\Xi^{(z)}(y^1,\yb^1)-\Xi^{(z)}(y^2,\yb^2))_{ts}(E) \right| \\
			&+ \left|(\Xi^{(z)}(y^1,\yb^1)-\Xi^{(z)}(y^2,\yb^2))_{ts}(E)-\omega^S_{ts}(E (\yb^1_s-\yb^2_s)) \right|\\
			&\leq C |E| 
			\left(1+\left\| (y^1,z^1) \right\|_X^2 +\left\| (y^2,z^2) \right\|_X^2 \right) \left\| (y^1-y^2,z^1-z^2) \right\|_X \\
			&\phantom{+}\left[\left(t-s \right)^{2\alpha} + s^{\alpha-\beta} \left(t-s \right)^{\alpha+\beta}\right].
			\end{align*}
			Consequently, we get
			\begin{align*}
			\left\| \zb^1-\zb^2 \right\|_{\alpha+\beta} 
			\leq C &\left(1+\left\| (y^1,z^1) \right\|_X^2 +\left\| (y^2,z^2) \right\|_X^2 \right) \left\| (y^1-y^2,z^1-z^2) \right\|_X T^{\alpha-\beta}.
			\end{align*}
		\qed \\ \end{proof}
	
		Now, putting all these results together, we can state the main theorem of this section.
		\begin{theorem}
			Let $r>0$ with $|\xi|\leq r$. 
			Then there exist $R=R(r,\omega)$ and  $T=T(\omega,R)>0$ such that the mapping
			$\aM_{T,R}:=\aM_T|_{B_{X}(0,R)} \colon B_{X}(0,R) \to B_{X}(0,R) $ is a contraction and possesses a unique fixed point.
		\end{theorem}
		\begin{proof}
			By Theorem \ref{theorem_M} we know that $\mathcal{M}_{T}$ maps $X_{\omega, T}$ into itself and
			\begin{align*}
			\left\|\aM_T(y,z)\right\|_{X} \leq C \left(\left|\xi\right| + \left(1+\left\|(y,z) \right\|_X^2\right)T^\alpha \right).
			\end{align*}
			Setting $R:=2Cr$, we have
			\begin{align*}
			\left\|\aM_{T,R}(y,z)\right\|_{X} \leq \frac{R}{2} + C \left(1+R^2\right) T^\alpha.
			\end{align*}
			Hence we can choose $T$ small enough and obtain
			\begin{align*}
			\left\|\aM_{T,R}(y,z)\right\|_{X} \leq R,
			\end{align*}
			which means that $\aM_{T,R}$ maps $B_{X}(0,R)$ into itself.\\
			
			Since $\yt^1-\yt^2=\yb^1-\yb^2$ and $\zt^1-\zt^2 = \zb^1-\zb^2$, applying Lemmas \ref{lemma_y_contraction} and \ref{lemma_z_contraction} we derive
			\begin{align*}
			\left\|\aM_{T}(y^1,z^1)- \aM_{T}(y^2,z^2)\right\|_{X}
			\leq C \left(1+\left\| (y^1,z^1) \right\|_X^2 +\left\| (y^2,z^2) \right\|_X^2 \right)
			\left\| (y^1-y^2,z^1-z^2) \right\|_X T^\alpha.
			\end{align*}
			Hence,
			\begin{align*}
			\left\|\aM_{T,R}(y^1,z^1)- \aM_{T,R}(y^2,z^2)\right\|_{X}
			\leq C \left(1+ 2 R^2\right)
			\left\| (y^1-y^2,z^1-z^2) \right\|_X T^\alpha.
			\end{align*}
			Again, we can choose $T$ small enough such that
			\begin{align*}
			\left\|\aM_{T,R}(y^1,z^1)- \aM_{T,R}(y^2,z^2)\right\|_{X}
			\leq \frac{1}{2} \left\| (y^1-y^2,z^1-z^2) \right\|_X,
			\end{align*}
			which proves the contraction property of $\mathcal{M}_{T,R}$.
			Consequently, Banach's fixed-point Theorem entails that $\mathcal{M}_{T,R}$ has a unique fixed point in $B_{X}(0,R)$.
		\qed \end{proof}\\

	We showed the existence of a unique local solution in an appropriate ball. This means that another local mild solution for (\ref{eq1}) could exist outside this ball. The following theorem excludes this case. In order to prove this statement we need some additional results. 
	For the existence proof we considered a fixed $\omega \in C^\alpha$ and a fixed initial condition $\xi \in W$.
     From now on we want to investigate the dependence of the solution on these parameters. Therefore we introduce the following notation 
		\begin{align*}
		\aM_{T,\omega ,\xi} &\colon X_{\omega,T} \to X_{\omega,T} \\
		\aM_{T,\omega, \xi}(y,z)^{(1)}_t &= S(t) \xi + \aI\Xi_\omega^{(y)}(y,z)_t \\
		\aM_{T,\omega, \xi}(y,z)^{(2)}_{ts}(E) &= \left(\hat\delta_1 \aI\Xi_\omega^{(z)}(y,y) \right)_{ts}(E) - \omega^S_{ts}(E y_s).
		\end{align*}
For notational simplicity we further set for $\tau>0$
\begin{align}
\theta_{\tau}y_t &:= y_{t+\tau} \\
\theta_{\tau}z_{ts}&:= z_{t+\tau,s+\tau}.
\end{align}
The next assertion provides the connection between \eqref{shift:process}, \eqref{shift5} and the expressions introduced above.
\begin{lemma}\label{lemma_shift_Xi}
	Let $\tau>0$. Then the following identities hold true
	\begin{align*}
	\theta_{\tau} \Xi_{\omega}^{(y)}(y,z)
	&=\Xi_{\theta_{\tau}\omega}^{(y)}(\theta_{\tau }y,\theta_{\tau }z), \\
		\theta_{\tau} \Xi_{\omega}^{(z)}(y,y)
	&=\Xi_{\theta_{\tau}\omega}^{(z)}(\theta_{\tau} y,\theta_{\tau}y). 
	\end{align*}
\end{lemma}
\begin{proof}
	The proof is a based on Remark \ref{remark_supportingprocesses_shift}. We directly have that
\begin{align*}
\theta_{\tau} \Xi_{\omega}^{(y)}(y,z)_{vu}
&= \Xi_{\omega}^{(y)}(y,z)_{v+\tau,u+\tau}
= \omega^S_{v+\tau,u+\tau}(G(y_{u+\tau})) + z_{v+\tau,u+\tau}(DG(y_{u+\tau}))\\
&=\theta_{\tau}\omega^S_{vu}(G(\theta_{\tau} y_{u})) + \theta_{\tau} z_{vu}(DG(\theta_{\tau}y_{u}))
= \Xi_{\theta_{\tau}\omega}^{(y)}(\theta_{\tau}y,\theta_{\tau}z)_{vu},
\end{align*}
as well as
\begin{align*}
\theta_{\tau} \Xi_{\omega}^{(z)}(y,y)_{vu}(E)
&= \Xi_{\omega}^{(z)}(y,y)_{v+\tau,u+\tau}(E)
= b_{v+\tau,u+\tau}(E,G(y_{u+\tau})) + a_{v+\tau,u+\tau}(E,y_{u+\tau}) \\&= \theta_{\tau}b_{vu}(E,G(\theta_{\tau}y_{u})) + \theta_{\tau}a_{vu}(E,\theta_{\tau}y_{u})
= \Xi_{\theta_{\tau}\omega}^{(z)}(\theta_{\tau}y,\theta_{\tau}y)_{vu}(E).
\end{align*}
\qed	
\end{proof}\\
The first step in establishing the uniqueness of the local solution is contained in the next result. Note that this is referred to as the cocycle property in the theory of random dynamical systems, see \cite{Arnold}. This will be dealt with in a forthcoming work when we investigate global solutions for rough evolution equations. 
			\begin{lemma}\label{lemma_localsolution_shift1} 	Let $T>0$ and $(y,z) \in X_{\omega,T}$ be a fixed-point of $\aM_{T,\omega ,\xi}$. Then for any $\tau \in [0,T)$ there exists a fixed-point of $\aM_{T-\tau,\theta_{\tau}\omega ,y_{\tau}}$ given by
			$(\theta_{\tau}y, \theta_{\tau}z)$.
		\end{lemma}
		\begin{proof}
		This is a direct consequence of Corollary \ref{corollary_sewing_shift} and Lemma \ref{lemma_shift_Xi}. By standard computations we get
			\begin{align*}
			\theta_{\tau}y_t &= y_{t+\tau} = S(t+\tau) \xi + \aI\Xi_\omega^{(y)}(y,z)_{t+\tau} \\
			&= S(t) y_{\tau} + (\hat\delta \aI\Xi_\omega^{(y)}(y,z))_{t+\tau,\tau} \\
			&= S(t) y_{\tau} +  \aI\Xi_{\theta_{\tau}\omega}^{(y)}(\theta_{\tau}y,\theta_{\tau}z)_{t}.
			\end{align*}
			Furthermore,
			\begin{align*}
			\theta_{\tau}z_{ts}(E) =z_{t+\tau,s+\tau}(E)
			&= (\hat\delta \aI \Xi_{\omega}^{(z)}(y,y))_{t+\tau,s+\tau} (E) - \omega^S_{t+\tau,s+\tau}(E y_{s+\tau})\\
			&= (\hat\delta \aI \Xi_{\theta_{\tau}\omega}^{(z)}(\theta_{\tau}y,\theta_{\tau}y))_{ts} (E) - \theta_{\tau}\omega^S_{ts}(E \theta_{\tau}y_{s}),
			\end{align*}
where we used Remark \ref{remark_supportingprocesses_shift} in the last step.
		\qed \end{proof}
		\begin{remark}\label{remark_localsolution_cut}
			If $(y,z)$ is a fixed point of $\aM_{T,\omega ,\xi}$ than for any $\tilde{T}< T$ the restriction of $(y,z)$ on $[0,\tilde{T}] \times \Delta_{\tilde{T}}$ is a fixed point of $\aM_{\tilde{T},\omega ,\xi}$. 
		\end{remark}
		Now we can state the uniqueness result of the local solution.
		\begin{theorem}
			Let $(y^i,z^i)$, $i=1,2$ be two fixed-points of $\aM_{T,\omega ,\xi}$. Then it must hold that $(y^1,z^1)=(y^2,z^2)$.
		\end{theorem}
		\begin{proof}
			We set $\Tb := \sup \left\{\tilde{T} > 0 \colon (y^1,z^1)\mid_{[0,\tilde{T}]}= (y^2,z^2)\mid_{[0,\tilde{T}]} \right\}$ and assume that $(y^1,z^1)\neq (y^2,z^2)$. Then $\Tb < T$.
		Using the continuity of the solution we have that $y^1_{\Tb}=y^2_{\Tb}$. By Lemma \ref{lemma_localsolution_shift1} we know that $(\theta_{\Tb}y^{1},\theta_{\Tb}z^{1})$ and $(\theta_{\Tb}y^{2},\theta_{\Tb}z^{2})$ are fixed points of $\aM_{T-\Tb,\theta_{\Tb}\omega ,y^1_{\Tb}}$. According to Remark \ref{remark_localsolution_cut} we can choose a small $T^* \in [0,T-\Tb] $, apply \eqref{estimate_deltay_norm} and \eqref{estimate_deltaz_norm}. This leads to
			\begin{align*}
			&\left\|(\theta_{\Tb}y^{1},\theta_{\Tb}z^{1})-(\theta_{\Tb}y^{2},\theta_{\Tb}z^{2}) \right\|_{X,T^*} \\
			&=  \left\|\aM_{T-\Tb,\theta_{\Tb}\omega ,y^1_{\Tb}} ((\theta_{\Tb}y^{1},\theta_{\Tb}z^{1})) 
			-\aM_{T-\Tb,\theta_{\Tb}\omega ,y^1_{\Tb}} ((\theta_{\Tb}y^{2},\theta_{\Tb}z^{2})) \right\|_{X,T^*} \\
		&	\leq C \left(1+\left\| (\theta_{\Tb}y^{1},\theta_{\Tb}z^{1}) \right\|_X^2 +\left\| (\theta_{\Tb}y^{2},\theta_{\Tb}z^{2}) \right\|_X^2 \right)
			\left\| (\theta_{\Tb}y^{1}-\theta_{\Tb}y^{2},\theta_{\Tb}z^{1}-\theta_{\Tb}z^{2}) \right\|_X ({T^*})^\alpha.
			\end{align*}
			If $T^*$ is sufficiently small we see that $(\theta_{\Tb}y^{1 },\theta_{\Tb}z^{1}) = (\theta_{\Tb}y^{2},\theta_{\Tb}z^{2})$ on $[0,T^*]$ which yields 
			$(y^1,z^1)=(y^2,z^2)$ on $[0,\Tb+T^*]$. Therefore, we obviously reached a contradiction with the definition of $\Tb$.
		\qed \end{proof} \\
	
		We conclude this section collecting two important results which immediately follow from the previous deliberations. We first indicate why taking more regular initial data leads to simpler arguments.		\begin{corollary}
			If $\xi\in D_{\beta}$ and $(y,z)$ is the unique fixed-point of $\mathcal{M}_{T,\omega,\xi}$ we have that $y\in C^{\beta}$ and $z\in C^{\alpha+\beta}$.
		\end{corollary}
	
	\begin{proof}
		Since
		\begin{align*}
		(\delta y)_{ts}= (\delta \widetilde{y})_{ts}= (\hat{\delta}\overline{y})_{ts} + (S(t-s)-\mbox{Id}) \overline{y}_{s} + (S(t)-S(s))\xi.
				\end{align*}
Using \eqref{estimate_y_hatdelta} and \eqref{estimate_y_beta} we conclude that
\begin{align*}
\ltn y \rtn_{\beta}\leq C |\xi|_{D_{\beta}} + C (1+ ||(y,z)||^{2}_{X} ) T ^{\alpha-\beta}.
\end{align*}
Recall that
\begin{align*}
\zt_{ts}(E) 
&= \zb_{ts}(E) +a_{ts}(E,S(s)\xi) - \omega^S_{ts}(E S(s) \xi).
\end{align*}
Therefore applying \eqref{estimate_z_norm} and \eqref{estimate_a2} proves the statement.
		\qed
		\end{proof}

\begin{remark}
	Keeping Lemma \ref{lemma:cont:dependence1}
		and  \eqref{convergence_b} in mind one can easily show that the solution continuously depends on the noisy input. 
\end{remark}
	\section{An application}\label{sect:app}
	We finally indicate an example for the abstract theory proven above. For further applications consult \cite[Section 5]{GarridoLuSchmalfuss1} and \cite[Section 7]{GarridoLuSchmalfuss2}. 
	
	\begin{example}
		We consider an open bounded $C^{2}$-domain $\mathcal{O}\in\mathbb{R}^{d},$ for $d\geq 1$. Furthermore, 
		we let $A$ stand for the Laplace operator or for second order uniformly elliptic operator augmented by Dirichlet boundary conditions. Then we know that $A$ generates an analytic $C_{0}$-semigroup on $W:=L^{2}(\mathcal{O})$. Moreover, we can identify the domains of the fractional powers of $A$ with Sobolev-Slobodetski spaces depending on the range of $\theta$. We have according to Theorem 16.12 in \cite{Yagi} that 
		\begin{align*}
		D( (-A)^{\theta})=\begin{cases}
		H^{2\theta}(\mathcal{O}), ~~ 0\leq \theta<1/4\\
		H^{2\theta}_{D}(\mathcal{O}), ~~ 1/4<\theta\leq 1.
		\end{cases}
		\end{align*}
		Here $H_{D}$ stands for the Sobolev space that incorporates the boundary conditions, in particular $D(-A)=H^{2}(\mathcal{O})\cap H^{1}_{0}(\mathcal{O})$.\\
		
		Having stated the assumptions on the linear part we now focus on $G$. Therefore we firstly set for simplicity $V:=L^{2}(\mathcal{O})$. Let $g:\overline{\mathcal{O}}\times \mathbb{R}\to \mathbb{R}$ be a three times continuously differentiable function with bounded derivatives which is zero on $\{0,1\}\times\mathbb{R}$. We interpret $g$ as the kernel of the following integral operator
		\begin{align}\label{example:G}
		G(\phi)(\psi)[x]:=\int\limits_{\mathcal{O}} g (x,\phi(\widetilde{x})) \psi(\widetilde{x})~ d\widetilde{x}.
		\end{align}
		As in \cite[Section XVII.3]{FA} one can show that $G$ is three times continuously Fr\`echet-differentiable and compute the derivatives as follows
		\begin{align*}
		DG(\phi)(\psi,h_{1})[x]&=\int\limits_{\mathcal{O}} D_{2}g (x,\phi(\widetilde{x}))\psi(\widetilde{x}) h_{1}(\widetilde{x}) ~d\widetilde{x},\\
		D^{2}G(\phi)(\psi,h_{1},h_{2})[x]&=\int\limits_{\mathcal{O}} D^{2}_{2}g (x,\phi(\widetilde{x}))\psi(\widetilde{x}) h_{1}(\widetilde{x})h_{2}(\widetilde{x}) ~d\widetilde{x},\\
			D^{3}G(\phi)(\psi,h_{1},h_{2},h_{3})[x]&=\int\limits_{\mathcal{O}} D^{3}_{2}g (x,\phi(\widetilde{x}))\psi(\widetilde{x}) h_{1}(\widetilde{x})h_{2}(\widetilde{x}) h_{3}(\widetilde{x}) ~d\widetilde{x},
		\end{align*}
		for $h_{1}$, $h_{2}$ and $h_{3}$ belonging to $W$. Due to the assumptions on $g$, these expressions are obviously bounded.\\
		
		It is left to show that 
		$G:W\to \mathcal{L}(W,D_{\beta})$ is Lipschitz continuous. Here $\beta\geq 1/3$ as assumed in (G). To this aim let $\psi^{1}$ and $\psi^{2}\in W $ and compute
		\begin{align*}
		 |G(\phi^{1}) - G (\phi^{2})|_{\mathcal{L}(W,D_{\beta})} &=\sup\limits_{|\psi|=1} | G(\phi^{1})(\psi) - G (\phi^{2})(\psi) |_{D_{\beta}}\\
		 &\leq C \sup\limits_{|\psi|=1} | G(\phi^{1})(\psi) - G (\phi^{2})(\psi) |_{D(-A)}\\
		 & = C \sup\limits_{|\psi|=1}\Bigg| \int\limits_{\mathcal{O}} \left(g (\cdot, \phi^{1}(\widetilde{x}))- g (\cdot, \phi^{2}(\widetilde{x})\right)\psi(\widetilde{x}) ~d\widetilde{x}    \Bigg|_{D(-A)}.
			\end{align*}
			
	Therefore we estimate for $k=0,1,2$:
	\begin{align*}
&	\Bigg|\int\limits_{\mathcal{O}} \left(D^{k}_{1}g (\cdot, \phi^{1}(\widetilde{x}))-D^{k}_{1} g (\cdot, \phi^{2}(\widetilde{x})\right)\psi(\widetilde{x}) ~d\widetilde{x}    \Bigg|^{2}_{L^{2}(\mathcal{O})}\\
&= \int\limits_{\mathcal{O}} \Bigg|\int\limits_{\mathcal{O}} \left(D^{k}_{1}g (x, \phi^{1}(\widetilde{x}))-D^{k}_{1} g (x, \phi^{2}(\widetilde{x})\right)\psi(\widetilde{x}) ~d\widetilde{x}    \Bigg|^{2} dx\\
& \leq C \int\limits_{\mathcal{O}} \Bigg| \int\limits_{\mathcal{O}} |D_{2}D^{k}_{1}g |_{\infty}~ |\phi^{1}(\widetilde{x}) -\phi^{2}(\widetilde{x})|~ |\psi(\widetilde{x})|
 ~d\widetilde{x} \Bigg|^{2} dx\\
 & \leq C|\mathcal{O}|~ |D_{2}D^{k}_{1} g|_{\infty}^{2}~ |\phi^{1}-\phi^{2}|^{2}_{L^{2}(\mathcal{O})}~ |\psi|^{2} _{L^{2}(\mathcal{O})} .
	\end{align*}
	We finally obtain that
	\begin{align*}
	|G(\phi^{1}) - G (\phi^{2})|_{\mathcal{L}(W,D_{\beta})}\leq C |\phi^{1}-\phi^{2}|_{L^{2}(\mathcal{O})}, 
	\end{align*}
	where the constant $C$ depends on $|\mathcal{O}|$ and $g$.
\end{example}
	
In conclusion our theory can be applied to parabolic SPDEs driven by multiplicative fractional noise as described in (\ref{example:G}).	
		\begin{remark}
			Note that we do not make any additional assumptions on the eigenvalues of $A$. This is natural in the context of rough path theory, compare \cite{DeyaGubinelliTindel}. However, working with different techniques such as presenting the infinite-dimensional integral as a sum of one-dimensional integrals \cite{MaslowskiNualart}, (\cite{GarridoLuSchmalfuss2}) may lead to further assumptions on the asymptotic of the eigenvalues and implicitly to a restriction of the domains.
		\end{remark}
		\vspace*{5 mm}
		
		\textbf{\Large{Acknowledgments}}: The authors are grateful to M. J. Garrido-Atienza and B. Schmalfu\ss{} for helpful comments. AN acknowledges support by a DFG grant in the
		D-A-CH framework (KU 3333/2-1).
		\appendix
		\section{Preliminary results for the Sewing Lemma}
		The following facts will be employed in Section \ref{sectsl}.
		
		\begin{lemma}\label{lemma_hatdelta_uniqueness}
			Let $\rho>1$ and $\kappa \in C(\left[0,T\right],W)$ such that
			\begin{align*}
			\kappa_0 = 0 \quad \quad \text{ and} \quad \quad 
			\left|(\hat\delta \kappa)_{ts} \right| \leq c \left(t-s \right)^{\rho}.
			\end{align*}
			Then $\kappa \equiv 0$.
		\end{lemma}
		\begin{proof}
			For any partition $\aP$ of $[0,t]$ we have
			\begin{align*}
			\kappa_t = (\hat\delta \kappa)_{t0} = \sum\limits_{[u,v]\in\aP} S(t-v)(\hat\delta \kappa)_{vu}.
			\end{align*}
			Hence we derive
			\begin{align*}
			\left|\kappa_t \right|\leq c_S c \left| \aP\right|^{\rho-1} t, ~~\mbox{ which tends to 0 }~\text{as } \left|\aP\right| \to 0.
			\end{align*}
			Consequently $\kappa \equiv 0$.
		\qed \\ \end{proof}
		\begin{lemma}\label{lemma_uniqueness}
			Let $0 < \alpha,~\beta<1$,  $\rho>1$ and $\kappa \in C(\left[0,T\right],W)$ such that
			\begin{align*}
			&\kappa_0 = 0, \\  
			&\left|(\hat\delta \kappa)_{ts} \right| \leq C \left(t-s \right)^{\alpha}, \\
			&\left|(\hat\delta \kappa)_{ts} \right| \leq C s^{-\beta} \left(t-s \right)^{\rho}, ~~\mbox{for } s\neq 0.
			\end{align*}
			 Then $\kappa \equiv 0$.
		\end{lemma}
		\begin{proof}
			Let $t \in \left[0,T\right]$ be arbitrary but fixed. 
			Consider $\aP_n$ a dyadic partition of $[0,t]$.
			We have
			\begin{align*}
			\kappa_t = (\hat\delta \kappa)_{t0} = \sum\limits_{[u,v]\in\aP_n} S(t-v)(\hat\delta \kappa)_{vu}.
			\end{align*}
			Consequently,
			\begin{align*}
			\left|\kappa_t \right| 
			&\leq C \sum\limits_{[u,v]\in\aP_n} \left|(\hat\delta \kappa)_{vu}\right| \\
		&	\leq C \frac{t^\alpha}{2^{n\alpha}} + C \sum\limits_{\substack{[u,v]\in\aP_n \\ u \neq 0}} u^{-\beta} \frac{t^\rho}{2^{n\rho}} \\
			&= C \frac{t^\alpha}{2^{n\alpha}} + C \frac{t^{\rho-1}}{2^{n-1}} \sum\limits_{\substack{[u,v]\in\aP_n \\ u \neq 0}} u^{-\beta} \frac{t}{2^{n}} \\
			&\leq  C \frac{t^\alpha}{2^{n\alpha}} + C \frac{t^{\rho-1}}{2^{n-1}} \int\limits_{0}^{t} q^{-\beta} dq\\
			&\leq C \frac{t^\alpha}{2^{n\alpha}} + C \frac{t^{\rho-\beta}}{2^{n-1}}.
			\end{align*}
			Since, this statement has to be valid for all $n \in \IN$ we must have $\kappa\equiv 0$.
		\qed \\ \end{proof}
		

		\begin{lemma} \label{lemma_appendix1}
			Given $0<\gamma,~\eps <1$ then there is a constant $C=C(\gamma,\eps)$ such that for all $n \in \IN$ the estimate
			\begin{align} \label{estimate_helpsum1}
			\sum\limits_{k=1}^{n-1} {k^{-\gamma} (n-k)^{-\eps}} \leq C \sum\limits_{k=0}^{n-1} {(k+1)^{-\gamma}(n-k)^{-\eps}}
			\end{align}
			holds true.
		\end{lemma}
		\begin{proof}
			For the right sum we see
			\begin{align*}
			\sum\limits_{k=0}^{n-1} {(k+1)^{-\gamma}(n-k)^{-\eps}} = n^{-\eps} + \sum\limits_{k=1}^{n-1} {(k+1)^{-\gamma}(n-k)^{-\eps}}
			\end{align*}
			Hence, \eqref{estimate_helpsum1} is fulfilled if
			\begin{align*}
			\sum\limits_{k=1}^{n-1} {\left[k^{-\gamma} - (k+1)^{-\gamma} \right](n-k)^{-\eps}} \leq C n^{-\eps}
			\end{align*}
			Furthermore we estimate
			\begin{align*}
			&\sum\limits_{k=1}^{n-1} {\left[k^{-\gamma} - (k+1)^{-\gamma} \right](n-k)^{-\eps}} 
			= \sum\limits_{k=1}^{n-1} {\int\limits_{k}^{k+1}{\gamma x^{-\gamma-1}} dx \:(n-k)^{-\eps}} \\
			\leq &\sum\limits_{k=1}^{n-1} {\int\limits_{k}^{k+1}{\gamma x^{-\gamma-1} (n-x)^{-\eps}}dx} 
			= \int\limits_{1}^{n}{\gamma x^{-\gamma-1} (n-x)^{-\eps}}dx.
			\end{align*}
			Therefore it is sufficient to show that for all $u \geq 1$ it holds
			\begin{align*}
			g(u,\gamma,\eps):= \int\limits_{1}^{u}{\gamma x^{-\gamma-1} (u-x)^{-\eps}}dx \leq C u^{-\eps}.
			\end{align*}
			To prove this we define 
			\begin{align*}
			h(u,\gamma,\eps):=&u^{\gamma+\eps}\, g(u,\gamma,\eps) \\
			=&u^{\gamma+\eps} \int\limits_{1}^{u}{\gamma x^{-\gamma-1} (u-x)^{-\eps}}dx \\
			= &\int\limits_{1}^{u} {\gamma \left(\frac{u}{x}\right)^{\gamma-1} \left(1-\frac{x}{u} \right)^{-\eps}} \frac{u}{x^2} dx \\
			= &- \int\limits_{1}^{u} {\gamma \left(\frac{u}{x}\right)^{\gamma-1} \left(1-\frac{x}{u} \right)^{-\eps}} d\left(\frac{u}{x} \right) \\
			= &\int\limits_{1}^{u} {\gamma x^{\gamma-1} \left(1-\frac{1}{x}\right)^{-\eps}} dx.
			\end{align*}
			If $\gamma+\eps \leq 1$ consider
			\begin{align*}
			h(u,\gamma,1-\gamma) = \int\limits_{1}^{u} {\gamma \left(x-1\right)^{\gamma-1}} dx 
			= (u-1)^{\gamma} \leq u^\gamma.
			\end{align*}
			Since $0< 1-\frac{1}{x}<1 $ for all $x\geq 1$ we see that $h$ is monotonously increasing in $\eps$.
			Hence we obtain for $0 \leq \eps \leq 1-\gamma$
			\begin{align*}
			&h(u,\gamma,\eps) \leq h(u,\gamma,1-\gamma) \leq u^\gamma, 
			\end{align*}
			consequently
			\begin{align*}
			 &g(u,\gamma,\eps) \leq u^{-\eps}.
			\end{align*}
			If $\gamma+\eps > 1$ we estimate
			\begin{align*}
			h(u,\gamma,\eps) 
			&= \int\limits_{1}^{u} {\gamma x^{\gamma-1} \left(1-\frac{1}{x}\right)^{-\eps}} dx \\
			&= \int\limits_{1}^{u} {\gamma x^{\gamma+\eps-1} \left(x-1\right)^{-\eps}} dx \\
			&\leq \gamma u^{\gamma+\eps-1} \int\limits_{1}^{u} {\left(x-1\right)^{-\eps}} dx \\
			&= \frac{\gamma}{1-\eps} u^{\gamma+\eps-1} \left(u-1 \right)^{1-\eps} 
			\leq \frac{\gamma}{1-\eps} u^\gamma,
			\end{align*}
			which directly yields
			\begin{align*}
			g(u,\gamma,\eps) \leq \frac{\gamma}{1-\eps} u^{-\eps}.
			\end{align*}
		\qed \\ \end{proof}
\section{Fundamental estimates for the fixed-point argument}
In the next deliberations we illustrate a technique which is required in Section \ref{sectfp}. This is based on the division property for smooth functions, see p. 109 in \cite{FritzHairer}. This is also used for rough SDEs, in order to estimate the difference of the norm of two controlled rough paths, consult \cite[Chapter 8]{FritzHairer}, especially the proof of Theorem 8.4 in \cite{FritzHairer}.\\

In the following $c_{G}$ stands for a universal constant which exclusively depends on $G$ and its derivatives. The next result can immediately be obtained applying the mean value theorem.
		\begin{lemma}\label{lemma_G1}
			Let $\hat{W}$ be a separable Banach space, $G \in C_b^2(W,\hat{W})$ and $x_1,x_2,x_3,x_4 \in W$. The following estimate
			\begin{align}
			\begin{split}
			&\left|G(x_2)-G(x_1) - G(x_4)+G(x_3)\right| \\
			&\leq c_G \left|x_2-x_1-x_4+x_3\right| + c_G\left|x_4-x_3\right| \left(\left|x_3-x_1\right| + \left|x_4-x_2\right| \right)
			\end{split}
			\end{align}
			holds true.  
		\end{lemma}
				Keeping this in mind we derive the following result. 
		\begin{corollary}\label{corollary_G1}
			Let $y^i \in C^{\beta,\beta}([0,T];W)$ for $i=1,2$ and $G \in C_b^2(W,\hat{W})$. Then, for all $0< s < t \leq T$ we have 
			\begin{align}
			\begin{split}
			&\left|G(y^1_t)-G(y^1_s) - G(y^2_t)+G(y^2_s) \right| \\
			&\leq c_G\left( \ltn y^1-y^2\rtn_{\beta,\beta} + \ltn y^2\rtn_{\beta,\beta} \left\|y^1-y^2 \right\|_{\infty} \right) 
			s^{-\beta} \left(t-s \right)^{\beta},
		\end{split}	\label{estimate_G1beta}
					\end{align}
			as well as
			\begin{align}		
			\begin{split}
			&\left|G(y^1_t)-G(y^1_s) - G(y^2_t)+G(y^2_s) \right|\\
			&\leq  c_G \left( \left\| y^1-y^2\right\|_{\infty} + \left\| y^2\right\|_{\infty} \left\|y^1-y^2 \right\|_{\infty} \right).
			\end{split}
			\label{estimate_G1infty}
			\end{align} 
		\end{corollary}
		\begin{proof}
			Applying Lemma \ref{lemma_G1}, see also Lemma~7.1~in~\cite{NualartRascanu}, with $x_1=y^1_s$, $x_2=y^1_t$, $x_3=y^2_s$ and $x_4=y^2_t$ we infer
			\begin{align*}
			&\left|G(y^1_t)-G(y^1_s) - G(y^2_t)+G(y^2_s) \right|\\ 
			&\leq  c_G \left|y^1_t-y^1_s-y^2_t+y^2_s \right| 
			+c_G\left|y^2_t-y^2_s\right| \left(\left|y^1_s-y^2_s\right| + \left|y^1_t-y^2_t\right| \right)
			\end{align*}
			We have two possibilities to estimate this expression. On the one hand we have
			\begin{align*}
			&\left|G(y^1_t)-G(y^1_s) - G(y^2_t)+G(y^2_s) \right|\\ 
			&\leq c_G \left|y^1_t-y^1_s-y^2_t+y^2_s \right| 
			+c_G\left|y^2_t-y^2_s\right| \left(\left|y^1_s-y^2_s\right| + \left|y^1_t-y^2_t\right| \right) \\
			&\leq c_G\left( \ltn y^1-y^2\rtn_{\beta,\beta} + \ltn y^2\rtn_{\beta,\beta} \left\|y^1-y^2 \right\|_{\infty} \right) 
			s^{-\beta} \left(t-s \right)^{\beta}.
			\end{align*}
			On the other hand we obtain
			\begin{align*}
			&\left|G(y^1_t)-G(y^1_s) - G(y^2_t)+G(y^2_s) \right|\\ 
			& \leq c_G \left|y^1_t-y^1_s-y^2_t+y^2_s \right| 
			+c_G\left|y^2_t-y^2_s\right| \left(\left|y^1_s-y^2_s\right| + \left|y^1_t-y^2_t\right| \right) \\
			&\leq c_G \left( \left\| y^1-y^2\right\|_{\infty} + \left\| y^2\right\|_{\infty} \left\|y^1-y^2 \right\|_{\infty} \right).
			\end{align*}
			This proves the statement.
		\qed  \end{proof}
		

		\begin{lemma}\label{lemma_G2}
			Let $G \in C_b^3(W,\hat{W})$ and $x_1,x_2,x_3,x_4 \in W$. Then 
			\begin{align}
			\begin{split}
			&\left|G(x_2)-G(x_1)-DG(x_1)(x_2-x_1) - G(x_4)+G(x_3)+DG(x_3)(x_4-x_3) \right| \\
			&\leq c_G \left(\left|x_2-x_1\right|+\left|x_4-x_3\right| \right) \left|x_2-x_1-x_4+x_3\right| \\
			&+ c_{G}\left|x_4-x_3\right|^2 \left(\left|x_3-x_1\right| + \left|x_4-x_2\right| \right).
			\end{split}
			\end{align}  
		\end{lemma}
	For a complete proof, see p.~2716 in~\cite{HuNualart}.\\[1ex]
		This helps us further obtain an essential estimate for our fixed-point argument.
		\begin{corollary}\label{corollary_G2}
			Given $y^i \in C^{\beta,\beta}([0,T];W)$ for $i=1,2$ and $G \in C_b^3(W,\hat{W})$. Then the following estimates are valid for all $0< s < t \leq T$:
			\begin{align}
			\begin{split}
			&\left|G(y^1_s)-G(y^1_t)+DG(y_s^1)(y_t^1-y_s^1) - \left(G(y^2_s)-G(y^2_t)+DG(y_s^2)(y_t^2-y_s^2)\right) \right|\\
			&\leq 
			 c_G \left[\left(\ltn y^1 \rtn_{\beta,\beta} + \ltn y^2\rtn_{\beta,\beta}\right) \ltn y^1-y^2\rtn_{\beta,\beta} 
			+ \ltn y^2\rtn_{\beta,\beta}^2 \left\|y^1-y^2 \right\|_{\infty} \right] s^{-2\beta} \left(t-s \right)^{2\beta},
			\end{split}\label{estimate_G2beta} 
			\end{align}
			as well as
			\begin{align}
			\begin{split}
			&\left|G(y^1_s)-G(y^1_t)+DG(y_s^1)(y_t^1-y_s^1) - \left(G(y^2_s)-G(y^2_t)+DG(y_s^2)(y_t^2-y_s^2)\right) \right|\\
			&\leq 
			 c_G \left(\ltn y^1\rtn_{\beta,\beta} + \ltn y^2\rtn_{\beta,\beta} + \ltn y^2\rtn_{\beta,\beta} \left\| y^2\right\|_{\infty} \right) \left\|y^1-y^2 \right\|_{\infty} s^{-\beta} \left(t-s \right)^{\beta}.
			\end{split}\label{estimate_G2infty}
			\end{align} 
		\end{corollary}
		\begin{proof}
			As previously argued, we apply Lemma \ref{lemma_G2} with $x_1=y_s^1$, $x_2=y_t^1$, $x_3=y_s^2$ and $x_4=y_t^2$. This results in
			\begin{align*}
			&\left|G(y^1_s)-G(y^1_t)+DG(y_s^1)(y_t^1-y_s^1) - \left(G(y^2_s)-G(y^2_t)+DG(y_s^2)(y_t^2-y_s^2)\right) \right|\\
			&\leq c_G  \left(\left|y_t^1-y_s^1 \right|+\left|y_t^2-y_s^2 \right|\right) \left|y_t^1-y_s^1-y_t^2+y_s^2 \right| 
			+ c_G\left|y_t^2-y_s^2 \right|^2 \left( \left|y_s^1-y_s^2\right| + \left|y_t^1-y_t^2 \right| \right). 
			\end{align*}
			Again, we use have two possibilities to obtain the following inequalities.
			First of all we infer that
			\begin{align*}
			&\left|G(y^1_s)-G(y^1_t)+DG(y_s^1)(y_t^1-y_s^1) - \left(G(y^2_s)-G(y^2_t)+DG(y_s^2)(y_t^2-y_s^2)\right) \right|\\
			&\leq c_G  \left(\left|y_t^1-y_s^1 \right|+\left|y_t^2-y_s^2 \right|\right) \left|y_t^1-y_s^1-y_t^2+y_s^2 \right| 
			+ c_G\left|y_t^2-y_s^2 \right|^2 \left( \left|y_s^1-y_s^2\right| + \left|y_t^1-y_t^2 \right| \right) \\
			&\leq c_G \left(\ltn y^1 \rtn_{\beta,\beta} + \ltn y^2\rtn_{\beta,\beta} \right) \ltn y^1-y^2 \rtn_{\beta,\beta} s^{-2\beta} \left(t-s \right)^{2\beta} 
			+c_G \ltn y^2 \rtn_{\beta,\beta}^2 \left\|y^1-y^2\right\|_{\infty} s^{-2\beta} \left(t-s \right)^{2\beta} \\
			&\leq c_G \left[\left(\ltn y^1 \rtn_{\beta,\beta} + \ltn y^2\rtn_{\beta,\beta}\right) \ltn y^1-y^2\rtn_{\beta,\beta} 
			+ \ltn y^2\rtn_{\beta,\beta}^2 \left\|y^1-y^2 \right\|_{\infty} \right] s^{-2\beta} \left(t-s \right)^{2\beta}.
			\end{align*} 
			On the other hand we finally get
			\begin{align*}
			&\left|G(y^1_s)-G(y^1_t)+DG(y_s^1)(y_t^1-y_s^1) - \left(G(y^2_s)-G(y^2_t)+DG(y_s^2)(y_t^2-y_s^2)\right) \right|\\
			& \leq c_G  \left(\left|y_t^1-y_s^1 \right|+\left|y_t^2-y_s^2 \right|\right) \left|y_t^1-y_s^1-y_t^2+y_s^2 \right| 
			+ c_G\left|y_t^2-y_s^2 \right|^2 \left( \left|y_s^1-y_s^2\right| + \left|y_t^1-y_t^2 \right| \right)\\
			&\leq c_G \left(\ltn y^1\rtn_{\beta,\beta} + \ltn y^2\rtn_{\beta,\beta} \right) \left\|y^1-y^2 \right\|_{\infty} s^{-\beta} \left(t-s \right)^{\beta}
			+ c_G \ltn y^2\rtn_{\beta,\beta} \left\| y^2\right\|_{\infty} \left\|y^1-y^2 \right\|_{\infty} s^{-\beta} \left(t-s \right)^{\beta} \\
			&\leq c_G \left(\ltn y^1\rtn_{\beta,\beta} + \ltn y^2\rtn_{\beta,\beta} + \ltn y^2\rtn_{\beta,\beta} \left\| y^2\right\|_{\infty} \right) \left\|y^1-y^2 \right\|_{\infty} s^{-\beta} \left(t-s \right)^{\beta}.
			\end{align*}
		\qed \\ \end{proof}

\end{document}